\documentclass[11pt, pdflatex, a4paper]{article}
\usepackage[utf8]{inputenc}
\usepackage{amssymb, graphicx, amsmath, amsthm, mathtools, tikz-cd, enumitem, nccmath}
\usepackage[mathscr]{eucal} 
\usepackage{comment} 
\usepackage[affil-it]{authblk} 
\usepackage{mathrsfs}
\usepackage{makeidx}
\usepackage{setspace}
\usepackage{titlesec}

\usepackage{titletoc}
\usepackage{lipsum} 

\let\OLDthebibliography\thebibliography
\renewcommand\thebibliography[1]{
  \OLDthebibliography{#1}
  \setlength{\parskip}{0pt}
  \setlength{\itemsep}{1pt plus 0.3ex}
}

\usepackage{geometry}
\usepackage{hyperref}
\usepackage{wrapfig}

\usepackage[font=footnotesize]{caption} 

\usepackage{tocloft}
\usepackage{fancyhdr}

\theoremstyle{plain}
	\newtheorem{thm}{Theorem}[section]
	\newtheorem{prop}[thm]{Proposition}
	
	\newtheorem{lemma}[thm]{Lemma}
	\newtheorem{prob}[thm]{Problem}
	\newtheorem{probb}{Problem}

\theoremstyle{definition}
	\newtheorem{defn}[thm]{Definition}
	\newtheorem{rem}[thm]{Remark}

\counterwithout{figure}{section}

\usepackage{chngcntr}

\counterwithin*{equation}{section}

\def\a{\mathfrak{a}}
\def\Ad{\text{\normalfont{Ad}}}

\def\ad{\text{\normalfont{ad}}}

\def\C{{\mathbb C}}

\def\Geg{\widetilde{C}}

\def\Diff{\text{\normalfont{Diff}}}
\def\spanned{\text{\normalfont{span}}}
\def\Weyl{\mathcal{D}}
\def\D{\mathbb{D}}

\def\End{\text{\normalfont{End}}}

\def\g{\mathfrak{g}}

\def\h{\mathfrak{h}}

\def\Harm{\mathcal{H}}
\def\Hom{\text{\normalfont{Hom}}}

\def\id{\text{\normalfont{id}}}
\def\Ind{\text{\normalfont{Ind}}}

\def\Im{\text{\normalfont{Im}}}

\def\l{\mathfrak{l}}
\def\L{\mathcal{L}}

\def\m{\mathfrak{m}}

\def\N{{\mathbb N}}
\def\n{\mathfrak{n}}

\def\p{\mathfrak{p}}

\def\Pol{\text{\normalfont{Pol}}}

\def\R{{\mathbb R}}
\def\Re{\text{\normalfont{Re}}}
\def\Rest{\text{\normalfont{Rest}}}

\def\so{\mathfrak{so}}
\def\Sol{\text{\normalfont{Sol}}}
\def\stab{\text{\normalfont{stab}}}
\def\supp{\text{\normalfont{supp }}}
\def\Symb{\text{\normalfont{Symb}}}

\def\tr{\text{\normalfont{tr }}}

\def\V{\mathcal{V}}

\def\W{\mathcal{W}}

\def\Z{{\mathbb Z}}

\newcommand\arrowsimeq{\stackrel{\mathclap{\thicksim}}{\longrightarrow}}

\newcommand\arrowiota{\stackrel{\mathclap{\iota}}{\hookrightarrow}}

\AtBeginEnvironment{pmatrix}{\setlength{\arraycolsep}{2pt}}
\renewcommand*{\arraystretch}{0.6} 
\setlist[enumerate]{topsep=0pt, itemsep=-4pt} 
\setlist[itemize]{topsep=0pt, itemsep=-4pt}

\begin{document}
\title{Conformally covariant differential symmetry breaking operators for a vector bundle of rank 3 on $S^3$}
\author{V\'{i}ctor P\'{E}REZ-VALD\'{E}S
\thanks{Graduate School of Mathematical Sciences, The University of Tokyo\\Email: perez@ms.u-tokyo.ac.jp}}
\date{}
\maketitle
\vspace{-0.5cm}
\begin{abstract}
We construct and give a complete classification of all the differential symmetry breaking operators $\D_{\lambda, \nu}^m: C^\infty(S^3, \mathcal{V}_\lambda^3) \rightarrow C^\infty(S^2, \mathcal{L}_{m, \nu})$, between the spaces of smooth sections of a vector bundle of rank $3$ over the $3$-sphere $\mathcal{V}_\lambda^3  \rightarrow S^3$, and a line bundle over the $2$-sphere $\mathcal{L}_{m, \nu} \rightarrow S^2$. In particular, we give necessary and sufficient conditions on the tuple of parameters $\left(\lambda, \nu, m\right)$ for which these operators exist.
\end{abstract}

\renewcommand{\contentsname}{Table of contents}
\tableofcontents

\section{Introduction}
Given a representation $\Pi$ of a Lie group $G$, and Lie subgroup $G^\prime \subset G$, we can consider the restriction $\Pi\rvert_{G^\prime}$ of $\Pi$ to $G^\prime$ and think about its nature: how does it behave? Or, more concretely, how is $\Pi\rvert_{G^\prime}$ decomposed (in a broad sense) into irreducible representations of $G^\prime$? The problem of understanding the behaviour of the restriction, in particular, its irreducible decomposition (supposed that $\Pi$ is irreducible), is what in representation theory is called a \textit{branching problem}.

T. Kobayashi proposed a program (\textit{the ABC Program} \cite{kob2015}), divided in three stages, to address these branching problems for real reductive Lie groups from different points of view:

\begin{enumerate}[itemsep=1pt]
\item[$\bullet$] \textbf{Stage A}: Abstract features of the restriction.
\item[$\bullet$] \textbf{Stage B}: Branching laws.
\item[$\bullet$] \textbf{Stage C}: Construction of symmetry breaking operators.
\end{enumerate}

The goal of Stage A is to analyze abstract aspects of the restriction $\Pi\rvert_{G^\prime}$, such as estimates on the multiplicities of irreducible representations of $G^\prime$ occurring in $\Pi\rvert_{G^\prime}$, or the study of the spectrum of $\Pi\vert_{G^\prime}$. 
Note that while the multiplicities may be infinite even when $G^\prime$ is a maximal subgroup of G, there are some cases when they are finite or even uniformly bounded (\cite{kob-osh}). There are also cases when the multiplicities are at most one, as is the case of visible actions (\cite{kob2013b}), or the Gross-Prasad conjecture (\cite{gross-prasad}, \cite{sun-zhu}), among others.

Stage B addresses the irreducible decomposition of the restriction. If $\Pi$ is a finite-dimensional representation such that the restriction $\Pi\rvert_{G^\prime}$ is completely reducible, it is clear what its irreducible decomposition means (as is the case of the Clebsch-Gordan formula for the tensor product of finite-dimensional irreducible representations of $SL(2,\R)$). When $\Pi$ is a unitary representation (not necessarily finite-dimensional), we can think about its irreducible decomposition by using the direct integral of Hilbert spaces (Mautner--Teleman theorem \cite{mautner}, \cite{teleman}). In a more general situation when $\Pi$ is not necessarily unitary, Stage B may be thought of as the study of the space $\Hom_{G^\prime}(\Pi\vert_{G^\prime}, \pi)$ (or $\Hom_{G^\prime}(\pi, \Pi\vert_{G^\prime})$), where $\pi$ is a fixed representation of $G^\prime$.

The $G^\prime$-intertwining operators in $\Hom_{G^\prime}(\Pi\vert_{G^\prime}, \pi)$ are called \textbf{symmetry breaking operators}. Stage C asks for concrete symmetry breaking operators in order to understand better the nature of the space $\Hom_{G^\prime}(\Pi\vert_{G^\prime}, \pi)$. 

In contrast with the abstract flavour of Stages A and B, Stage C depends on concrete realizations of the representations, so it often interacts with geometric or analytic problems.

In this paper, we focus on Stage C. Concretely, we focus on the construction of symmetry breaking operators that can be written as \emph{differential} operators for a concrete pair of reductive Lie groups $(G, G^\prime)$. One classical example following this line is given by the Rankin--Cohen bidifferential operators (\cite{cohen}, \cite{rankin}), which are a special case of symmetry breaking operators for the tensor product of two holomorphic discrete series representations of $SL(2,\R)$. 

In the following, we give a specific geometric setting that we consider throughout the paper. This setting allows us to consider Stage C in a significantly wide class of cases. 

Given two smooth manifolds $X$ and $Y$, two vector bundles over them $\V \rightarrow X$ and $\W \rightarrow Y$ and a smooth map $p: Y \rightarrow X$, we can talk about the notion of a \textit{differential operator} between the spaces of smooth sections $T:C^\infty(X, \V) \rightarrow C^\infty(Y, \W)$, following the line of the famous result of J. Peetre (\cite{peetre}, \cite{peetre2}) about the characterization of differential operators in terms of their support (see Definition \ref{def-diffop}). In addition, suppose that we have a pair of Lie groups $G^\prime \subset G$ acting equivariantly on $\W \rightarrow Y$ and $\V \rightarrow X$ respectively, and that $p$ is a $G^\prime$-equivariant map between $Y$ and $X$. In this setting, we can consider the following problem:

\begin{prob}\label{prob-first} Give a description of the space of $G^\prime$-intertwining differential operators (\textbf{differential symmetry breaking operators})	
$$D: C^\infty(X,\V) \rightarrow C^\infty(Y, \W).$$
\end{prob}

Even if this is a geometric setting for Problem \ref{prob-first}, it is very general and differential symmetry breaking operators may or may not exist. For instance, if $\W$ is  isomorphic to the pullback $p^*\V$, the restriction map $f \mapsto f\lvert_Y$ is clearly a $G^\prime$-intertwining differential operator from $C^\infty(X, \V)$ to $C^\infty(Y, \W)$. 

In the general setting where there is no morphism between $p^*\V$ and $\W$, there are cases where non-zero differential symmetry breaking operators exist, but determining when this occurs is a considerably hard task, and there are many unsolved problems. As for the ones that have been solved during the past years, for example, T. Kobayashi and B. Speh constructed and classified not only the ones that can be written as differential operators, but all the symmetry breaking operators for the tuple $(X, Y, G, G^\prime) = (S^n, S^{n-1}, O(n+1,1), O(n,1))$ and for a concrete pair of vector bundles $(\V, \W)$ (\cite{kob-speh1}, \cite{kob-speh2}).

In the case where $X = G/P \supset Y = G^\prime/P^\prime$ are flag varieties and the vector bundles $\V, \W$ are those associated to representations of the parabolic subgroups $P$ and $P^\prime$ respectively, T. Kobayashi proposed a method (\textit{the F-method} \cite{kob2013}) to construct and classify all differential symmetry breaking operators
\begin{equation*}
D: C^\infty(G/P, \V) \rightarrow C^\infty(G^\prime/P^\prime, \W).
\end{equation*}

This F-method allows us to reduce the problem of constructing differential symmetry breaking operators to the problem of finding polynomial solutions of a system of PDE's by applying the \lq\lq algebraic Fourier transform\rq\rq{} to certain generalized Verma modules, and
to use invariant theory to solve the latter. It is known (\cite[Prop. 3.10]{kob-pev1}) that if the nilradical of the parabolic subgroup $P$ is abelian, then the principal term of the equations of the system of PDE's is of order two.

In this setting, T. Kobayashi and M. Pevzner used the F-method to construct differential symmetry breaking operators for some concrete pairs of hermitian symmetric spaces $X \supset Y$ (\cite{kob-pev2}). In a slightly different context, T. Kobayashi together with B. \O rsted, V. Souček and P. Somberg used the F-method to solve a similar problem in the setting of conformal geometry. In addition, in \cite{kkp} all conformal symmetry breaking operators for differential forms on spheres were constructed and classified by T. Kobayashi, T. Kubo and M. Pevzner.

In the present paper, we consider the case $(X,Y, G, G^\prime) = (S^3, S^2, SO_0(4,1), SO_0(3,1))$ and solve Problem \ref{prob-first} for a concrete pair of vector bundles $(\V, \W) = (\V_\lambda^3,\L_{m, \nu})$ (see (\ref{vectorbundle-V}) and (\ref{vectorbundle-L}) for the definition). In particular, we provide complete solutions of the following two problems (see Theorems \ref{mainthm1} and \ref{mainthm2} respectively):

\begin{probb} Give necessary and sufficient conditions on the parameters $\lambda, \nu \in \C$, $m \in \Z\setminus\{0\}$
such that the space
\begin{equation}\label{DSBO-space}
\Diff_{SO_0(3,1)}\left(C^\infty(S^3, \V_\lambda^3), C^\infty(S^2, \L_{m,\nu})\right)
\end{equation}
of differential symmetry breaking operators
$\D_{\lambda, \nu}^m: C^\infty(S^3, \V_\lambda^3)\rightarrow C^\infty(S^2, \L_{m,\nu})$
is non-zero. In particular, determine
\begin{equation*}
\dim_\C \Diff_{SO_0(3,1)}\left(C^\infty(S^3, \V_\lambda^3), C^\infty(S^2, \L_{m,\nu})\right).
\end{equation*}
\end{probb}

\begin{probb} Construct explicitly the generators
 \begin{equation*}
\D_{\lambda, \nu}^m \in \Diff_{SO_0(3,1)}\left(C^\infty(S^3, \V_\lambda^3), C^\infty(S^2, \L_{m,\nu})\right).
\end{equation*}
\end{probb}

Throughout the paper we use the following notation: $\N:= \{0, 1, 2, \ldots\}$, $\N_+:= \{1, 2, 3, \ldots\}$.
\subsection{Main results}\label{section-mainthms}
In this subsection, we state our main results, which provide a complete solution to Problems A and B. We start with the solution to Problem A.

\begin{thm}\label{mainthm1} Let $\lambda, \nu \in \C$ and $m \in \Z\setminus\{0\}$. Then, the following three conditions on the triple $(\lambda, \nu, m)$ are equivalent:
\begin{enumerate}[label=\normalfont{(\roman*)}]
\item $\Diff_{SO_0(3,1)}\left(C^\infty(S^3, \V_\lambda^3), C^\infty(S^2, \L_{m,\nu})\right) \neq \{0\}$.
\item $\dim_\C \Diff_{SO_0(3,1)}\left(C^\infty(S^3, \V_\lambda^3), C^\infty(S^2, \L_{m,\nu})\right) = 1$.
\item The triple $(\lambda, \nu, m)$ belongs to one of the following cases.
	\begin{enumerate}[leftmargin=1.5cm]
	\item[\normalfont{Case 1.}] $m = \pm1$ and $\nu - \lambda \in \N$.
	\item[\normalfont{Case 2.}] $|m| > 1$, $\lambda \in \Z_{\leq 1-|m|}$ 		and $\nu \in \{0, 1, 2\}$.
	\end{enumerate}
\end{enumerate}
\end{thm}

In Figure~\ref{fig-parameters}, we can see the relations on the parameters in Case 2 of Theorem \ref{mainthm1} above. For a given $(\nu, m)$ (marked with $\times$), the pairs $(\lambda, m)$ that satisfy the condition of Case 2 are marked with {\footnotesize{$\blacksquare$}}.

\begin{rem} The infinitesimal characters of $G = SO_0(4,1)$ on $C^\infty(S^3, \V_\lambda^3)$ and of $G^\prime = SO_0(3,1)$ on $C^\infty(S^2, \L_{m,\nu})$ are respectively:
\begin{equation*}
\left(\frac{3}{2}, \lambda - \frac{3}{2}\right) \text{ and } \left(m, \nu -1\right).
\end{equation*}
We identify a Cartan subalgebra $\h$ of the complexified Lie algebra $\g \simeq \so(5,\C)$ with $\C^2$. The Weyl group is given by $W(\Delta(\g,\h)) = W(BC_2) \simeq \mathfrak{S}_2 \ltimes \left(\Z/2\Z\right)^2$, and its action on $\h^* \simeq \C^2$ is generated by:
\begin{equation*}
\begin{aligned}
\C^2 & \rightarrow \C^2\\
(x,y) & \mapsto (y,x),\\
(x,y) & \mapsto (-x,y),\\
(x,y) & \mapsto (x,-y).
\end{aligned}
\end{equation*}
Analogously, we identify a Cartan subalgebra $\h^\prime$ of the complexified Lie algebra $\g^\prime \simeq \so(4,\C)$ again with $\C^2$. The Weyl group is given by $W(\Delta(\g^\prime,\h^\prime)) = W(D_2) \simeq \mathfrak{S}_2 \ltimes \Z/2\Z$, and its action on $(\h^\prime)^* \simeq \C^2$ is generated by:
\begin{equation*}
\begin{aligned}
\C^2 & \rightarrow \C^2\\
(x,y)& \mapsto (y,x),\\
(x,y)& \mapsto (-x,-y).\\
\end{aligned}
\end{equation*}
Thus, we have bijections
\begin{equation*}
\begin{aligned}
\Hom_{\C\text{-algebra}}\left(Z(\g), \C\right) &\simeq \C^2/W(\Delta(\g,\h)),\\
\Hom_{\C\text{-algebra}}\left(Z(\g^\prime), \C\right) &\simeq \C^2/W(\Delta(\g^\prime,\h^\prime)),
\end{aligned}
\end{equation*}
where $Z(\g)$ (respectively $Z(\g^\prime)$) is the center of the universal enveloping algebra $U(\g)$ (respectively $U(\g^\prime)$) of $\g$ (respectively $\g^\prime$).
We \lq\lq normalize\rq\rq{} the Harish-Chandra isomorphism in a way that the infinitesimal characters of the trivial one-dimensional representation $\textbf{1}$ of $G$ and $G^\prime$ are given as follows:
\begin{equation*}
\rho_G = \left(\frac{3}{2}, \frac{1}{2}\right), \quad \rho_{G^\prime} = \left(1,0\right).
\end{equation*}
\end{rem}

\begin{figure}
\centering
\includegraphics[scale=0.7]{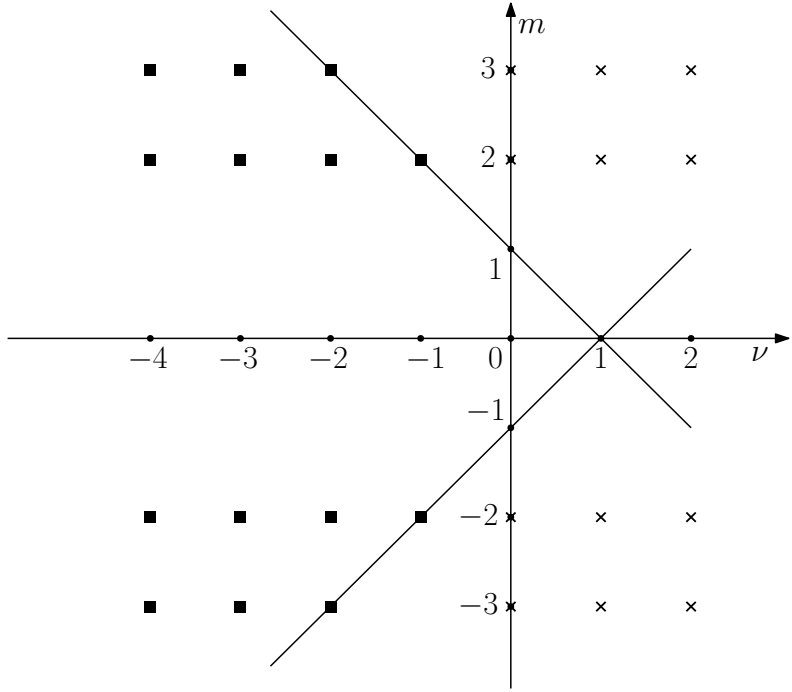}
\caption{relation of the parameters in Case 2 of Theorem \ref{mainthm1}}\label{fig-parameters}
\end{figure}

It follows from Theorem \ref{mainthm1} that the dimension of the space of differential symmetry breaking operators (\ref{DSBO-space}) is at most one. In Theorem \ref{mainthm2} below, we give an explicit formula of the generator $\D_{\lambda, \nu}^m$ in terms of the coordinates $(x_1, x_2, x_3) \in \R^3$ via the conformal compactification $\R^3 \arrowiota S^3$:
\begin{equation*}
\begin{tikzcd}
C^\infty(S^3, \V_\lambda^3) \arrow[d, hookrightarrow, "{\displaystyle{\iota^*}}"'] \arrow[r, dashed] & C^\infty(S^2, \L_{m, \nu}) \arrow[d, hookrightarrow, "{\displaystyle{\iota^*}}"']\\
C^\infty(\R^3, V^3) \arrow[r, "\D_{\lambda, \nu}^m"] & C^\infty(\R^2)
\end{tikzcd}
\end{equation*}
where $\R^2 \subset \R^3$ is realized as $\R^2 = \{(x_1, x_2, 0) : x_1, x_2 \in \R\}$.

To give an explicit formula of $\D_{\lambda, \nu}^m$, let $\{u_1, u_2, u_3\}$ be the standard basis of $V^3 \simeq \C^3$ (see (\ref{basis-V3})), and denote by $\{u_1^\vee, u_2^\vee, u_3^\vee\}$ the dual basis of $(V^3)^\vee$. We write $z = x_1 +ix_2$ so that the Laplacian $\Delta_{\R^2} = \frac{\partial^2}{\partial x_1^2} + \frac{\partial^2}{\partial x_2^2}$ on $\R^2$ is given as $\Delta_{\R^2} = 4\frac{\partial^2}{\partial z \partial \overline{z}}$.

Suppose $\nu-\lambda \in \N$ and let $\widetilde{\C}_{\lambda, \nu}$ denote the following (scalar-valued) differential operator (cf. \cite[2.22]{kkp})

\begin{equation}\label{def-operator-Ctilda}
\begin{aligned}
\widetilde{\C}_{\lambda, \nu} & = \Rest_{x_3 = 0} \circ \left(I_{\nu - \lambda} \widetilde{C}^{\lambda-1}_{\nu-\lambda}\right)\left(-4\frac{\partial^2}{\partial z \partial \overline{z}}, \frac{\partial}{\partial x_3}\right),
\end{aligned}
\end{equation}
where $(I_\ell\Geg_\ell^\mu)(x,y) := x^{\frac{\ell}{2}}\Geg_\ell^\mu\left(\frac{y}{\sqrt{x}}\right)$ is a polynomial of two variables associated with the renormalized Gegenbauer polynomial (see (\ref{Gegenbauer-polynomial(renormalized)})).

 We define constants $A, B, C \in \C$ as follows, where for $\mu \in \C$ and $\ell \in \N$,  $\gamma(\mu, \ell) = 1$ ($\ell$ odd), $\gamma(\mu,\ell) = \mu + \frac{\ell}{2}$ ($\ell$ even) (see (\ref{gamma-def})):
\begin{align}
A &:= (-1)^\nu\left(\lambda + \left[\frac{-\lambda-|m|}{2}\right]\right)^{1-\nu} =
\begin{cases}
\lambda + \left[\frac{-\lambda-|m|}{2}\right], &\text{if } \nu = 0,\label{const-A}\\
-1, &\text{if } \nu = 1,\\
\left(\lambda + \left[\frac{-\lambda-|m|}{2}\right]\right)^{-1}, &\text{if } \nu = 2,
\end{cases}
\\[6pt]
B &:= -2\gamma(\lambda-1, \nu-\lambda-|m|),\label{const-B}\\[4pt]
C &:= |m|(\nu-1) + \lambda - 2.\label{const-C}
\end{align}

Now, we give the complete solution to Problem B.

\begin{thm}\label{mainthm2} Let $\lambda, \nu \in \C$ and $m \in \Z\setminus\{0\}$. Then, any differential symmetry breaking operator in {\normalfont{(\ref{DSBO-space})}} is proportional to $\D_{\lambda, \nu}^m$ given as follows:
\begin{enumerate}[leftmargin=2cm]
\item[\normalfont{Case 1.}] $m = \pm1$ and $\nu - \lambda \in \N$.
\begin{itemize}[leftmargin=-0.5cm]
\item[\normalfont{$\bullet$}] $\nu - \lambda = 0:$
\begin{equation}\label{Operator-particular1}
\D_{\lambda, \lambda}^1 = \Rest_{x_3 = 0} \otimes u_1^\vee,
\end{equation}
\begin{equation}\label{Operator-particular-1}
\D_{\lambda, \lambda}^{-1} = \Rest_{x_3 = 0} \otimes u_3^\vee.
\end{equation}
\item[\normalfont{$\bullet$}] $\nu - \lambda \geq 1:$
\begin{multline}\label{Operator-general1}
\D_{\lambda,\nu}^{1}  =  \left(\lambda + \left[\frac{\nu-\lambda-1}{2}\right]\right) \widetilde{\C}_{\lambda+1, \nu+1} \otimes u_1^\vee\\
+2\gamma(\lambda-1, \nu-\lambda)\widetilde{\C}_{\lambda+1, \nu} \frac{\partial}{\partial \overline{z}} \otimes u_2^\vee
+ 4\widetilde{\C}_{\lambda+1, \nu-1} \frac{\partial^2}{\partial \overline{z}^2} \otimes u_3^\vee,
\end{multline}
\begin{multline}\label{Operator-general-1}
\D_{\lambda,\nu}^{-1} = 4\widetilde{\C}_{\lambda+1, \nu-1} \frac{\partial^2}{\partial z^2} \otimes u_1^\vee\\ -2\gamma(\lambda-1, \nu-\lambda)\widetilde{\C}_{\lambda+1, \nu} \frac{\partial}{\partial z} \otimes u_2^\vee
+ \left(\lambda + \left[\frac{\nu-\lambda-1}{2}\right]\right) \widetilde{\C}_{\lambda+1, \nu+1} \otimes u_3^\vee.
\end{multline}
\end{itemize}
\item[\normalfont{Case 2.}] $|m| > 1$, $\lambda \in \Z_{\leq 1-|m|}$ and $\nu \in \{0, 1, 2\}$.
\begin{itemize}[leftmargin=-0.8cm]
\item[\normalfont{$\bullet$}] $\nu-\lambda = |m|-1:$
\begin{itemize}[leftmargin=0.5cm]
\item[$*$] $m >1$
\begin{equation}\label{Operator-particular+}
\hspace{-0.7cm}\D_{\lambda, \lambda + m-1}^m = \Rest_{x_3 = 0} \circ \frac{\partial^{m-1}}{\partial \overline{z}^{m-1}} \otimes u_1^\vee,
\end{equation}
\item[$*$] $m <-1$
\begin{equation}\label{Operator-particular-}
\D_{\lambda, \lambda-m-1}^m = \Rest_{x_3 = 0} \circ \frac{\partial^{-m-1}}{\partial z^{-m-1}} \otimes u_3^\vee.
\end{equation}
\end{itemize}
\item[\normalfont{$\bullet$}] $\nu - \lambda \geq |m|:$
\begin{itemize}[leftmargin=0.5cm]
\item[$*$] $m >1$
\begin{multline}\label{Operator-general+}
\D_{\lambda, \nu}^m =
2^{2\nu-1}AB \widetilde{\C}_{\lambda+1, 2-\nu-m}\frac{\partial^{2\nu+m-1}}{\partial z^{\nu} \partial \overline{z}^{\nu+m-1}}\otimes u_1^\vee\\
+ \left(-C\widetilde{\C}_{\lambda, \nu-m} + B\frac{\partial}{\partial x_3}\widetilde{\C}_{\lambda+1, \nu -m} \right)\frac{\partial^{m}}{\partial \overline{z}^{m}} \otimes u_2^\vee\\
+ 2B\widetilde{\C}_{\lambda+1, \nu-m} \frac{\partial^{m+1}}{\partial \overline{z}^{m+1}} u_3^\vee,
\end{multline}
\end{itemize}
\begin{itemize}[leftmargin=0.5cm]
\item[$*$] $m < -1$
\begin{multline}\label{Operator-general-}
 \D_{\lambda, \nu}^m = 
 +  2B\widetilde{\C}_{\lambda+1, \nu+m} \frac{\partial^{-m+1}}{\partial z^{-m+1}}\otimes u_1^\vee\\
+ \left(C\widetilde{\C}_{\lambda, \nu+m} - B\frac{\partial}{\partial x_3}\widetilde{\C}_{\lambda+1, \nu +m} \right)\frac{\partial^{-m}}{\partial z^{-m}} \otimes u_2^\vee\\
+2^{2\nu-1}AB \widetilde{\C}_{\lambda+1, 2-\nu+m}\frac{\partial^{2\nu-m-1}}{\partial z^{\nu-m-1} \partial \overline{z}^{\nu}}\otimes u_3^\vee.
\end{multline}
\end{itemize}
\end{itemize}
\end{enumerate}
\end{thm}
Both Theorems \ref{mainthm1} and \ref{mainthm2} will be proved in Section \ref{section-proof-of-main-thms} for $m \geq 1$ and in Section \ref{section-case_m_lessthan_-1} for $m \leq -1$.

\begin{rem}\label{rem-mainthms}
(1) Note that the case $m = 0$ is not included in Theorems \ref{mainthm1} and \ref{mainthm2} since it is already known. The results for $m = 0$ can be deduced from \cite[Thms. 2.8 and 2.9]{kkp} by using the branching rule $O(3,1) \downarrow SO_0(3,1)$. In fact, using the same notation as in \cite{kkp}, for $\nu, \lambda \in \C$ by \cite[(2.33), (2.34)]{kkp} we have: 
\begin{equation*}
\begin{aligned}
I(1, \lambda)_\alpha\big|_{SO_0(4,1)} &\simeq I(1,\lambda) \simeq I(2,\lambda) \simeq C^\infty(S^3, \V^3_\lambda),\\
J(0, \nu)_\beta\big|_{SO_0(3,1)} &\simeq J(0, \nu) \simeq J(2, \nu) \simeq C^\infty(S^2, \L_{0, \nu}).
\end{aligned}
\end{equation*}
Now, for $\delta := \nu -\lambda \mod 2\Z$ by \cite[Thm. 2.10]{kkp} we deduce
\begin{equation*}
\begin{aligned}
\Diff_{SO_0(3,1)}&\big(C^\infty(S^3, \V^3_\lambda), C^\infty(S^2, \L_{0, \nu})\big) \\
& \simeq \Diff_{SO_0(3,1)}\left(I(1,\lambda), J(0,\nu)\right) \\
& \simeq
\Diff_{O(3,1)}\left(I(1,\lambda)_0, J(0, \nu)_\delta\right) \oplus
\Diff_{O(3,1)}\left(I(2,\lambda)_0, J(0, \nu)_{\delta}\right)\\
& \simeq \C\widetilde{\C}_{\lambda, \nu}^{1,0} \oplus \C\widetilde{\C}_{1,2}^{2,0}.
\end{aligned}
\end{equation*}
In particular, the space above is non-zero if and only if $\nu - \lambda \in \N$, and its dimension is smaller than or equal to two, being two only in the case $(\lambda, \nu) = (1,2)$.

(2) Although it cannot be deduced from \cite{kkp}, the cases $m = \pm 1$ are related to the results in \cite{kkp} in a similar way as in the case $m = 0$ above. As before, note that for $\nu, \lambda \in \C$ we have: 
\begin{equation*}
\begin{gathered}
J(1, \nu)_\beta\big|_{SO_0(3,1)} \simeq J(1, \nu) \simeq C^\infty(S^2, \L_{1, \nu}) \oplus C^\infty(S^2, \L_{-1, \nu}).
\end{gathered}
\end{equation*}
Thus, for $\delta := \nu -\lambda \mod 2\Z$ by \cite[Thm. 2.10]{kkp} we deduce
\begin{equation*}
\begin{aligned}
\Diff_{SO_0(3,1)}&\left(C^\infty(S^3, \V^3_\lambda), C^\infty(S^2, \L_{1, \nu}) \oplus C^\infty(S^2, \L_{-1, \nu})\right) \\
& \simeq \Diff_{SO_0(3,1)}\left(I(1,\lambda), J(1, \nu)\right) \\
& \simeq \Diff_{O(3,1)}\left(I(1,\lambda)_0, J(1, \nu)_\delta\right) \oplus
\Diff_{O(3,1)}\left(I(1,\lambda)_0, J(1, \nu)_{\delta}\right)\\ 
&\simeq \C\widetilde{\C}_{\lambda, \nu}^{1,1} \oplus \C\widetilde{\C}_{\lambda, \nu}^{1,1}.
\end{aligned}
\end{equation*}
In fact, we will show that the differential operator $\widetilde{\C}_{\lambda, \nu}^{1,1}$ can be given as a linear combination of the operators $\D_{\lambda, \nu}^1$ and $\D_{\lambda, \nu}^{-1}$ (see Section \ref{section-KoKuPe}).
\end{rem}

\section{The F-method for flag varieties}
In this section we summarize the key tool we use to prove our main results: The F-method (\cite{kob2013}, \cite{kob2014}, \cite{kob-pev1}, \cite{kob-pev2}). This method is based on the \lq\lq algebraic Fourier transform\rq\rq{} of generalized Verma modules (\cite[Sec. 3]{kob-pev1}). 

First, we start recalling the definition of the algebraic Fourier transform. Let $E$ be a finite-dimensional complex vector space and denote by $\Weyl(E)$ the Weyl algebra over $E$, i.e., the ring of holomorphic differential operators on $E$ with polynomial coefficients.

\begin{defn}[{\cite[Def. 3.1]{kob-pev1}}] The \textbf{algebraic Fourier transform} is defined as an isomorphism between the Weyl algebras on $E$ and its dual space $E^\vee$
\begin{equation*}
\Weyl(E) \rightarrow \Weyl(E^\vee), \enspace T\mapsto \widehat{T},
\end{equation*}
induced by the following relations
\begin{equation*}
\widehat{\frac{\partial}{\partial z_j}} := -\zeta_j, \enspace \widehat{z_j}:= \frac{\partial}{\partial \zeta_j}, \quad \text{ for all } 1\leq j \leq n := \dim E,
\end{equation*}
where $(z_1, \ldots, z_n)$ are coordinates on $E$ and $(\zeta_1, \ldots, \zeta_n)$ are the dual coordinates on $E^\vee$.
\end{defn}

Now, let $X, Y$ be two smooth manifolds and suppose that a smooth map between them $p: Y \rightarrow X$ is given. Let $\V, \W$ be two vector bundles over $X$ and $Y$ respectively. We write $C^\infty(X, \V)$ (resp. $C^\infty(Y, \W)$) for the space of smooth sections, endowed with the Fréchet topology of uniform convergence of sections and their derivatives of finite order on compact sets. 

We recall the definition of a \textit{differential operator}, not only for sections on the same manifold, but in general for sections on two different manifolds.

\begin{defn}[{\cite[Def. 2.1]{kob-pev1}}]\label{def-diffop} A continuous linear map $T: C^\infty(X, \V)\rightarrow C^\infty(Y, \W)$ is said to be a \textbf{differential operator} if it satisfies
\begin{equation*}
p\left(\supp Tf \right) \subset \supp f, \quad \text{for all } f \in C^\infty(X,\V).
\end{equation*}
\end{defn}

Moreover, suppose that $G$ is a reductive Lie group acting transitively on $X$ such that $X = G/P$ for some parabolic subgroup $P$ of $G$, and that $G^\prime \subset G$ is a reductive subgroup of $G$ acting transitively on $Y$ such that $Y = G^\prime/P^\prime$ for some parabolic subgroup $P^\prime$ of $G^\prime$. 

Let $P = MAN_+$ be the Langlands decomposition of $P$ and denote the Lie algebras of $P, M, A, N_+$ by $\p(\R), \m(\R), \a(\R), \n_+(\R)$ and their complexified Lie algebras by $\p, \m, \a, \n_+$ respectively. Let $\g(\R) = \n_-(\R) + \p(\R)$ be the Gelfand--Naimark decomposition of $\g(\R)$. A similar notation will be used for $P^\prime$.

Given $\lambda \in \a^* \simeq \Hom_\R(\a(\R), \C)$, we define a one-dimensional representation $\C_\lambda$ of $A$ by 
\begin{equation*}
A \rightarrow \C^\times, \enspace a\mapsto a^\lambda := e^{\langle\lambda, \log a\rangle}.
\end{equation*}
Let $(\sigma, V)$ be a representation of $M$ and let $\lambda \in \a^*$. We write $\sigma_\lambda \equiv \sigma \boxtimes \C_\lambda$ for the representation of $MA$ given by $ma \mapsto a^\lambda\sigma(m)$. This representation may be also regarded as a representation of $P$ by letting $N_+$ act trivially. We define $\V := G \times_P V$ as the $G$-equivariant vector bundle over $X = G/P$ associated to $\sigma_\lambda$. Analogously, given $(\tau, W)$ a representation of $M^\prime$ and $\nu \in \a^*$, we set $\tau_\nu \equiv \tau \boxtimes \C_\nu$ and define $\W := G^\prime \times_{P^\prime} W$ as the $G^\prime$-equivariant vector bundle over $Y$ associated to $\tau_\nu$.

Now, let $\C_{2\rho}$ be the following one-dimensional representation of $P$
\begin{equation*}
P \ni p \mapsto |\det \left(\Ad(p): \n_+(\R) \rightarrow \n_+(\R)\right)|,
\end{equation*}
whose infinitesimal representation is given by
\begin{equation*}
\a(\R) \ni Z \mapsto \tr \left(\ad(Z): \n_+(\R) \rightarrow \n_+(\R)\right).
\end{equation*}
Observe that the density bundle $\Omega_X = |\bigwedge^\text{top}T^\vee X|$ over $X$ is given by the line bundle $G \times_P \C_{2\rho}$. For $\lambda \in \a^*$ we define a new representation $\mu$ of $P$ by taking the tensor product of the contragredient representation of $\sigma_\lambda$ and $\C_{2\rho}:$
\begin{equation*}
\mu \equiv \sigma^*_\lambda := \sigma^\vee \boxtimes \C_{2\rho-\lambda},
\end{equation*}
and form the induced representation
\begin{equation*}
\pi_\mu \equiv \pi_{(\sigma, \lambda)^*} = \Ind_{P}^G(\sigma_\lambda^*)
\end{equation*}
of $G$ on $C^\infty(X,\V^*)$, where $\V^* := G \times_P \sigma_\lambda^*$. We may regard this space as a subspace of $C^\infty(\n_-(\R), V^\vee) \simeq C^\infty(\n_-(\R)) \otimes V^\vee$ via the restriction to the open Bruhat cell $\n_-(\R) \hookrightarrow G/P = X$.

The infinitesimal action $d \pi_\mu \equiv d \pi_{(\sigma, \lambda)^*}$ induces (see \cite{kob-pev1}) a Lie algebra homomorphism 
\begin{equation*}
d\pi_\mu: \g\rightarrow \Weyl(\n_-) \otimes \End(V^\vee).
\end{equation*}
The composition of $d\pi_\mu$ with the algebraic Fourier transform leads to a Lie algebra homomorphism
\begin{equation*}
\widehat{d\pi_\mu}: \g\rightarrow \Weyl(\n_+) \otimes \End(V^\vee),
\end{equation*}
where we identify $\n_-^\vee \simeq \n_+$.

In fact, a closed formula for $d \pi_\mu$ was given in \cite{kob-pev1} as follows. Take $G_\C$ and $P_\C = L_\C\exp(\n_+)$ connected complex Lie groups with Lie algebras $\g$ and $\p = \l + \n_+ = \m + \a + \n_+$ respectively. According to the Gelfand--Naimark decomposition $\g = \n_- + \l + \n_+$ of $\g$, we have a diffeomorphism
\begin{equation*}
\n_- \times L_\C \times \n_+ \rightarrow G_\C, (Z, \ell, Y) \mapsto (\exp Z) \ell (\exp Y),
\end{equation*}
onto an open dense subset $G_\C^\text{reg}$ of $G_\C$. Now, let
\begin{equation*}
p_\pm: G_\C^\text{reg} \rightarrow \n_\pm, \quad p_0: G_\C^\text{reg}\rightarrow L_\C,
\end{equation*}
be the projections characterized by
\begin{equation*}
\exp(p_-(g))p_0(g)\exp(p_+(g)) = g,
\end{equation*}
and let $\alpha, \beta$ be the following maps independent of the choice of the complex Lie group $G_\C$:
\begin{equation*}
\begin{aligned}
\alpha: \g \times \n_- &\rightarrow \l, &(Y,Z)\mapsto \frac{d}{dt}\Big|_{t=0}p_0(e^{tY}e^Z),\\
\beta: \g \times \n_- &\rightarrow \n_-, &(Y,Z)\mapsto \frac{d}{dt}\Big|_{t=0}p_-(e^{tY}e^Z).
\end{aligned}
\end{equation*}
Then, for $F \in C^\infty(\n_-, V^\vee)$, the following formula holds \cite[(3.13)]{kob-pev1}:
\begin{equation}\label{dpimu-general_expression}
\left(d \pi_\mu(Y)F\right)(Z) = d\mu(\alpha(Y,Z))F(Z) - \left(\beta(Y, \cdot)F\right)(Z), \enspace \text{ for } Y \in \g.
\end{equation}
Here $\beta(Y, \cdot)$ is regarded as a holomorphic vector field on $\n_-$ through the identification 
\begin{equation*}
\n_- \ni Z \mapsto \beta(Y,Z) \in \n_- \simeq T_Z \n_-.
\end{equation*}

\begin{defn}[{\cite[Sec. 4.1]{kob-pev1}}] Let $E$ be a finite-dimensional complex vector space and denote by $E^\vee$ its dual space. Write $n:= \dim E$ and let $\Diff^\text{const}(E)$ denote the ring of differential operators over $E$ with constant coefficients. Then, there is a natural isomorphism
\begin{equation}\label{symbolmap}
\Symb: \Diff^\text{const}(E) \arrowsimeq \Pol(E^\vee),
\hspace{0.3cm} \sum_{\alpha \in \N^n} a_\alpha \frac{\partial^\alpha}{\partial z^\alpha} \longmapsto  \sum_{\alpha \in \N^n} a_\alpha \zeta^\alpha, 
\end{equation}
which we will call \textbf{the symbol map}.
\end{defn}
We let
\begin{equation*}
\Hom_{L^\prime}\left(V, W \otimes \Pol(\n_+)\right) := \{\psi \in \Hom_\C\left(V, W \otimes \Pol(\n_+)\right) :\text{\normalfont{ (\ref{F-system-1}) holds}}\},
\end{equation*}
\begin{equation}\label{F-system-1}
\psi \circ \sigma_\lambda(\ell) = \tau_{\nu}(\ell) \circ \Ad_\#(\ell)\psi \quad \forall \enspace \ell \in L^\prime,
\end{equation}
where the action $\Ad_\#$ of $L^\prime$ on $\Pol(\n_+)$ is given by
\begin{equation*}
\begin{aligned}
\Ad_\#(\ell): \Pol(\n_+) &\longrightarrow \Pol(\n_+)\\
\hspace{2.8cm} p(\cdot) &\longmapsto p\left(\Ad(\ell^{-1}) \cdot\right).
\end{aligned}
\end{equation*}
In addition, we set
\begin{equation*}
\Sol(\n_+; \sigma_\lambda, \tau_\nu) := \{\psi \in \Hom_{L^\prime}\left(V, W \otimes \Pol(\n_+)\right) :\text{\normalfont{(\ref{F-system-2}) holds}}\},
\end{equation*}
\begin{equation}\label{F-system-2}
\left(\widehat{d \pi_\mu}(C) \otimes \id_{W} \right)\psi = 0 \quad \forall \enspace C\in \n_+^\prime.
\end{equation}
Suppose that $\n_+$ is abelian. Then we have (see \cite[Sec. 4.3]{kob-pev1})
\begin{equation*}
\Diff_{G^\prime}\left(C^\infty(X, \V), C^\infty(Y, \W)\right) \subset \Diff^\text{const}(\n_+) \otimes \Hom_\C(V,W).
\end{equation*}
By applying the symbol map, we obtain the following.
\begin{thm}[{\cite[Thm. 4.1]{kob-pev1}}]\label{F-method-thm} Suppose that the Lie algebra $\n_+$ is abelian. Then, there exists a linear isomorphism
\begin{equation*}
\begin{tikzcd}[column sep = 1cm]
\Diff_{G^\prime}\left(C^\infty(X, \V), C^\infty(Y, \W)\right) \arrow[r, "\displaystyle{_{\Symb \otimes \id}}", "\thicksim"'] & \Sol(\n_+; \sigma_\lambda, \tau_\nu).
\end{tikzcd}
\end{equation*}
\end{thm}
By Theorem \ref{F-method-thm} above, in order to determine $\Diff_{G^\prime}\left(C^\infty(X, \V), C^\infty(Y, \W)\right)$, it suffices to compute $\Sol(\n_+; \sigma_\lambda, \tau_\nu)$. To do so, one may proceed with the following two steps:
\begin{itemize}
\item \underline{Step 1}. Determine the generators of $\Hom_{L^\prime}\left(V, W \otimes \Pol(\n_+)\right)$.
\item \underline{Step 2}. Solve $\left(\widehat{d \pi_\mu}(C) \otimes \id_{W} \right)\psi = 0, \quad \forall \enspace C\in \n_+^\prime$.
\end{itemize}
In Sections \ref{section-step1} and \ref{section-step2} respectively, we shall carry out Steps 1 and 2 above for the pair $\left(G, G^\prime\right) = \left(SO_0(4,1), SO_0(3,1)\right)$.

\section{Principal series representations of $G = SO_0(4,1)$ and $G^\prime = SO_0(3,1)$}\label{section-setting}
In the previous section we presented the F-method in a general setting between two flag varieties $X=G/P$ and $Y=G^\prime/P^\prime$ with vector bundles $\V \rightarrow X$ and $Y\rightarrow \W$ over them. Through the rest of this paper, we restrict ourselves to the case $(G, G^\prime) = (SO_0(4,1), SO_0(3,1))$ and $(X,Y) = (S^3,S^2)$. In this section, we realize $G$ and $G^\prime$ in a geometric setting and define their principal series representations $\Ind_P^G(\sigma_\lambda^3)$ and $\Ind_{P^\prime}^{G^\prime}(\tau_{m, \nu})$ respectively. We use practically the same notation as in \cite{kkp}.
 
Let $Q_{4,1}$ be the following bilinear form on $\R^5$ of signature $(4,1)$.
\begin{equation*}
Q_{4,1}(x) := x_0^2 + x_1^2 + x_2^2 + x_3^2 - x_4^2, \quad \text{for } x = {}^t(x_0, \dots, x_4) \in \R,
\end{equation*}
and realize the de Sitter group $O(4,1)$ as follows:
\begin{equation*}
O(4,1) = \{g \in GL(5,\R) : Q_{4,1}(gx) = Q_{4,1}(x) \enspace \forall x \in \R^5\}.
\end{equation*}
Let $\g(\R) := \so(4,1)$ be its Lie algebra and denote by $E_{p,q} \enspace (1 \leq p, q \leq 5)$ the elementary $5\times 5$ matrix consisting of a $1$ in the $(p,q)$ position and zeros in the others. We define the following elements of $\g(\R)$:
\begin{equation}\label{elements}
\begin{aligned}
H_0 & := E_{1,5} + E_{5,1},\\
C_j^+ & := -E_{1,j+1} + E_{j+1,1} - E_{5,j+1} - E_{j+1,5} & (1 \leq j  \leq 3),\\
C_j^- & := -E_{1,j+1} + E_{j+1,1} + E_{5,j+1} + E_{j+1,5} & (1 \leq j  \leq 3),\\
X_{p,q} & := E_{q+1, p+1} - E_{p+1, q+1} & (1 \leq p \leq q \leq 3).
\end{aligned}
\end{equation}
Note that the indexing notation is slightly different from that of \cite[(2.1)]{kkp}, where the first row (column) of a matrix has index 0, while here it has index 1. 

A simple computation shows the following relations:
\begin{equation}\label{brackets-Cpm}
\frac{1}{2}[C_j^+, C_k^-] = X_{j, k} -\delta_{j,k}H_0, \quad X_{p,q} = -X_{q,p}.
\end{equation}
Observe that the sets $\{C_j^+: 1 \leq j  \leq 3\}$, $\{C_j^-: 1 \leq j  \leq 3\}$ and $\{X_{p,q} : 1 \leq p,q \leq 3\}\cup\{H_0\}$ determine bases of the following Lie subalgebras of $\g(\R)$:
\begin{equation*}
\begin{aligned}
\n_+(\R) & := \ker(\ad(H_0) - \id)\subset \g(\R),\\
\n_-(\R) & := \ker(\ad(H_0) + \id)\subset \g(\R),\\
\l(\R) & := \ker\ad(H_0)\subset \g(\R).
\end{aligned}
\end{equation*}

Let $\Xi$ be the isotropic cone defined by $Q_{4,1}$:
\begin{equation*}
\Xi := \{x \in \R^5\setminus\{0\} : Q_{4,1}(x) = 0\} \subset \R^5,
\end{equation*}
which is clearly invariant under the dilatation of the multiplicative group $\R^\times = \R\setminus \{0\}$.
Observe that the projection $p: \Xi \rightarrow S^3, \enspace x \mapsto p(x) := \frac{1}{x_4}{}^t(x_0, x_1, x_2, x_3)$ determines a bijection $\widetilde{p}: \Xi/\R^\times \arrowsimeq S^3$.

Note that the natural action of $O(4,1)$ on $\R^5$ leaves $\Xi$ and $\Xi/\R^\times$ invariant. Now let $G := SO_0(4,1)$ be the identity component of $O(4,1)$ and define 
\begin{equation*}
P := \stab_G(1:0:0:0:1).
\end{equation*}
This is a parabolic subgroup of $G$ with Langlands decomposition $P = LN_+ = MAN_+$, where
\begin{equation*}
\begin{gathered}
M = \left\{
\begin{pmatrix}
1 & &\\
& g &\\
& & 1
\end{pmatrix} : g \in SO(3)
\right\} \simeq SO(3),\\[6pt]
A = \exp(\R H_0) \simeq SO_0(1,1), \quad
N_+ = \exp(\n_+(\R)) \simeq \R^3.
\end{gathered}
\end{equation*}
The action of $G$ on $\Xi/\R^\times$ is transitive; therefore, we obtain bijections
\begin{equation*}
G/P \arrowsimeq \Xi/\R^\times \arrowsimeq S^3.
\end{equation*}

Moreover, for $x = {}^t(x_1, x_2, x_3) \in \R^3$ let $Q_3(x) \equiv \|x\|_{\R^3} := x_1^2 + x_2^2 + x_3^2$ and let $N_- := \exp(\n_-(\R))$. We define a diffeomorphism $n_-: \R^3 \rightarrow N_-$ by
\begin{equation}
\label{open-Bruhat-cell}
\begin{gathered}
n_-(x) := \exp\left(\sum_{j = 1}^3 x_jC_j^-\right) = 
\begin{pmatrix}
1 - \frac{1}{2}Q_3(x) & -x_1 & -x_2 & -x_3 & -\frac{1}{2}Q_3(x)\\
x_1 & 1 & 0 & 0 & x_1\\
x_2 & 0 & 1 & 0 & x_2\\
x_3 & 0 & 0 & 1 & x_3\\
\frac{1}{2}Q_3(x) & x_1 & x_2 & x_3 & 1 + \frac{1}{2}Q_3(x)
\end{pmatrix}.
\end{gathered}
\end{equation}
Then, this map gives the coordinates of the open Bruhat cell:
\begin{equation*}
\n_-(\R) \hookrightarrow G/P \simeq S^3, \enspace Z \mapsto \exp(Z) \cdot P.
\end{equation*}
In fact, if we define
\begin{equation*}
\iota: \R^3 \hookrightarrow S^3, \quad {}^t\!(x_1, x_2, x_3) \mapsto \frac{1}{1+Q_3(x)}{}^t\!(1-Q_3(x), 2x_1, 2x_2, 2x_3),
\end{equation*}
and note that
\begin{equation*}
n_-(x){}^t\!(1, 0, 0, 0, 1) = {}^t\!(1-Q_3(x), 2x_1, 2x_2, 2x_3, 1+Q_3(x)),
\end{equation*}
then we have $\iota(x) = \widetilde{p}\left[n_-(x)^t(1,0,0,0,1)\right]$. We observe that $\iota$ is nothing but the inverse of the stereographic projection
\begin{equation*}
S^3\setminus\{{}^t(-1,0,0,0,0)\} \rightarrow \R^3, \quad {}^t(\omega_0, \omega_1, \omega_2, \omega_3) \mapsto \frac{1}{1+\omega_0}{}^t(\omega_1, \omega_2, \omega_3).
\end{equation*}

For any $\lambda \in \C$ let $(\sigma^3_\lambda, V_\lambda^3)$ be the following $3$-dimensional representation of $L \equiv MA$
\begin{equation*}
\begin{aligned}
\sigma^3_\lambda: L = M \times A &\longrightarrow GL_\C(V^3)\\
\hspace{-0.2cm}(B, \exp(t H_0)) &\longmapsto e^{\lambda t}\sigma^3(B),
\end{aligned}
\end{equation*}
where $\sigma^{2N+1}$ denotes the $(2N+1)$-dimensional irreducible representation of $SO(3)$ realized in $V^{2N+1}$; the space of homogeneous polynomials of degree $2N$ in two variables (see Section \ref{section-fdrepsofSO(3)}). Note that as a representation of $SO(3)$, $\sigma^3$ is isomorphic to the natural representation on $\C^3$. We observe that this representation can be written as $\sigma^3 \boxtimes \C_\lambda$, where $\C_\lambda$ is the character of $A$ defined by exponentiation:
\begin{equation*}
\C_\lambda: A \rightarrow \C^\times, \quad a = e^{tH_0} \mapsto a^\lambda := e^{\lambda t}.
\end{equation*}

By letting $N_+$ act trivially, we think of $\sigma_\lambda^3$ as a representation of $P$ and define the following homogeneous vector bundle over $G/P\simeq S^3$ associated to $(\sigma_\lambda^3, V_\lambda^3):$
\begin{equation}\label{vectorbundle-V}
\V_\lambda^3 := G \times_P V_\lambda^3.
\end{equation}
Then, we form the principal series representation $\Ind_P^G(\sigma_\lambda^3)$ of $G$ on the space of smooth sections $C^\infty(G/P, \V_\lambda^3)$, or equivalently, on the space
\begin{equation*}
\begin{aligned}
C^\infty(G, V_\lambda^3)^P =
\{f \in C^\infty(G, V_\lambda^3) : f(gman) = \sigma^3(m)^{-1}& a^{-\lambda}f(g) \\
& \text{for all } m\in M, a\in A, n\in N_+, g\in G\}.
\end{aligned}
\end{equation*}
Its flat picture ($N$-picture) is defined on $C^\infty(\R^3, V^3)$ via the restriction to the open Bruhat cell (\ref{open-Bruhat-cell}):
\begin{equation*}
\begin{aligned}
C^\infty(G/P, \V^3_\lambda) \simeq C^\infty(G, V_\lambda^3)^P &\longrightarrow C^\infty(\R^3, V^3)\\
\hspace{5.2cm}f &\longmapsto F(x) := f(n_-(x)).
\end{aligned}
\end{equation*}

Now, we realize $G^\prime = SO_0(3,1)$ as $G^\prime \simeq \stab_G{}^t(0,0,0,1,0) \subset G$, and denote its Lie algebra by $\g^\prime(\R)$. The action of this subgroup leaves invariant $\Xi\cap\{x_3 = 0\}$, thus it acts transitively on
\begin{equation*}
S^2 = \{(y_0, y_1, y_2, y_3) \in S^3 : y_3 = 0\} \simeq \left(\Xi\cap\{x_3 = 0\}\right)/\R^\times.
\end{equation*}
Then, $P^\prime := P \cap G^\prime$ is a parabolic subgroup of $G^\prime$ with Langlands decomposition $P^\prime = L^\prime N_+^\prime = M^\prime A N_+^\prime$, where
\begin{equation}\label{realization-Mprime}
\begin{gathered}
M^\prime = M\cap G^\prime =\left\{
\begin{pmatrix}
1 & & &\\
& B & &\\
& & 1 &\\
& & & 1
\end{pmatrix} : B\in SO(2) \right\} \simeq SO(2),\\[6pt]
N^\prime_+ =  N_+ \cap G^\prime = \exp(\n_+^\prime(\R)) \simeq \R^2.
\end{gathered}
\end{equation}
Here, the Lie algebras $\n_\pm^\prime(\R)$ stand for $\n_\pm^\prime(\R) := \n_\pm(\R)\cap \g^\prime$. In particular, the sets $\{C_j^+: 1 \leq j  \leq 2\}$ and $\{C_j^-: 1 \leq j  \leq 2\}$ 
determine bases of $\n_+^\prime(\R)$ and $\n_-^\prime(\R)$ respectively.
 
For any $\nu \in \C$ and any $m \in \Z$, let $(\tau_{m, \nu}, \C_{m, \nu})$ be the following one-dimensional representation of $L^\prime \equiv M^\prime A$
\begin{equation*}
\begin{gathered}
\tau_{m, \nu}: L^\prime = M^\prime \times A \longrightarrow \C^\times\\
\hspace{1.4cm}(
\begin{pmatrix}
\cos s & -\sin s\\
 \sin s & \cos s
\end{pmatrix}
, \exp(t H_0)) \longmapsto e^{ims}e^{\nu t}.
\end{gathered}
\end{equation*}
As before, by letting $N_+^\prime$ act trivially, we think of $\tau_{m, \nu}$ as a representation of $P^\prime$ and define the following homogeneous vector bundle over $G^\prime/P^\prime\simeq S^2$ associated to $(\tau_{m,\nu}, \C_{m,\nu}):$
\begin{equation}\label{vectorbundle-L}
\L_{m,\nu} := G^\prime \times_{P^\prime} \C_{m,\nu}.
\end{equation}
We form the principal series representation $\Ind_{P^\prime}^{G^\prime}(\tau_{m,\nu})$ of $G^\prime$ on the space $C^\infty(G^\prime/P^\prime, \L_{m,\nu})$ of smooth sections for $\L_{m,\nu}$ .

It follows from Theorem \ref{F-method-thm}, that the symbol map $\Symb$ gives a linear isomorphism
\begin{equation}\label{isomorphism-Fmethod-concrete}
\begin{tikzcd}[column sep = 1cm]
\Diff_{G^\prime}\left(C^\infty(S^3, \V_\lambda^3), C^\infty(S^2, \L_{m,\nu})\right) \arrow[r, "\displaystyle{_{\Symb \otimes \id}}", "\thicksim"'] & \Sol(\n_+; \sigma_\lambda^3, \tau_{m, \nu}),
\end{tikzcd}
\end{equation}
where
\begin{equation}\label{F-system}
\Sol(\n_+; \sigma_\lambda^3, \tau_{m, \nu}) := \{\psi \in \Hom_\C\left(V_\lambda^3, \C_{m,\nu} \otimes \Pol(\n_+)\right) :\text{(\ref{F-system-1-concrete}), (\ref{F-system-2-concrete}) hold}\},
\end{equation}
\begin{equation}\label{F-system-1-concrete}
\psi \circ \sigma_\lambda^3(\ell) = \tau_{m, \nu}(\ell) \circ \Ad_\#(\ell)\psi, \quad \forall \enspace \ell \in L^\prime=M^\prime A,
\end{equation}
\begin{equation}\label{F-system-2-concrete}
\left(\widehat{d \pi_\mu}(C) \otimes \id_{\C_{m,\nu}} \right)\psi = 0, \quad \forall \enspace C\in \n_+^\prime.
\end{equation}
From the next section we consider Steps 1 and 2 of the F-method to determine $\Sol(\n_+; \sigma_\lambda^3, \tau_{m, \nu})$.

\section{Step 1: Generators of $\Hom_{L^\prime}\left(V_\lambda^3, \C_{m,\nu} \otimes \Pol(\n_+)\right)$}\label{section-step1}
In this section we tackle the first step of the F-method, in which we use finite-dimensional representation theory to determine the generators of the space
\begin{equation}\label{space-step1}
 \Hom_{L^\prime}\left(V_\lambda^3, \C_{m,\nu} \otimes \Pol(\n_+)\right).
\end{equation}
The second step, devoted to solving the differential equation (\ref{F-system-2-concrete}), is summarized in Section \ref{section-step2}.

In the following, we identify $\Pol(\n_+)$ with $\Pol[\zeta_1,\zeta_2, \zeta_3]$, where $(\zeta_1, \zeta_2, \zeta_3)$ are the coordinates of $\n_+$ with respect to the basis $\{C_j^+: 1\leq j\leq 3\}$. A direct computation shows that, via the realizations (\ref{realization-Mprime}), the action $\Ad_\#$ of $L^\prime = SO(2) \times A$ on $\Pol(\n_+)$ is equivalent to the following standard action:
\begin{equation}\label{action-SO(2)-on-Pol}
\begin{aligned}
SO(2) \times A \times \Pol(\zeta_1,\zeta_2, \zeta_3) & \longrightarrow \Pol(\zeta_1,\zeta_2, \zeta_3)\\
\enspace (B, \exp(tH_0), p(\cdot))& \longmapsto p\left(e^{-t}\begin{pmatrix}
B &\\
& 1
\end{pmatrix}^{-1}\cdot\right).
\end{aligned}
\end{equation}

\subsection{Finite-dimensional irreducible representations of SO(3)}\label{section-fdrepsofSO(3)}
Here we recall a well-known result about the classification of the finite-dimensional irreducible representations of $SO(3)$ (see, for example \cite{hall}, \cite{knapp}).

It is known that any finite-dimensional irreducible representation of $SU(2)$ is parametrized by a non-negative integer $k \in \N$ and can be realized in the vector space of homogeneous polynomials of degree $k$ in two variables as follows:
\begin{equation*}
\begin{gathered}
\theta^k: SU(2) \longrightarrow GL_\C\left(V^{k+1}\right), \quad
g \longmapsto \left(p(\cdot) \mapsto p\left(g^{-1} \cdot\right)\right),
\end{gathered}
\end{equation*}
where $V^{k+1} := \{p \in \C[\xi_1, \xi_2]: \deg(p) = 2k\}$.

Let $\varpi$ be the $2:1$ covering map of $SO(3)$:
\begin{equation*}
\varpi: SU(2) \twoheadrightarrow SO(3),
\end{equation*}
which has kernel $\ker \varpi = \{\pm I_2\}$ and can be written explicitly as follows (cf. \cite[Prop. 1.19]{hall}):
\begin{equation}\label{expression-covering}
\begin{gathered}
\varpi(U(x,y)) = \begin{pmatrix}
\Re(x^2-y^2) & \Im(x^2+y^2) & -2\Re(xy)\\[3pt]
-\Im(x^2-y^2) & \Re(x^2+y^2) & 2\Im(xy)\\[3pt]
2\Re(x\overline{y}) & 2\Im(x\overline{y}) & |x|^2-|y|^2
\end{pmatrix},
\end{gathered}
\end{equation}
for $U(x,y) := \begin{pmatrix}
x & y\\
-\overline{y} &\overline{x}
\end{pmatrix} \in SU(2)$.

We denote by $\widetilde{\varpi}$ the associated isomorphism $\widetilde{\varpi}: SU(2)/\{\pm I_2\} \arrowsimeq SO(3)$.
Then, it can be proved that the representation $\theta^k$ factors through $\varpi$ as a representation of $SO(3)$ if and only if $k = 2N$ (for some $N \in \N$). In other words, the following holds:
\begin{equation*}
\theta^k(\ker \varpi) = \{\id_{V^k}\} \Leftrightarrow k \in 2\N.
\end{equation*}
Thus, any finite-dimensional representation of $SO(3)$ is of the form
\begin{equation*}
\begin{gathered}
\sigma^{2N+1}: SO(3) \longrightarrow GL_\C\left(V^{2N+1}\right), \quad
B \longmapsto \left(p(\cdot) \mapsto p\left(\left(\widetilde{\varpi}^{-1}(B)\right)^{-1} \cdot\right)\right).
\end{gathered}
\end{equation*}
For $N = 1$, we consider the following basis of $V^3$
\begin{equation}\label{basis-V3}
B(V^3) := \{u_1, u_2, u_3\} = \{\xi_1^2, \xi_1\xi_2, \xi_2^2\}.
\end{equation}

\subsection{Description of $ \Hom_{L^\prime}\left(V_\lambda^3, \C_{m,\nu} \otimes \Pol(\n_+)\right)$}
In this subsection we give a complete characterization of the space (\ref{space-step1}). As in (\ref{realization-Mprime}), we realize $SO(2)$ in $SO(3)$ as
\begin{equation}\label{realization-SO(2)inSO(3)}
SO(2) \simeq \left\{\begin{pmatrix}
B &\\
&1
\end{pmatrix}: B\in SO(2)\right\} \subset SO(3).
\end{equation}
We start from a simple but useful result about harmonic polynomials.

For $k, \ell \in \N$ we set 
$$\Harm^k(\C^\ell) := \{h\in \Pol^k[\zeta_1, \dots, \zeta_\ell] : \Delta h = 0\},$$ that is, the space of harmonic polynomials on $\C^\ell$ of degree $k$. When $\ell = 2$, this space has dimension at most $2$ and can be expressed as follows (see, for example \cite[Thm. 3.1]{helgason}):
\begin{equation}\label{decomposition-Harm-2}
\Harm^k(\C^2) = 
\begin{cases}
\C, & \text{if } k = 0,\\
\C(\zeta_1 - i\zeta_2)^k + \C(\zeta_1 + i\zeta_2)^k, &\text{if } k \geq 0.
\end{cases}
\end{equation}

\begin{lemma}\label{lemma-Harm} Let $k \in \N$ and $m \in \Z\setminus\{0\}$. Then, the following three conditions on the pair $(k,m)$ are equivalent.
\begin{enumerate}[label=\normalfont{(\roman*)}, topsep=0pt]
\item $\Hom_{SO(2)}\left(V^3, \C_m\otimes \Harm^k(\C^2)\right) \neq \{0\}$.
\item $\dim_\C \Hom_{SO(2)}\left(V^3, \C_m\otimes \Harm^k(\C^2)\right) = 1$.
\item $k \in K(m) := \{|m-1|, |m|, |m+1|\}$.
\end{enumerate}
Moreover, if one of the above (therefore all) conditions is satisfied, the generators $h_k^\pm$ of the space $\Hom_{SO(2)}\left(V^3, \C_m\otimes \Harm^k(\C^2)\right)$ are given in \text{{\normalfont{(\ref{generators-Hom(V3,CmHk)})}}} below.
\end{lemma}

\begin{equation}\label{generators-Hom(V3,CmHk)}
\begin{array}{l l}
h_{m-1}^+: V^3 \longrightarrow \C_m \otimes \Harm^{m-1}(\C^2), &\enspace
u_1 \longmapsto 1 \otimes(\zeta_1 + i\zeta_2)^{m-1},\\[7pt]
h_{m}^+: V^3 \longrightarrow \C_m \otimes \Harm^{m}(\C^2), &\enspace
u_2 \longmapsto 1 \otimes (\zeta_1 + i\zeta_2)^{m},\\[7pt]
h_{m+1}^+: V^3 \longrightarrow \C_m \otimes \Harm^{m+1}(\C^2), &\enspace
u_3 \longmapsto 1 \otimes (\zeta_1 + i\zeta_2)^{m+1},\\[7pt]
h_{-m+1}^-: V^3 \longrightarrow \C_m \otimes \Harm^{-m+1}(\C^2), &\enspace
u_3 \longmapsto 1 \otimes(\zeta_1 - i\zeta_2)^{-m+1},\\[7pt]
h_{-m}^-: V^3 \longrightarrow \C_m \otimes \Harm^{-m}(\C^2), &\enspace
u_2 \longmapsto 1 \otimes (\zeta_1 - i\zeta_2)^{-m},\\[7pt]
h_{-m-1}^-: V^3 \longrightarrow \C_m \otimes \Harm^{-m-1}(\C^2), &\enspace
u_1 \longmapsto 1 \otimes (\zeta_1 - i\zeta_2)^{-m-1}.
\end{array}
\end{equation}
The sign index $\pm$ in $h_k^\pm$ is chosen to coincide with that of $m$, being $+$ if $m \geq 1$ and $-$ if $m\leq -1$.
\begin{proof}
We set $g := \begin{pmatrix}
\cos t & -\sin t & 0\\
\sin t & \cos t & 0\\
0 & 0 & 1 
\end{pmatrix} \in SO(2)\hookrightarrow SO(3)$ (see (\ref{realization-SO(2)inSO(3)})). It follows from (\ref{expression-covering}) that $\varpi(U(e^{-it/2},0)) = g$. Thus, for $p(\xi_1, \xi_2) = \xi_1^{2-j}\xi_2^j$ $(j = 0,1,2)$, we have 
\begin{equation*}
\sigma^3(g)p(\xi_1, \xi_2) = e^{i(1-j)t}p(\xi_1, \xi_2).
\end{equation*}
Hence,
\begin{equation*}
V^3 = \C u_1 \oplus \C u_2 \oplus \C u_3 \simeq \C_{1} \oplus \C_0 \oplus \C_{-1} \quad \text{ as } SO(2)\text{-modules.}
\end{equation*}

On the other hand, the decomposition (\ref{decomposition-Harm-2}) is precisely the decomposition of $\Harm^k(\C^2)$ as an $SO(2)$-module:
\begin{equation*}
\C(\zeta_1 - i\zeta_2)^k + \C(\zeta_1 + i\zeta_2)^k \simeq \C_{k} + \C_{-k}.
\end{equation*}
Thus, we have
\begin{equation*}
\C_m \otimes \Harm^k(\C^2) \simeq \begin{cases}
\C_m, &\text{if } k = 0,\\
\C_{m+k} \oplus \C_{m-k},&\text{if } k > 0.
\end{cases}
\end{equation*}

Now, the lemma follows by taking $SO(2)$-invariants of the tensor product $(V^3)^\vee \otimes \C_m \otimes \Harm^k(\C^2)$. The maps $h_k^\pm$ can be obtained by looking at the factors that do not vanish in the decompositions above.
\end{proof}

Given $b\in \Z$ and $g \in \C[t]$, we define a meromorphic function $(T_bg)(\zeta)$ of three variables $\zeta = (\zeta_1, \zeta_2,\zeta_3)$ as follows (cf. {\cite[(4.4)]{kob-pev1}}):
\begin{equation}\label{Def-T_b}
(T_bg)(\zeta) := (\zeta_1^2+\zeta_2^2)^\frac{b}{2} g\left(\frac{\zeta_3}{\sqrt{\zeta_1^2+\zeta_2^2}}\right).
\end{equation}

We observe that $T_bg$ is a homogeneous polynomial on three variables of degree $b$ if $b \in \N$ and $g \in \Pol_b[t]_\text{{\normalfont{even}}}$, where
\begin{equation}\label{def-Pol_even}
\Pol_b[t]_\text{{\normalfont{even}}}:= \spanned_\C\left\{t^{b-2j}: j = 0, \dots, \left[\frac{b}{2}\right]\right\}.
\end{equation}
Moreover, we obtain the following bijection:
\begin{equation}\label{bijection-T_b}
T_b: \Pol_b[t]_\text{{\normalfont{even}}} \arrowsimeq \bigoplus_{2b_1+b_2 = b} \Pol^{b_1}[\zeta_1^2 + \zeta_2^2]\otimes\Pol^{b_2}[\zeta_3].
\end{equation}

We recall from \cite{kkp} the following useful lemma. 

\begin{lemma}[{\cite[Lem. 4.2]{kkp}}]\label{Harm-Pol-isomorphism} For any $a \in \N$, we have the following bijection.
\begin{equation*}
\begin{aligned}
\bigoplus_{k = 0}^a\Pol_{a-k}[t]_\text{{\normalfont{even}}} \otimes \Hom_{SO(2)}\left(V^3, \C_m\otimes \Harm^k(\C^2)\right) &\arrowsimeq \Hom_{SO(2)}\left(V^3, \C_m \otimes \Pol^a[\zeta_1, \zeta_2, \zeta_3]\right),\\
\sum_{k=0}^a g_k \otimes H_k &\longmapsto \sum_{k = 0}^a \left(T_{a-k}g_k\right)H_k.
\end{aligned}
\end{equation*}
\end{lemma}
This result comes from combining (\ref{bijection-T_b}) with the natural decomposition
\begin{equation*}
\begin{gathered}
\Pol[\zeta_1^2 + \cdots + \zeta_\ell^2] \otimes \Harm(\C^\ell) \arrowsimeq \Pol[\zeta_1, \ldots, \zeta_\ell],
\end{gathered}
\end{equation*}
where $\Harm(\C^\ell) := \bigoplus_{k=0}^\infty \Harm^k(\C^\ell)$.
\begin{prop}\label{Firststep-prop} Let $\lambda, \nu \in \C$, $m \in \Z\setminus\{0\}$ and $a \in \N$. Then, the following two conditions on the tuple $(\lambda, \nu, m, a)$ are equivalent.
\begin{enumerate}[label=\normalfont{(\roman*)}, topsep=0pt]
\item $\Hom_{L^\prime}\left(V^3_\lambda, \C_{m, \nu} \otimes \Pol^a(\n_+)\right) \neq \{0\}$.
\item $a = \nu - \lambda$ and $a \geq \min K(m)$.
\end{enumerate}
Moreover, if {\normalfont{(ii)}} is satisfied, the space $\Hom_{L^\prime}\left(V^3_\lambda, \C_{m, \nu} \otimes \Pol^a(\n_+)\right)$ is equal to
\begin{equation*}
\begin{gathered}
\spanned_\C\{\left(T_{a-k}g_k\right)h_k^{\eta(m)}: g_k \in \Pol_{a-k}[t]_\text{{\normalfont{even}}}, k\in K(m)\},
\end{gathered}
\end{equation*}
where $\eta(m) = +$ \text{{\normalfont{(if $m \geq 1$)}}} and \text{{\normalfont{$\eta(m) = -$ (if $m \leq -1$)}}}. Here, we regard $\Pol_{d}[t]_\text{{\normalfont{even}}} = \{0\}$ for $d < 0$.
\end{prop}

\begin{proof}
First, we consider the action of the last factor of $L^\prime \simeq SO(2) \times A$. By (\ref{action-SO(2)-on-Pol}), this action on $\Pol^a(\n_+)$ is given by:
\begin{equation*}
A \times \Pol^a(\n_+) \rightarrow \Pol^a(\n_+), \enspace (\exp(tH_0), p) \mapsto e^{-ta}p.
\end{equation*}
Therefore, (i) is satisfied if and only if $e^{\lambda t} = e^{(\nu - a)t}$ for all $t\in \R$, or equivalently, if and only if $a = \nu-\lambda$. Then, the proof of the proposition reduces to the action of $SO(2)$, in which case the result is clear from Lemmas \ref{lemma-Harm} and \ref{Harm-Pol-isomorphism}.
\end{proof}

\section{Step 2: Solving the differential equation (\ref{F-system-2-concrete})}\label{section-step2}
In this section we consider the second step of the F-method for the case $m \in \N_+$. The case $m \in -\N_+$ will be treated in Section \ref{section-case_m_lessthan_-1}.

\subsection{Description of $\Sol(\n_+; \sigma_\lambda^3, \tau_{m, \nu})$ and proof of Theorems \ref{mainthm1} and \ref{mainthm2} for $m \geq 1$}\label{section-proof-of-main-thms}
The aim of Step 2 is to determine the space $\Sol(\n_+; \sigma_\lambda^3, \tau_{m, \nu})$. By Proposition \ref{Firststep-prop}, this is to solve the system of differential equations
\begin{equation*}
\left(\widehat{d\pi_\mu}(C)\otimes \id_{\C_{m, \nu}}\right)\psi = 0 \quad (\forall C\in \n_+^\prime),
\end{equation*}
where  $\psi = \sum_{k \in K(m)}\left(T_{\nu-\lambda-k}g_k\right)h_k^{\eta(m)}, \text{ for } g_k \in \Pol_{\nu-\lambda-k}[t]_\text{{\normalfont{even}}}$.
Theorem \ref{Thm-step2} below gives a complete description of $\Sol(\n_+; \sigma_\lambda^3, \tau_{m, \nu})$ for $m \in \N_+$.

\begin{thm}\label{Thm-step2} The following three conditions on the triple $(\lambda, \nu, m) \in \C^2\times \N_+$ are equivalent.
\begin{enumerate}[label=\normalfont{(\roman*)}, topsep=0pt]
\item $\Sol(\n_+; \sigma_\lambda^3, \tau_{m, \nu}) \neq \{0\}$.
\item $\dim_\C \Sol(\n_+; \sigma_\lambda^3, \tau_{m, \nu}) = 1.$
\item The triple $(\lambda, \nu, m)$ belongs to one of the following cases.
\begin{itemize}[leftmargin=1.5cm]
\item[\normalfont{Case 1.}] $m = 1$ and $\nu - \lambda \in \N$.
\item[\normalfont{Case 2.}] $m > 1$, $\lambda \in \Z_{\leq 1-m}$, and $\nu \in \{0, 1, 2\}$.
\end{itemize}
\end{enumerate}
Moreover, if one of the above (therefore all) conditions is satisfied, then
\begin{equation*}
\Sol(\n_+; \sigma_\lambda^3, \tau_{m, \nu}) = \C\sum_{k = m-1}^{m+1}\left(T_{a -k}g_k\right)h_k^+,
\end{equation*}
where $a := \nu - \lambda$ and the polynomials $g_{m-j}(t) \enspace (j = -1,0,1)$ are given as follows:
\begin{equation}\label{solution-all}
\begin{array}{ll}
\bullet \text{ If }a=m-1: &(g_{m-1}(t), g_m(t), g_{m+1}(t)) = (1, 0, 0).\\[7pt]
\bullet \text{ If } m = 1 \text{ and } a>0: &(g_{0}(t), g_1(t), g_{2}(t)) = (\ref{solution-m=1}).\\[7pt]
\bullet \text{ If } m > 1 \text{ and } a>m-1: &(g_{m-1}(t), g_m(t), g_{m+1}(t)) = (\ref{solution-m>1}).
\end{array}
\end{equation}
\begin{equation}\label{solution-m=1}
\left(-\left(\lambda + \left[\frac{a-1}{2}\right]\right)\Geg_{a}^\lambda(it), -i\gamma(\lambda-1, a)\Geg_{a-1}^\lambda(it), \Geg_{a-2}^\lambda(it)\right)
\end{equation}
\begin{equation}\label{solution-m>1}
\left((-1)^{\nu+1}iAB \widetilde{C}^\lambda_{a-m+1-2\nu}(it), -C\Geg_{a-m}^{\lambda-1}(it) + iBt \Geg_{a-m-1}^\lambda(it), iB \widetilde{C}^\lambda_{a-m-1}(it)\right)
\end{equation}
Here, $\widetilde{C}_\ell^\mu(z)$ stands for the renormalized Gegenbauer polynomial (see {\normalfont{(\ref{Gegenbauer-polynomial(renormalized)})}}), and we regard $\widetilde{C}_d^\mu = 0$ for $d < 0$. The constants $A, B, C$ are defined as in {\normalfont{(\ref{const-A}})}, {\normalfont{(\ref{const-B}})} and {\normalfont{(\ref{const-C}})}.
\end{thm}
We shall discuss the proof of Theorem \ref{Thm-step2} in Section \ref{section-reductiontheorems}. By using Theorem \ref{Thm-step2} we first give a proof of Theorems \ref{mainthm1} and \ref{mainthm2} for $m\in \N_+$ below.

\begin{proof}[Proof of Theorem \ref{mainthm1} for $m \geq 1$]
Recall from (\ref{isomorphism-Fmethod-concrete}) that there exists a linear isomorphism
\begin{equation*}
\begin{tikzcd}[column sep = 1cm]
\Diff_{G^\prime}\left(C^\infty(S^3, \V_\lambda^3), C^\infty(S^2, \L_{m,\nu})\right) \arrow[r, "\displaystyle{_{\Symb \otimes \id}}", "\thicksim"'] & \Sol(\n_+; \sigma_\lambda^3, \tau_{m, \nu}).
\end{tikzcd}
\end{equation*}
Now, Theorem \ref{mainthm1} follows from Theorem \ref{Thm-step2}.
\end{proof}

\begin{proof}[Proof of Theorem \ref{mainthm2} for $m \geq 1$]
By (\ref{isomorphism-Fmethod-concrete}) and Theorem \ref{Thm-step2}, it suffices to compute the inverse of the symbol map $\Symb \otimes \id$ of $\psi = \sum_{k = m-1}^{m+1}\left(T_{a-k}g_k\right)h_k^+$, where $(g_{m-1}, g_m, g_{m+1})$ is the triple given by (\ref{solution-all}).

We suppose $a \geq m$ and $m>1$. The other two cases are simpler and can be done in a similar fashion. By the definition of $T_b$ (see (\ref{Def-T_b})), a direct computation shows that 
\begin{equation}
\begin{array}{ll}
\left(T_{a-m+1}g_{m-1}\right)(\zeta_1, \zeta_2, \zeta_3) &= (-1)^\frac{a-m}{2}AB \left(\zeta_1^2+\zeta_2^2\right)^\nu \left(I_{a-m+1-2\nu}\Geg_{a-m+1-2\nu}^\lambda\right)\left(-(\zeta_1^2 + \zeta_2^2), \zeta_3\right),\\[6pt]
\left(T_{a-m}g_{m}\right)(\zeta_1, \zeta_2, \zeta_3) &= (-1)^\frac{a-m}{2}\Big(-C\left(I_{a-m}\Geg_{a-m}^{\lambda-1}\right)\left(-(\zeta_1^2 + \zeta_2^2), \zeta_3\right)\\[6pt]
&\hspace*{\fill}+ B \zeta_3\left(I_{a-m-1}\Geg_{a-m-1}^{\lambda}\right)\left(-(\zeta_1^2 + \zeta_2^2), \zeta_3\right)\Big),\\[6pt]
\left(T_{a-m-1}g_{m+1}\right)(\zeta_1, \zeta_2, \zeta_3) &= (-1)^\frac{a-m}{2}B \left(I_{a-m-1}\Geg_{a-m-1}^\lambda\right)\left(-(\zeta_1^2 + \zeta_2^2), \zeta_3\right),
\end{array}
\end{equation}
where $(I_\ell\Geg_\ell^\mu)(x,y) := x^{\frac{\ell}{2}}\Geg_\ell^\mu\left(\frac{y}{\sqrt{x}}\right)$. Thus, by (\ref{generators-Hom(V3,CmHk)}) and dividing each coordinate by $(-1)^\frac{a-m}{2}$, we obtain that the generator of (\ref{DSBO-space}) is
\begin{multline*}
AB \left(\Delta_{\R^2}\right)^\nu\widetilde{\C}_{\lambda+1, 2-\nu-m}\left(\frac{\partial}{\partial x_1} + i\frac{\partial}{\partial x_2} \right)^{m-1} \otimes u_1^\vee\\
+ \left(-C\widetilde{\C}_{\lambda, \nu-m} + B\frac{\partial}{\partial x_3}\widetilde{\C}_{\lambda+1, \nu -m} \right)\left(\frac{\partial}{\partial x_1} + i\frac{\partial}{\partial x_2} \right)^m \otimes u_2^\vee\\
+ B\widetilde{\C}_{\lambda+1, \nu-m} \left(\frac{\partial}{\partial x_1} + i\frac{\partial}{\partial x_2} \right)^{m+1}  \otimes u_3^\vee,
\end{multline*}
where
\begin{equation*}
\begin{gathered}
\widetilde{\C}_{\lambda, \nu} := \Rest_{x_3 = 0} \circ \left(I_{\nu - \lambda} \widetilde{C}^{\lambda-1}_{\nu-\lambda}\right)\left(-\Delta_{\R^2}, \frac{\partial}{\partial x_3}\right).
\end{gathered}
\end{equation*}
If we multiply by $2^{-m}$ and use the identities below that come from the change of variables $z = x_1 + ix_2$ in the generator above, we obtain the expression (\ref{Operator-general+}) of the operator $\D_{\lambda, \nu}^m$.
\begin{equation*}
2 \frac{\partial}{\partial \overline{z}} = \left(\frac{\partial}{\partial x_1} + i\frac{\partial}{\partial x_2}\right), \quad 4\frac{\partial^2}{\partial z \partial \overline{z}} = \Delta_{\R^2}.
\end{equation*}
\end{proof}

\subsection{Reduction Theorems}\label{section-reductiontheorems}
The proof of Theorem \ref{Thm-step2} will be separated into to reduction theorems: Theorem \ref{Thm-findingequations} (finding equations), and Theorem \ref{Thm-solvingequations} (solving equations). We state them in this subsection and give a proof in the following sections. From now on, we assume $a = \nu - \lambda\in \N$ and $m \in \N_+$.

For any $\mu \in \C$ and any $\ell \in \N$, the \textit{imaginary Gegenbauer differential operator} $S_\ell^\mu$ is defined as follows (see (\ref{Gegen-imaginary})) (cf. \cite[(6.5)]{kkp}):
\begin{equation*}
S_\ell^\mu = - \left((1 + t^2) \frac{d^2}{dt^2} + (1 + 2\mu)t\frac{d}{dt} - \ell(\ell + 2\mu)\right).
\end{equation*}
Now, for $f_0, f_1, f_2 \in \C[t]$ let $L_r(f_0, f_1, f_2)(t)$ $(r = 1, \dots, 6)$ be the following polynomials:
\begin{equation}\label{expression-operators-Lr}
\begin{array}{l}
L_1(f_0, f_1, f_2)(t) := S_{a+m-1}^{\lambda} f_0.\vspace{0.2cm}\\

L_2(f_0, f_1, f_2)(t) := S_{a-m-1}^{\lambda} f_2.\vspace{0.2cm}\\

L_3(f_0, f_1, f_2)(t) := S_{a+m}^{\lambda - 1}f_1 - 2\frac{d}{dt}f_0.\vspace{0.2cm}\\

L_4(f_0, f_1, f_2)(t) := S_{a-m}^{\lambda - 1}f_1 + 2\frac{d}{dt}f_2.\vspace{0.2cm}\\

L_5(f_0, f_1, f_2)(t) := (-m(\lambda + a-1) + \lambda -1 + \vartheta_t) f_0 + \frac{d}{dt}f_1.\vspace{0.2cm}\\

L_6(f_0, f_1, f_2)(t) :=(m(\lambda + a-1) + \lambda -1 + \vartheta_t) f_2 - \frac{d}{dt}f_1.
\end{array}
\end{equation}
Here, $\vartheta_t$ stands for the Euler operator $\vartheta_t := t\frac{d}{dt}$.
 
\begin{thm}[\textbf{Finding equations}]\label{Thm-findingequations} Let $m \in \N_+$ and set $\psi = \sum_{k = m-1}^{m+1}\left(T_{a-k}g_k\right)h_k^+$. Then, the following two conditions on $(g_{m-1}, g_m, g_{m+1})$ are equivalent.
\begin{enumerate}[label=\normalfont{(\roman*)}, topsep=0pt]
\item $\left(\widehat{d\pi_\mu}(C)\otimes \id_{\C_{m,\nu}}\right)\psi = 0$, for all $C \in \n_+^\prime$.
\item $L_r(g_{m-1}, g_m, g_{m+1}) = 0$, for all $r = 1, \dots, 6$.
\end{enumerate}
\end{thm}
To state the other reduction theorem, let
\begin{equation}\label{equation-space}
\Xi_{\lambda, a, m} := \left\{(g_{m-1}, g_m, g_{m+1}) \in \bigoplus_{k = m-1}^{m+1} \Pol_{a-k}[t]_\text{{\normalfont{even}}} : \begin{array}{l}
L_r(g_{m-1}, g_m, g_{m+1}) = 0\vspace{0.2cm}\\
\forall r= 1, \dots, 6
\end{array}
\right\}.
\end{equation}
\begin{thm}[\textbf{Solving equations}]\label{Thm-solvingequations} The following three conditions on the triple $(\lambda, a, m) \in \C\times\N \times \N_+$ are equivalent.
\begin{enumerate}[label=\normalfont{(\roman*)}, topsep=0pt]
\item $\Xi_{\lambda, a, m} \neq \{0\}$.
\item $\dim_\C \Xi_{\lambda, a, m}= 1$.
\item The triple $(\lambda, a, m)$ belongs to one of the following cases.
\begin{enumerate}[label=\normalfont{(\arabic*)}, leftmargin=1cm]
\item If $a = m-1$\\
\text{{\normalfont{(1.I)}}} $m = 1$ and $\lambda \in \C$, or \text{{\normalfont{(1.II)}}} $m > 1$ and $\lambda = -a$.
\item If $a = m$\\
\text{{\normalfont{(2.I)}}} $m = 1$ and $\lambda \in \C$, or \text{{\normalfont{(2.II)}}} $m > 1$ and $\lambda \in \{-a, 1-a\}$.
\item If $a \geq m+1$\\
\text{{\normalfont{(3.I)}}} $m = 1$ and $\lambda \in \C$, or \text{{\normalfont{(3.II)}}} $m > 1$ and $\lambda \in \{-a, 1-a, 2-a\}$.
\end{enumerate}
\end{enumerate}

Moreover, if \text{{\normalfont{(iii)}}} is satisfied, the generator of $\Xi_{\lambda, a, m} $ is given by \text{{\normalfont{(\ref{solution-all})}}}.
\end{thm}

Theorems \ref{Thm-findingequations} and \ref{Thm-solvingequations} will be proved in Sections \ref{section-proof_of_finding_equations} and \ref{section-proof_of_solving_equations} respectively.

\section{Proof of Theorem \ref{Thm-findingequations}}\label{section-proof_of_finding_equations}
In this section we prove Theorem \ref{Thm-findingequations}. This is done in particular in Section \ref{section-proof_of_finding_equations-subsection}. We start with the decomposition of the differential operator $\widehat{d\pi_\mu}(C_1^+)\otimes \id_{\C_{m,\nu}}$.

\subsection{Decomposition of the equation $\left(\widehat{d\pi_\mu}(C_1^+)\otimes \id_{\C_{m,\nu}}\right)\psi=0$}

Following the same strategy as in \cite{kkp}, we decompose the equation
\begin{equation*}
\left(\widehat{d\pi_\mu}(C)\otimes \id_{\C_{m,\nu}}\right)\psi = 0 \quad (\forall C\in \n_+^\prime),
\end{equation*}
into a \textit{scalar} part (differential operator of second order), and a \textit{vector} part (differential operator of first order). With a little abuse of notation, we regard $\lambda$ as the one-dimensional representation $\C_\lambda$ of $A$ defined by exponentiation.

In the following, we consider the basis $\{C_j^-: 1 \leq j  \leq 3\}$ of $\n_-$ and write $(z_1, z_2, z_3)$ for the corresponding  coordinates. We identify $\n_-^\vee \simeq \n_+$ and write $(\zeta_1, \zeta_2, \zeta_3)$ for the dual coordinates.

For $Y \in \g$, we define $d\pi_{(\sigma^3, \lambda)^*}^\text{{\normalfont{scalar}}}(Y), d\pi_{(\sigma^3, \lambda)^*}^\text{{\normalfont{vect}}}(Y) \in \Weyl(\n_-) \otimes \End\left((V^3)^\vee\right)$ by
\begin{equation*}
\begin{aligned}
\left(d\pi_{(\sigma^3, \lambda)^*}^\text{{\normalfont{scalar}}}(Y)\right)F(Z)
& = d\lambda^*(\alpha(Y,Z)\big\rvert_\a)F(Z) - \left(\beta(Y, \cdot)F\right)(Z), \\
\left(d\pi_{(\sigma^3, \lambda)^*}^\text{{\normalfont{vect}}}(Y)\right)F(Z)  &= d(\sigma^3)^\vee(\alpha(Y,Z)\big\rvert_\m),
\end{aligned}
\end{equation*}
for $F\in C^\infty\left(\n_-, (V^3)^\vee\right)$, in such a way that $d\pi_\mu \equiv d\pi_{(\sigma^3, \lambda)^*}$ is given by
\begin{equation*}
d\pi_{(\sigma^3, \lambda)^*} = d\pi_{(\sigma^3, \lambda)^*}^\text{{\normalfont{scalar}}} + d\pi_{(\sigma^3, \lambda)^*}^\text{{\normalfont{vect}}},
\end{equation*}
(see (\ref{dpimu-general_expression})). We call $d\pi_{(\sigma^3, \lambda)^*}^\text{{\normalfont{scalar}}}$ and $d\pi_{(\sigma^3, \lambda)^*}^\text{{\normalfont{vect}}}$ the \emph{scalar part} and the \emph{vector part} of $d\pi_{(\sigma^3, \lambda)^*}$ respectively. We have
\begin{equation*}
\begin{aligned}
d\pi_{(\sigma^3, \lambda)^*}^\text{{\normalfont{scalar}}}(Y)
& = d\pi_{\lambda^*}(Y) \otimes \id_{(V^3)^\vee},\\
d\pi_{(\sigma^3, \lambda)^*}^\text{{\normalfont{vect}}}(Y) & =  -\sum_{\ell=1}^3 z_\ell F \circ d\sigma^3\left([Y, C_\ell^-]\big\rvert_\m\right).
\end{aligned}
\end{equation*}
Therefore, by taking the algebraic Fourier transform, $\widehat{d\pi_{(\sigma^3, \lambda)^*}} \otimes \id_{\C_{m, \nu}}$ is of the form
\begin{equation}\label{dpimu-decomposition}
\begin{gathered}
\widehat{d\pi_{(\sigma^3, \lambda)^*}}\otimes \id_{\C_{m,\nu}} = \widehat{d\pi_{\lambda^*}} \otimes \id_{\Hom(V^3, \C_{m,\nu})} + \widehat{d\pi_{(\sigma^3, \lambda)^*}^\text{{\normalfont{vect}}}} \otimes \id_{\C_{m, \nu}},
\end{gathered}
\end{equation}
with $\widehat{d\pi_{(\sigma^3, \lambda)^*}^\text{{\normalfont{vect}}}}(Y) :=  -\sum_{\ell=1}^3 \frac{\partial}{\partial \zeta_\ell} \circ d\sigma^3\left([Y, C_\ell^-]\big\rvert_\m\right)$ 
(see \cite[Prop. 3.5]{kkp}).

Since the Levi subgroup $L^\prime$ acts irreducibly on
the nilradical $\n_+^\prime(\R)$ of $\p^\prime(\R) = \l^\prime(\R) + \n_+^\prime(\R)$, the equation (\ref{F-system-2-concrete}) is satisfied for all $C \in \n_+^\prime$ if and only if it is satisfied for some non-zero $C_0 \in \n_+^\prime$. (For a detailed proof of this fact see \cite[Lem. 3.4]{kkp}.)

In our case, we take $C_0 = C_1^+$ (see (\ref{elements})).
Thus, by (\ref{dpimu-decomposition}), we solve the differential equation
\begin{equation}\label{dpimu-equation}
\left(\widehat{d\pi_{(\sigma^3, \lambda)^*}^\text{{\normalfont{scalar}}}}(C_1^+)\otimes \id_{\C_{m,\nu}}\right)\psi + \left(\widehat{d\pi_{(\sigma^3, \lambda)^*}^\text{{\normalfont{vect}}}}(C_1^+)\otimes \id_{\C_{m,\nu}}\right)\psi= 0,
\end{equation}
where
\begin{equation}\label{dpimu-scalar-vect-formulas}
\begin{aligned}
\widehat{d\pi_{(\sigma^3, \lambda)^*}^\text{{\normalfont{scalar}}}}(C_1^+) &=
\widehat{d\pi_{\lambda^*}}(C_1^+) \otimes \id_{\left(V^3\right)^\vee},\\
\widehat{d\pi_{(\sigma^3, \lambda)^*}^\text{{\normalfont{vect}}}}(C_1^+) &= \sum_{\ell=1}^3 d\sigma^3\left(2X_{\ell, 1}\right)\frac{\partial}{\partial \zeta_\ell},
\end{aligned}
\end{equation}
as $[C_1^+, C_\ell^-] = 2X_{1, \ell}$ (see (\ref{brackets-Cpm})).

For actual computations in the next sections, it is convenient to rewrite the scalar part $\widehat{d\pi_{(\sigma^3, \lambda)^*}^\text{{\normalfont{scalar}}}}$ and the vector part $\widehat{d\pi_{(\sigma^3, \lambda)^*}^\text{{\normalfont{vect}}}}$ in terms of their \textit{vector coefficients} $M_s^\text{{\normalfont{scalar}}}$ and $M_s^\text{{\normalfont{vect}}}$ with respect to the basis $\{u_1, u_2, u_3\}$ of $V^3$ in (\ref{basis-V3}); that is, we write
\begin{equation*}
\begin{aligned}
\left(\widehat{d\pi_{(\sigma^3, \lambda)^*}^\text{{\normalfont{scalar}}}}(C_1^+)\otimes \id_{\C_{m,\nu}}\right)\psi &= \sum_{s=1}^3 M_s^\text{{\normalfont{scalar}}}(\psi)u_s^\vee,\\
\left(\widehat{d\pi_{(\sigma^3, \lambda)^*}^\text{{\normalfont{vect}}}}(C_1^+)\otimes \id_{\C_{m,\nu}}\right)\psi &= \sum_{s=1}^3 M_s^\text{{\normalfont{vect}}}(\psi)u_s^\vee.
\end{aligned}
\end{equation*}
Then, the differential equation (\ref{dpimu-equation}) is equivalent to
\begin{equation}\label{Ms-equal-zero}
M_s(\psi) = M_s^\text{{\normalfont{scalar}}}(\psi) + M_s^\text{{\normalfont{vect}}}(\psi) = 0, \enspace \text{ for } s = 1,2,3.
\end{equation}

\subsection{Vector coefficients $M_s^\text{{\normalfont{scalar}}}(\psi)$ and $M_s^\text{{\normalfont{vect}}}(\psi)$ of $\left(\widehat{d\pi_{(\sigma^3, \lambda)^*}^\text{{\normalfont{scalar}}}}(C_1^+)\otimes \id_{\C_{m,\nu}}\right)\psi$}
Take $\psi = \sum_{s=1}^3 \psi_su_s^\vee \in \Hom_\C\left(V_\lambda^3, \C_{m, \nu} \otimes \Pol^a(\n_+)\right)$. Then, the vector coefficients of the scalar part $M_s^\text{scalar}(\psi)$ are clearly given by
\begin{equation}\label{vector-coefficients-scalar}
M_s^\text{{\normalfont{scalar}}}(\psi) = \widehat{d\pi_{\lambda^*}^\text{{\normalfont{scalar}}}}(C_1^+)\psi_s \quad (s = 1,2,3).
\end{equation}
For the vector part $M_s^\text{vect}(\psi)$, we define
\begin{equation*}
A_{ss^\prime} := \sum_{\ell = 1}^3 (d\sigma^3(2X_{\ell,1})_{ss^\prime})\frac{\partial}{\partial \zeta_\ell} \quad (s, s^\prime = 1,2,3).
\end{equation*}
Then, by (\ref{dpimu-scalar-vect-formulas}), we have
\begin{equation}\label{vector-coefficients-vect}
M_s^\text{{\normalfont{vect}}}(\psi) = \sum_{s^\prime = 1}^3 A_{ss^\prime}\psi_{s^\prime}\quad (s = 1,2,3).
\end{equation}
More concretely, we have the following.
\begin{lemma} \label{lemma-coefficients-Ms-scalar-vector} For $\psi = \sum_{s=1}^3 \psi_su_s^\vee \in \Hom_\C\left(V_\lambda^3, \C_{m, \nu} \otimes \Pol^a(\n_+)\right)$, we have
\begin{equation}\label{Ms-scalar}
M_s^\text{{\normalfont{scalar}}}(\psi) = \left(2\lambda \frac{\partial}{\partial \zeta_1} + 2E_{\zeta}\frac{\partial}{\partial \zeta_1} - \zeta
_1 \Delta^\zeta_{\C^3}\right)\psi_s \quad (s = 1,2,3),
\end{equation}
\begin{equation}\label{Ms-vect}
\begin{pmatrix}
M_1^\text{{\normalfont{vect}}}(\psi)\\[4pt]
M_2^\text{{\normalfont{vect}}}(\psi)\\[4pt]
M_3^\text{{\normalfont{vect}}}(\psi)
\end{pmatrix}
= \begin{pmatrix}
-2i \frac{\partial}{\partial \zeta_2} & 2 \frac{\partial}{\partial \zeta_3}& 0\\[6pt]
 -\frac{\partial}{\partial \zeta_3} & 0 & \frac{\partial}{\partial \zeta_3}\\[6pt]
 0 & -2 \frac{\partial}{\partial \zeta_3} & 2i \frac{\partial}{\partial \zeta_2}
\end{pmatrix}
\begin{pmatrix}
\psi_1\\
\psi_2\\
\psi_3
\end{pmatrix}.
\end{equation}
\end{lemma}
\begin{proof}
The vector coefficients $M_s^\text{{\normalfont{scalar}}}$ come from the fact that
\begin{equation*}
\widehat{d\pi_{\lambda^*}}(C_1^+) = 2\lambda \frac{\partial}{\partial \zeta_1} + 2E_{\zeta}\frac{\partial}{\partial \zeta_1} - \zeta
_1 \Delta^\zeta_{\C^3},
\end{equation*}
which was proved in \cite[Lem. 6.5]{kob-pev2}. For the vector part $M_s^\text{{\normalfont{vect}}}$, first we need to compute the differential of $\sigma^3$ on the elements $2X_{\ell,1}$, $(\ell = 1,2,3)$. To that purpose, we make use of the isomorphism $\m \simeq \so(3,\C)$ and identify each element $X_{\ell,1}$ with $E_{1,\ell} - E_{\ell, 1} \in \so(3,\C)$ for $\ell = 1,2,3$. By a direct computation, one can verity that for $t \in \R$ one has
\begin{equation*}
\begin{aligned}
e^{tX_{1,1}} &= I_3 = \varpi(I_2),\\[4pt]
e^{tX_{2,1}} & = 
\begin{pmatrix}
\cos t & \sin t & 0\\
-\sin t & \cos t & 0\\
0 & 0 & 1
\end{pmatrix} = \varpi(U(e^{it/2},0)),\\[4pt]
e^{tX_{3,1}} & = 
\begin{pmatrix}
\cos t & 0 & \sin t\\
0 & 1 & 0\\[2pt]
-\sin t & 0 & \cos t
\end{pmatrix} = \varpi(U(\cos(t/2), -\sin(t/2)),
\end{aligned}
\end{equation*}
where $\varpi: SU(2) \twoheadrightarrow SO(3)$ is the covering map defined in (\ref{expression-covering}). Thus, we obtain
\renewcommand{\arraystretch}{0.8}
\begin{equation}
d\sigma^3(2X_{1,1}) = 0, \enspace d\sigma^3(2X_{2,1}) =
\begin{pmatrix}
-2i & 0 & 0 \\
0 & 0 & 0\\
0 & 0 & 2i
\end{pmatrix}, \enspace
d\sigma^3(2X_{3,1}) = \begin{pmatrix}
0 & 2 & 0\\
-1 & 0 & 1\\
0 & -2 & 0
\end{pmatrix}.
\end{equation}
Hence, the matrix $(A_{ss^\prime})$ is given by
\begin{equation*}
(A_{ss^\prime}) = \begin{pmatrix}
-2i \frac{\partial}{\partial \zeta_2} & 2 \frac{\partial}{\partial \zeta_3}& 0\\[6pt]
 -\frac{\partial}{\partial \zeta_3} & 0 & \frac{\partial}{\partial \zeta_3}\\[6pt]
 0 & -2 \frac{\partial}{\partial \zeta_3} & 2i \frac{\partial}{\partial \zeta_2}
\end{pmatrix}.
\end{equation*}
\end{proof}

\subsection{Explicit formul\ae{} for $M_s(\psi)$}
Recall from Lemma \ref{lemma-Harm} and Proposition \ref{Firststep-prop} that $\psi \in \Hom_{L^\prime}\left(V^3_\lambda, \C_{m, \nu} \otimes \Pol^a(\n_+)\right)$ is of the form
\begin{equation}\label{expression-psi}
\psi= \sum_{k=m-1}^{m+1}\left(T_{a-k} g_k\right)h_k^+ = \sum_{k=m-1}^{m+1}\left(T_{a-k} g_k\right)(\zeta_1 + i\zeta_2)^k u_{k-m+2}.
\end{equation}
We next give explicit formulas of the vector coefficients
\begin{equation}\label{decomposition-Ms}
M_s(\psi) = M_s^\text{{\normalfont{scalar}}}(\psi) + M_s^\text{{\normalfont{vect}}}(\psi) \quad (s =1,2,3),
\end{equation}
for $\psi$ in (\ref{expression-psi}), in terms of the polynomials $L_r(g_{m-1}, g_m, g_{m+1})$ ($r = 1, \ldots 6$) defined in (\ref{expression-operators-Lr}). To do so, we first recall the notion of the \emph{$T$-saturation} (cf. \cite[Sec. 3.2]{kob-pev2} and \cite[Sec. 6.5]{kkp}).

Recall from (\ref{Def-T_b}) that, given $\ell \in \N$,
\begin{equation*}
T_\ell: \C[t] \longrightarrow \C(\zeta_1, \zeta_2, \zeta_3)
\end{equation*}
is a linear map defined by
\begin{equation*}
(T_\ell g)(\zeta) := Q_2(\zeta^\prime)^\frac{\ell}{2} g\left(\frac{\zeta_3}{\sqrt{Q_2(\zeta^\prime)}}\right), \quad \text{where } Q_2(\zeta^\prime) := \zeta_1^2 + \zeta_2^2.
\end{equation*}

We say that a differential operator $D$ on $\C^3$ is $T$-\textbf{saturated} if there exists an operator $S$ on $\C[t]$ such that the following diagram commutes
\begin{equation*}
\begin{tikzcd}
\C[t]\arrow[r,"T_\ell"]\arrow[d, "S"'] & \C(\zeta_1, \zeta_2, \zeta_3)\arrow[d, "D"]\\
\C[t]\arrow[r, "T_\ell"] & \C(\zeta_1, \zeta_2, \zeta_3)
\end{tikzcd}
\end{equation*}
Such operator $S$ is unique as far as it exists and it will be denoted by $S = T_\ell^\sharp(D)$, namely, $T_\ell^\sharp$ satisfies
\begin{equation*}
T_\ell \circ T_\ell^\sharp(D) = D \circ T_\ell.
\end{equation*}
Note that the following multiplicative property is satisfied whenever it makes sense:
\begin{equation*}
T_\ell^\sharp(D_1 \cdot D_2) = T_\ell^\sharp(D_1) \cdot T_\ell^\sharp(D_2).
\end{equation*}

\begin{lemma}[{\cite[Lem 6.27]{kkp}}]\label{lemma-Tsaturation-formulas} Let $S_{\ell}^\mu$ be the imaginary Gegenbauer differential operator (see \text{{\normalfont{(\ref{Gegen-imaginary})}}}) and let $\vartheta_t = \frac{d}{dt}$ be the Euler operator. Then, for any $\ell \in \N$ and any $g \in \Pol_\ell[t]_\text{{\normalfont{even}}}$  (see \text{{\normalfont{(\ref{def-Pol_even})}}}), the following identities hold.
\begin{enumerate}[label=\normalfont{(\arabic*)}, topsep=6pt, itemsep=6pt]
\item $\displaystyle{T_\ell^\sharp\left(\frac{Q_2(\zeta^\prime)}{\zeta_1}\widehat{
d\pi_{\lambda^*}}(C_1^+)\right) = S_{\ell}^{\lambda-1}}$.
\item $\displaystyle{\left(T_\ell g\right)(\zeta) = Q_2(\zeta^\prime)\left(T_{\ell-2}g\right)(\zeta)}$.
\item $\displaystyle{\frac{\partial}{\partial \zeta_j}\left(T_\ell g\right)(\zeta) = \frac{\zeta_j}{Q_2(\zeta^\prime)}T_\ell\left((\ell-\vartheta_t)g\right)(\zeta)}$, for $j=1,2$.
\item $\displaystyle{\frac{\partial}{\partial \zeta_3}\left(T_\ell g\right)(\zeta)= T_{\ell-1}\left(\frac{dg}{dt}\right)(\zeta)}$.
\end{enumerate}
\end{lemma}

\begin{lemma}\label{lemma-formulae-Ms-final} Let $\psi = \sum_{k = m-1}^{m+1}(T_{a-k}g_k)h_k^+$. Then, the vector coefficients $M_s(\psi) = M_s^\text{{\normalfont{scalar}}}(\psi) + M_s^\text{{\normalfont{vect}}}(\psi)$ $(s=1,2,3)$ are given as follows.
\begin{equation*}
\begin{aligned}
M_1(\psi) = (&\zeta_1 + i\zeta_2)^{m-2} \bigg[\zeta_1^2 T_{a-m-1}\left(S_{a-m+1}^{\lambda - 1}g_{m-1} + 2\frac{d}{dt}g_m\right)\\
&+ \zeta_2^2T_{a-m-1}\left( 2(a-m+1 -\vartheta_t) g_{m-1} - 2 \frac{d}{dt}g_m \right) \\
&+ \zeta_1 \zeta_2 T_{a-m -1} \left(iS_{a-m+1}^{\lambda-1}g_{m-1} - 2i (a-m+1 - \vartheta_t)g_{m-1} + 4i \frac{d}{dt}g_m \right) \\
&+ T_{a-m+1}\left(2(m-1)(\lambda+a)g_{m-1}\right)\bigg],\\[5pt]
M_2(\psi) = (&\zeta_1 + i \zeta_2)^{m-1} \bigg[\zeta_1^2 T_{a-m-2}\left( S_{a-m}^{\lambda-1} g_m + \frac{d}{dt}g_{m+1}\right) - \zeta_2^2 T_{a-m-2}\left(\frac{d}{dt}g_{m+1} \right) \\
&+ \zeta_1 \zeta_2 T_{a-m-2}\left(iS_{a-m}^{\lambda-1}g_m + 2i \frac{d}{dt}g_{m+1} \right) + T_{a-m}\left(2m(\lambda + a -1) g_m - \frac{d}{dt}g_{m-1} \right) \bigg],\\[5pt]
M_3(\psi) = (&\zeta_1 + i \zeta_2)^m \bigg[\zeta_1^2 T_{a-m-3}\left(S_{a-m-1}^{\lambda-1}g_{m+1}\right)- \zeta_2^2 T_{a-m-3}\left(2(a-m-1-\vartheta_t)g_{m+1}\right) \\
& + \zeta_1\zeta_2 T_{a-m-3}\left(iS_{a-m-1}^{\lambda-1}g_{m+1} + 2i(a - m -1 -\vartheta_t)g_{m+1} \right)\\
&+ T_{a-m-1}\left(2(m+1)(\lambda + a-2)g_{m+1} - 2\frac{d}{dt}g_m \right) \bigg].
\end{aligned}
\end{equation*}
\end{lemma}
\begin{proof}
We prove the lemma for the first vector coefficient $M_1$. The other two can be computed in a similar fashion.
By the expressions of the generators $h_k^+$ (see (\ref{generators-Hom(V3,CmHk)})), we can write $\psi$ in coordinates as follows:
\begin{equation*}
\psi = \begin{pmatrix}
\displaystyle{(T_{a-m+1}g_{m-1})(\zeta) (\zeta_1 + i\zeta_2)^{m-1}}\\[4pt]
\displaystyle{(T_{a-m}g_{m})(\zeta) (\zeta_1+ i\zeta_2)^{m}}\\[4pt]
\displaystyle{(T_{a-m-1}g_{m+1})(\zeta) (\zeta_1 + i\zeta_2)^{m+1}}
\end{pmatrix}.
\end{equation*}

On the other hand, by \cite[Lem. 4.6]{kkp} we obtain the following formula:
\begin{equation*}
\widehat{d\pi_{\lambda^*}}(C_1^+)\left((T_{a-k}g_k)h_k^+\right) = \widehat{d\pi_{\lambda^*}}(C_1^+)(T_{a-k}g_k)h_k^+ + 2(\lambda + a -1)(T_{a-k}g_k)\frac{\partial h_k^+}{\partial \zeta_1}.
\end{equation*}
Now, by (\ref{Ms-scalar}) and Lemma \ref{lemma-Tsaturation-formulas} (1), we obtain the following expression of $M_1^\text{{\normalfont{scalar}}}$
\begin{equation*}
\begin{aligned}
M_1^\text{{\normalfont{scalar}}}(\psi) = &\enspace \frac{\zeta_1}{Q_2(\zeta^\prime)}T_{a-m+1}\left(S_{a-m+1}^{\lambda-1}g_{m-1} \right)(\zeta_1 +i\zeta_2)^{m-1} \\[4pt]
&+ 2(m-1)(\lambda + a -1)(\zeta_1 + i \zeta_2)^{m-2}\left(T_{a-m+1}g_{m-1} \right),
\end{aligned}
\end{equation*}
and by (\ref{Ms-vect}) the following one of $M_1^\text{{\normalfont{vect}}}$:
\begin{equation*}
M_1^\text{{\normalfont{vect}}}(\psi) = -2i\frac{\partial}{\partial \zeta_2} \left[\left(T_{a-m+1}g_{m-1}\right)(\zeta_1 +i\zeta_2)^{m-1} \right] + 2\frac{\partial}{\partial \zeta_3} \left(T_{a-m}g_m \right)(\zeta_1 + i\zeta_2)^m.
\end{equation*}

Simplifying the expressions above by using Lemma \ref{lemma-Tsaturation-formulas} leads to
\begin{equation*}
M_1^\text{{\normalfont{scalar}}}(\psi) = \zeta_1(\zeta_1 +i\zeta_2)^{m-1} T_{a-m-1}\left(S_{a-m+1}^{\lambda-1}g_{m-1} \right) + 2(m-1)(\lambda + a -1)(\zeta_1 + i \zeta_2)^{m-2}\left(T_{a-m+1}g_{m-1} \right),
\end{equation*}
and
\begin{equation*}
\begin{aligned}
M_1^\text{{\normalfont{vect}}}(\psi) = & -2i\frac{\partial}{\partial \zeta_2} \left(T_{a-m+1}g_{m-1}\right)(\zeta_1 + i\zeta_2)^{m-1} \\
&+2(m-1)(\zeta_1 + i\zeta_2)^{m-2}\left(T_{a-m+1} g_{m-1}\right)
+ 2\frac{\partial}{\partial \zeta_3} \left(T_{a-m}g_m \right)(\zeta_1 + i\zeta_2)^m \\
= & -2i\frac{\zeta_2}{Q_2(\zeta^\prime)}T_{a-m+1}\left((a - m + 1 -\vartheta_t)g_{m-1}\right)(\zeta_1 + i \zeta_2)^{m-1} \\
&+2(m-1)(\zeta_1 + i\zeta_2)^{m-2}\left(T_{a-m+1} g_{m-1}\right) + 2T_{a-m-1}\left(\frac{d}{dt}g_m\right)(\zeta_1 + i\zeta_2)^m \\
=& -2i \zeta_2 (\zeta_1 + i \zeta_2)^{m-1} T_{a-m-1}\left((a - m + 1 -\vartheta_t)g_{m-1}\right) \\
&+2(m-1)(\zeta_1 + i\zeta_2)^{m-2}\left(T_{a-m+1} g_{m-1}\right)
+ 2(\zeta_1 + i\zeta_2)^mT_{a-m-1}\left(\frac{d}{dt}g_m\right),
\end{aligned}
\end{equation*}
where we applied Lemma \ref{lemma-Tsaturation-formulas} (2) for $M_1^\text{{\normalfont{scalar}}}$, and Lemma \ref{lemma-Tsaturation-formulas} (3), (4) and (2) in the last two equalities of $M_1^\text{{\normalfont{vect}}}$ respectively. Now, combining the scalar and vector parts of $M_1$ we obtain
\begin{equation*}
\begin{aligned}
M_1(\psi) = M_1^\text{{\normalfont{scalar}}}(\psi) + M_1^\text{{\normalfont{vect}}}(\psi) = &\enspace (\zeta_1 + i\zeta_2)^{m-2} \Big[\zeta_1(\zeta_1 + i\zeta_2)
T_{a-m-1}\left(S_{a-m+1}^{\lambda - 1} g_{m-1}\right) \\
& +T_{a-m+1}\left(2(m-1)(\lambda+a)g_{m-1}\right)\\
&- i \zeta_2 (\zeta_1 + i\zeta_2)T_{a-m-1}\left(2(a-m+1-\vartheta_t)g_{m-1}\right) \\
&+ (\zeta_1 + i \zeta_2)^2T_{a-m-1}\left(2\frac{d}{dt}g_m\right)\Big].
\end{aligned}
\end{equation*}

A re-arrangement of the terms above gives the desired formula.
\end{proof}

\subsection{Proof of Theorem \ref{Thm-findingequations}}\label{section-proof_of_finding_equations-subsection}
In this subsection we prove Theorem \ref{Thm-findingequations}. The following lemma will play a key role.

\begin{lemma} \label{lemma-invariant-polynomials} Suppose that $p_j\in \Pol(\C^3)$ for $j=1, \ldots, 4$ are $O(2,\C)$-invariant polynomials with respect to the action given by {\normalfont{(\ref{action-SO(2)-on-Pol})}}. Then, the following two conditions on $p_j$ are equivalent:
\begin{enumerate}[label=\normalfont{(\roman*)}, topsep=0pt]
\item $\zeta_1^2 p_1 + \zeta_2^2 p_2 + \zeta_1\zeta_2 p_3 + p_4 = 0$.
\item $p_1 = p_2$, $p_3 = 0$, and $Q_2(\zeta^\prime) p_1 + p_4 = 0$.
\end{enumerate}
\end{lemma}
\begin{proof}
The direction (ii) $\Rightarrow$ (i) is trivial. We prove (i) $\Rightarrow$ (ii). Suppose (i), then 
\begin{equation}\label{equation-lemma-invariant-polynomials}
\zeta_1^2 p_1 + \zeta_2^2 p_2 + \zeta_1\zeta_2 p_3  = -p_4 \in \Pol(\C^3)^{O(2,\C)}.
\end{equation}
We denote by $\cdot$ the natural action of $O(2,\C)$ on $\Pol(\C^3)$. 
Take $B_1 = E_{1,2} + E_{2,1}
\in O(2,\C)$, and take invariants with respect to $B_1$ in both sides of (\ref{equation-lemma-invariant-polynomials}). We obtain:
\begin{equation*}
\begin{gathered}
B_1 \cdot ( \zeta_1^2 p_1 + \zeta_2^2 p_2 + \zeta_1\zeta_2 p_3 )= B_1 \cdot (-p_4) = - p_4\\
\Longrightarrow \zeta_2^2 p_1 + \zeta_1^2 p_2 + \zeta_1\zeta_2 p_3 = \zeta_1^2 p_1 + \zeta_2^2 p_2 + \zeta_1\zeta_2 p_3\\
\Longrightarrow (\zeta_1^2 - \zeta_2^2)(p_2 - p_1) = 0.
\end{gathered}
\end{equation*}
Thus $p_1 = p_2$. 

On the other hand, if we take invariants with respect to
$B_2 := -E_{1,2}+E_{2,1}
\in O(2,\C)$
in both sides of (\ref{equation-lemma-invariant-polynomials}), we have
\begin{equation*}
\begin{gathered}
B_2 \cdot ( (\zeta_1^2 + \zeta_2^2)p_1 + \zeta_1\zeta_2p_3) = B_2 \cdot (-p_4) = -p_4\\
\Longrightarrow (\zeta_1^2 + \zeta_2^2)p_1 - \zeta_1\zeta_2p_3 = (\zeta_1^2 + \zeta_2^2)p_1 + \zeta_1\zeta_2p_3\\
\Longrightarrow 2\zeta_1\zeta_2p_3 = 0.
\end{gathered}
\end{equation*}
Hence, $p_3 = 0$. Finally, the identity $Q_2(\zeta^\prime)p_1 + p_4 = 0$ follows from $p_1 = p_2$ and $p_3 = 0$.
\end{proof}
Let $L_r(g_{m-1}, g_m, g_{m+1})(t) \enspace (r=1, \ldots, 6)$ be the polynomials defined in (\ref{expression-operators-Lr}) with $(f_0, f_1, f_2) = (g_{m-1}, g_m, g_{m+1})$. The following proposition leads to the proof of Theorem \ref{Thm-findingequations}.





\begin{prop}\label{prop-equivalence-Ms-Lr}
Let $M_s$ be the vector coefficients defined in {\normalfont{(\ref{decomposition-Ms})}}. Then, the following hold.
\begin{enumerate}[label=\normalfont{(\arabic*)}, topsep=0pt]
\item $M_1 = 0 \Leftrightarrow L_r(g_{m-1}, g_m, g_{m+1}) = 0$, for $r = 1,5$.
\item $M_2 = 0 \Leftrightarrow L_r(g_{m-1}, g_m, g_{m+1}) = 0$, for $r = 3,4$.
\item $M_3 = 0 \Leftrightarrow L_r(g_{m-1}, g_m, g_{m+1}) = 0$, for $r = 2,6$.
\end{enumerate}
\end{prop}
\begin{proof}[Proof of Proposition \ref{prop-equivalence-Ms-Lr}]
First, we observe that
\begin{equation*}
L_r(g_{m-1},g_m, g_{m+1}) = 0 \enspace (\forall r=1, \ldots 6) \enspace \Leftrightarrow \enspace L_r^\prime(g_{m-1},g_m, g_{m+1}) = 0 \enspace (\forall r=1, \ldots 6),
\end{equation*}
where the operators $L_r^\prime$ are defined as follows:
\begin{equation*}
\begin{array}{l}
L_1^\prime(g_{m-1}, g_m, g_{m+1})(t) := S_{a-m+1}^{\lambda - 1}g_{m-1} -2(a-m+1 - \vartheta_t)g_{m-1} + 4\frac{d}{dt}g_m.\\[0.3cm]
 
L_2^\prime(g_{m-1}, g_m, g_{m+1})(t) := S_{a-m+1}^{\lambda - 1}g_{m-1} + 2(m-1)(\lambda + a)g_{m-1}
 +2\frac{d}{dt}g_m.\\[0.3cm]

L_3^\prime(g_{m-1}, g_m, g_{m+1})(t) := S_{a-m}^{\lambda - 1}g_m +2\frac{d}{dt}g_{m+1}.\\[0.3cm]

L_4^\prime(g_{m-1}, g_m, g_{m+1})(t) := S_{a-m}^{\lambda-1}g_m + 2m(\lambda + a-1)g_m +\frac{d}{dt}g_{m+1} - \frac{d}{dt}g_{m-1}.\\[0.3cm]

L_5^\prime(g_{m-1}, g_m, g_{m+1})(t) := S_{a-m-1}^{\lambda-1}g_{m+1} + 2(a-m-1-\vartheta_t)g_{m+1}.\\[0.3cm]

L_6^\prime(g_{m-1}, g_m, g_{m+1})(t) := S_{a-m-1}^{\lambda-1}g_{m+1} + 2(m+1)(\lambda + a -2)g_{m+1} - 2 \frac{d}{dt}g_m.
\end{array}
\end{equation*}
In fact, this a consequence of the linear relations below, which can be proved easily.
\begin{equation*}
\begin{cases}
L_1 = 2L_2^\prime - L_1^\prime,\\
L_2 = L_5^\prime,
\end{cases}
\begin{cases}
 L_3 = 2L_4^\prime - L_3^\prime,\\
 L_4 = L_3^\prime,
\end{cases}
\begin{cases}
2L_5 = L_1^\prime-L_2^\prime,\\
2L_6 = L_6^\prime - L_5^\prime.
\end{cases}
\end{equation*}

We prove the first equivalence, the others can be proved in the same manner. From the observation above, it suffices to show $M_1 = 0 \Leftrightarrow L_r^\prime(g_{m-1}, g_m, g_{m+1}) = 0$, for $r = 1,2$. \\
We set
\begin{equation*}
\begin{aligned}
p_1 &:= T_{a-m-1}\left(S_{a-m+1}^{\lambda-1}g_{m-1} + 2\frac{d}{dt}g_m\right),\\
p_2 &:=  T_{a-m-1}\left(2(a-m+1-\vartheta_t)g_{m-1}-2\frac{d}{dt}g_m\right),\\
p_3 &:=  T_{a-m-1}\left(iS_{a-m+1}^{\lambda-1}g_{m-1} - 2i(a - m +1 -\vartheta_t)g_{m-1} +4i\frac{d}{dt}g_m \right) = i(p_1-p_2),\\
p_4 &:= T_{a-m+1}\left(2(m-1)(\lambda + a)g_{m-1}\right).
\end{aligned}
\end{equation*}

From the expression of $M_1$ obtained in Lemma \ref{lemma-formulae-Ms-final} we deduce that 
\begin{equation*}
M_1 = 0 \Leftrightarrow \zeta_1^2 p_1 + \zeta_2^2 p_2 + \zeta_1 \zeta_2 p_3 + p_4 = 0,
\end{equation*}
and by Lemma \ref{lemma-invariant-polynomials}, this is equivalent to
\begin{equation*}
p_1 - p_2 = p_3 = \enspace Q_2(\zeta^\prime)p_1 + p_4 = 0,
\end{equation*}
in other words, to
\begin{equation*}
\begin{aligned}
&T_{a-m-1}\left(S_{a-m+1}^{\lambda-1}g_{m-1} - 2(a - m +1 -\vartheta_t)g_{m-1} + 4\frac{d}{dt}g_m \right) &= 0,\\
& Q_2(\zeta^\prime)T_{a-m-1}\left(S_{a-m+1}^{\lambda-1}g_{m-1}+ 2\frac{d}{dt}g_m\right) +  T_{a-m+1}\left(2(m-1)(\lambda + a)g_{m-1}\right) &= 0.
\end{aligned}
\end{equation*}

Since $T_\ell$ is a bijection (see (\ref{bijection-T_b})), and by Lemma \ref{lemma-Tsaturation-formulas} (2), the two equations above are equivalent to
\begin{equation*}
\begin{aligned}
&S_{a-m+1}^{\lambda-1}g_{m-1} - 2(a-m+1-\vartheta_t)g_{m-1} + 4\frac{d}{dt}g_m &= 0,\\
&S_{a-m+1}^{\lambda-1}g_{m-1} + 2(m-1)(\lambda + a)g_{m-1} + 2\frac{d}{dt}g_m &= 0,
\end{aligned}
\end{equation*}
which are nothing but $L_1^\prime(g_{m-1},g_m,g_{m+1}) = L_2^\prime(g_{m-1},g_m,g_{m+1}) = 0$. Thus, the equivalence $M_1 = 0 \Leftrightarrow L_r^\prime(g_{m-1},g_m,g_{m+1}) = 0, \enspace (r = 1,5)$ has been proved.
\end{proof}
We now give a proof for Theorem \ref{Thm-findingequations}.
\begin{proof}[Proof of Theorem \ref{Thm-findingequations}]
As we pointed out before, (i) is satisfied for all $C \in \n_+^\prime$ if and only if it is satisfied for one element of $\n_+^\prime$, say $C_1^+ \in \n_+^\prime$ (cf. \cite[Lem. 3.4]{kkp}). In other words, (i) is equivalent to
\begin{equation*}
\left(\widehat{d\pi_\mu}(C_1^+)\otimes \id_{\C_{m,\nu}}\right)\psi = 0,
\end{equation*}
which in turn is equivalent to $M_s(\psi) = M_s^\text{scalar}(\psi) + M_s^\text{vect}(\psi) = 0$ $(s=1,2,3)$ as we pointed out in (\ref{Ms-equal-zero}). The result follows now from Proposition \ref{prop-equivalence-Ms-Lr}.
\end{proof}
\section{Proof of Theorem \ref{Thm-solvingequations}}
\label{section-proof_of_solving_equations}
In this section we prove Theorem \ref{Thm-solvingequations}; in other words, we solve the system (\ref{equation-space}), which is given by the following differential equations (for $f_j\in \Pol_{a-m+1-j}[t]_\text{{\normalfont{even}}}$, $j = 0,1,2$):
\begin{equation*}
\begin{array}{l}
(E_1)\enspace S_{a+m-1}^{\lambda} f_0 = 0.\\[0.2cm]

(E_2)\enspace S_{a-m-1}^{\lambda} f_2 = 0.\\[0.2cm]

(E_3)\enspace S_{a+m}^{\lambda - 1}f_1 - 2\frac{d}{dt}f_0 = 0.\\[0.2cm]

(E_4)\enspace S_{a-m}^{\lambda - 1}f_1 + 2\frac{d}{dt}f_2 = 0.\\[0.2cm]

(E_5)\enspace (-m(\lambda + a-1) + \lambda -1 + \vartheta_t) f_0 + \frac{d}{dt}f_1 = 0. \\[0.2cm]

(E_6)\enspace (m(\lambda + a-1) + \lambda -1 + \vartheta_t) f_2 - \frac{d}{dt}f_1 = 0. 
\end{array}
\end{equation*}
The proof is a little technical, and we use several properties of Gegenbauer polynomials which are listed in the appendix. We follow the three-step strategy below:
\begin{equation*}
\begin{array}{ll}
\text{\textbf{Step 1}:} & \text{Obtain } f_0, f_1 \text{ and } f_2 \text{ up to constant from equations } (E_1)\text{--}(E_4).\\[2pt]
& \text{For } f_1 \text{ we obtain two different expressions}.\\[6pt]
\text{\textbf{Step 2}:} & \text{Check that } f_0, f_1, f_2 \text{ obtained in the first step satisfy equations } (E_5) \text{ and } (E_6).\\[2pt]
& \text{Here, we obtain linear relations on the constants that appear in the}\\[2pt]
& \text{expressions of } f_0, f_1, f_2.\\[6pt]
\text{\textbf{Step 3}:} & \text{Check that the two expressions for } f_1 \text{ obtained in the first step coincide}. 
\end{array}
\end{equation*}
The first two steps are quite straightforward, and can be done in a couple of lines by using some useful identities of the appendix. On the contrary, the third step is the hardest part of the proof and conforms the heart of this section.

\subsection{Steps 1 and 2: Obtaining the solution up to constant}
In this section we assume $a \geq m+1$. The case $a \leq m$ will be considered at the end of Section \ref{section-proof_of_solving_equations-subsection}.

First, suppose that $f_0, f_2$ ($f_j\in \Pol_{a-m+1-j}[t]_\text{{\normalfont{even}}}$) satisfy equations $(E_1), (E_2)$. Then, by Theorem \ref{Gegenbauer-solutions} and Lemma \ref{relationS-G}, we deduce that 
\begin{equation}\label{expression_f0-and-f2}
f_0(t) = q\Geg_{a+m-1}^\lambda(it), \quad f_2(t) = s\Geg_{a-m-1}^\lambda(it),
\end{equation}
for some constants $q,s \in \C$. Observe that the degree of $\Geg_{a-m-1}^\lambda(it)$ coincides with that of the space $\Pol_{a-m-1}[t]_\text{{\normalfont{even}}}$, so the expression for $f_2$ make sense. However, this does not happen for $f_0$. By assumption
\begin{equation*}
f_0 \in \Pol_{a-m+1}[t]_\text{{\normalfont{even}}} \subseteq \Pol_{a+m-1}[t]_\text{{\normalfont{even}}}, \quad \text{for } m \in \N_+,
\end{equation*}
so a priori the degree of the polynomial $\Geg^\lambda_{a+m-1}(it)$ is greater than the degree of $f_0$. We will see that $\Geg^\lambda_{a+m-1}(it)$ actually has degree $a-m+1$ for an appropriate parameter $\lambda \in \C$ (in other words, that the higher terms of $\Geg^\lambda_{a+m-1}(it)$ vanish), so that the expression for $f_0$ is indeed correct. An example of this type of phenomenon can be seen in Lemma \ref{lemma-KOSS}.

\begin{lemma}\label{lemma-step-1} 
Suppose that $f_0, f_2$ are given by {\normalfont{(\ref{expression_f0-and-f2})}} and take $f_1 \in \Pol_{a-m}[t]_\text{{\normalfont{even}}}$. Then
\begin{enumerate}[label=\normalfont{(\arabic*)}, topsep=0pt]
\item $f_1$ satisfies equation $(E_3)$ if and only if
\begin{equation}\label{expression_f1_1}
f_1(t) = -p\Geg_{a+m}^{\lambda-1}(it) - qt\Geg_{a+m-1}^\lambda(it),
\end{equation}
for some constant $p \in \C$.
\item $f_1$ satisfies equation $(E_4)$ if and only if
\begin{equation}\label{expression_f1_2}
f_1(t) = -r\Geg_{a-m}^{\lambda -1}(it) + st\Geg_{a-m-1}^\lambda(it),
\end{equation}
for some constant $r \in \C$.
\end{enumerate}
\end{lemma}
\begin{proof}
(1) By the identity (\ref{Slmu-identity-2}) and equation $(E_1)$ we have 
\begin{equation*}
S_{a+m}^{\lambda-1}\left(f_1 +tf_0\right) = 0.
\end{equation*}
Thus, by Theorem \ref{Gegenbauer-solutions} and Lemma \ref{relationS-G}, $f_1$ is given by (\ref{expression_f1_1}) for some constant $p\in \C$. The second statement can be proved in a similar fashion.
\end{proof}

\begin{lemma} \label{lemma-step-2}
{\normalfont{(1)}} Suppose that $f_0, f_2$ are given by {\normalfont{(\ref{expression_f0-and-f2})}} and that $f_1$ is given by {\normalfont{(\ref{expression_f1_1})}}. Then, $(E_5)$ is satisfied if and only if the following linear relation between $p$ and $q$ holds:
\begin{equation}\label{p_and_q_relation}
2i\gamma(\lambda-1, a+m)p +(m(\lambda + a-1) - \lambda +2)q = 0.
\end{equation}
{\normalfont{(2)}} Suppose that $f_0, f_2$ are given by {\normalfont{(\ref{expression_f0-and-f2})}} and that $f_1$ is given by {\normalfont{(\ref{expression_f1_2})}}. Then, $(E_6)$ is satisfied if and only if the following linear relation between $r$ and $s$ holds:
\begin{equation}\label{r_and_s_relation}
2i\gamma(\lambda-1, a-m)r + (m(\lambda + a-1) + \lambda -2)s = 0.
\end{equation}
\end{lemma}
\begin{proof}
(1) Substituting (\ref{expression_f0-and-f2}) and (\ref{expression_f1_1}) in $(E_5)$ and using the identities (\ref{derivative-Gegenbauer-1}), (\ref{derivative-Gegenbauer-2}) and (\ref{KKP-2}), $(E_5)$ amounts to
\begin{equation*}
\Big[- 2i\gamma(\lambda-1, a+m)p +(-m(\lambda + a-1) + \lambda -2)q\Big]\Geg_{a+m-1}^\lambda(it) = 0.
\end{equation*}
Thus, since $a+m-1 \geq 2m > 0$, $(E_5)$ is satisfied if and only if $p$ and $q$ are in the linear relation (\ref{p_and_q_relation}).\\
For (2), repeating the same process with (\ref{expression_f1_2}), $(E_6)$ can be shown to be equivalent to
\begin{equation*}
\Big[2i\gamma(\lambda-1, a-m)r +(m(\lambda + a-1) + \lambda -2)s\Big]\Geg_{a-m-1}^\lambda(it) = 0.
\end{equation*}
Hence, since $a-m-1 \geq 0$, $(E_6)$ is satisfied if and only if $r$ and $s$ are in the linear relation (\ref{r_and_s_relation}).
\end{proof}

\begin{rem}\label{remark-case-a=m} In Lemmas \ref{lemma-step-1} and \ref{lemma-step-2} we assumed $a \geq m+1$, but by looking carefully at the proof, one can deduce that Lemma \ref{lemma-step-1} and Lemma \ref{lemma-step-2} (1) also hold in the case $a = m$. On the contrary, Lemma \ref{lemma-step-2} (2) does not hold since equation $(E_6)$ is trivially satisfied and (\ref{r_and_s_relation}) does not need to hold. This follows from the fact that $f_1$ is constant and $f_2 = 0$ by the condition $f_j\in \Pol_{a-m+1-j}[t]_\text{{\normalfont{even}}}$.
\end{rem}

With the lemmas above, we have completed Steps 1 and 2. All that remains now is to show that the two expressions obtained for $f_1$ coincide, that is, that the following identity is satisfied
\begin{equation}\label{the-equation-pqrs-twosided}
p\Geg_{a+m}^{\lambda-1}(it) + qt\Geg_{a+m-1}^\lambda(it) = r\Geg_{a-m}^{\lambda -1}(it) - st\Geg_{a-m-1}^\lambda(it).
\end{equation}

\subsection{Step 3: Proving identity (\ref{the-equation-pqrs-twosided})}
This section is devoted to proving Lemma \ref{lemma-equation-pqrs} below, which gives a necessary and sufficient condition on the parameters $(\lambda, m)$ and on the constants $(p,q,r,s)$ such that identity (\ref{the-equation-pqrs-twosided}) is satisfied.

\begin{lemma} \label{lemma-equation-pqrs} Let $(\lambda, m, a) \in \C \times \N_+ \times \N$ with $a \geq m+1$, and suppose that $(p,q,r,s) \in \C^4\setminus\{(0,0,0,0)\}$ satisfies {\normalfont{(\ref{p_and_q_relation})}} and {\normalfont{(\ref{r_and_s_relation})}}. Then, the following conditions are equivalent.
\begin{enumerate}[label=\normalfont{(\roman*)}, topsep=0pt]
\item $(p, q, r, s)$ satisfies {\normalfont{(\ref{the-equation-pqrs-twosided})}}.
\item $(\lambda, m)$ belongs to one of the following cases:
\begin{itemize}[leftmargin=1cm]
\item[]{\normalfont{Case (I).}} $m = 1$ and $\lambda \in \C$,
\item[]{\normalfont{Case (II).}} $m >1$ and $\lambda \in \{-a,1-a,2-a\}$,
\end{itemize}
and $(p, q, r, s)$ satisfies {\normalfont{(\ref{coefficients-ps-special-case-1})}}, {\normalfont{(\ref{coefficients-ps-special-case-2})}} and {\normalfont{(\ref{coefficients-qs})}} below.
\begin{equation}\label{coefficients-ps-special-case-1}
2ip - (a+1)s = 0, \quad \text{if } m = 1, a \in 2\N+1, \text{ and } \lambda = -\frac{a-1}{2},
\end{equation}
\begin{equation}\label{coefficients-ps-special-case-2}
ir + s = q + s = 0, \quad \text{if } m = 1, a \in 2\N+1, \text{ and } \lambda = -\frac{a-3}{2},
\end{equation}
\begin{equation}\label{coefficients-qs}
q + \frac{\Gamma\left(\lambda + \left[\frac{a+m}{2}\right]\right)}{\Gamma\left(\lambda + \left[\frac{a-m}{2}\right]\right)}s = 0, \quad \text{otherwise.}
\end{equation}
\end{enumerate}
\end{lemma}

In order to prove Lemma \ref{lemma-equation-pqrs}, we make use of several technical results (see Lemmas \ref{lemma-lambda-sets} and \ref{lemma-lowerterms} below). Let us first introduce  some notation. We rewrite equation (\ref{the-equation-pqrs-twosided}) as 
\begin{equation}\label{the-equation-pqrs-onesided}
\begin{gathered}
p\Geg_{a+m}^{\lambda-1}(it) + qt\Geg_{a+m-1}(it) - r\Geg_{a-m}^{\lambda-1}(it) + st\Geg_{a-m-1}^\lambda(it) \\
= A_{a+m}t^{a+m} + A_{a+m-2}t^{a+m-2} + \cdots + A_{a+m-2\left[\frac{a+m}{2}\right]} t^{a+m-2\left[\frac{a+m}{2}\right]}.
\end{gathered}
\end{equation}
For $\lambda \in \C$, $m \in \N_+$, $a\geq m$, and $k = 0, \dots, m-1$, we set
\begin{equation*}
\begin{aligned}
\Lambda^m_{a+m-2k} := & \left\{ -\left[\frac{a+m}{2}\right]-j : j = 0, 1, 2, \dots, \left[\frac{a+m-1}{2}\right]-1-k\right\}\\
 = &\left\{ \alpha + k -a -m +1 : \alpha = 1, 2, \dots, \left[\frac{a+m-1}{2}\right]-k\right\}.
\end{aligned}
\end{equation*}
Note that if $a = m$ and $k = m-1$, then $\left[\frac{a+m-1}{2}\right]-1-k = m-1 - 1 - k = -1 < 0$. So in this case we shall define $\Lambda_{a+m-2k}^m = \Lambda_{2}^m = \emptyset$. Otherwise, we have $\left[\frac{a+m-1}{2}\right]-1-k \geq 0$, so $\Lambda_{a+m-2k}^m \neq \emptyset$.

\begin{rem}\label{Remark-Lambdasets} (1) $\Lambda^m_{a+m-2(k-1)} \supset \Lambda^m_{a+m-2k}$, for any $k = 1, \dots, m-1$.\\
(2) $\Lambda^m_{a+m-2(k-1)} \setminus \Lambda^m_{a+m-2k} = \{-(a+m-1) + k\}$, for any $k = 1, \dots, m-1$.
\end{rem}
In the following, we suppose that $(p,q,r,s) \in \C^4\setminus\{(0,0,0,0)\}$ and $(\lambda, m, a) \in \C \times \N_+ \times \N$, with $a\geq m$. Observe that we have
\begin{equation*}
\begin{aligned}
\deg\left(p\Geg_{a+m}^{\lambda-1}(it) + qt\Geg_{a+m-1}^\lambda(it)\right) & \leq a+m,\\
\deg\left(r\Geg_{a-m}^{\lambda -1}(it) - st\Geg_{a-m-1}^\lambda(it)\right) & \leq  a-m.
\end{aligned}
\end{equation*}
Thus, if (\ref{the-equation-pqrs-twosided}) holds, then $A_{a+m-2k} = 0$ for $k=0, 1, \dots, m-1$. 

In the following result, we give precisely a necessary and sufficient condition such that these \lq\lq higher terms\rq\rq{ }vanish.
\begin{lemma}\label{lemma-lambda-sets} Suppose $a \geq m$ and $k = 0, 1, \dots, m-1$. Then
\begin{enumerate}[label=\normalfont{(\arabic*)}, topsep=0pt]
\item $A_{a+m-2k} = 0$ if and only if $\lambda \in \Lambda^m_{a+m-2k}$ or $2i\gamma(\lambda-1, a+m)p + (a+m-2k)q = 0$.
\item Moreover, if $k >0$ and $m>1$, then $A_{a+m} = \dots = A_{a+m-2k} = 0$ if and only if one of the following four conditions is satisfied:
\begin{itemize}[leftmargin=1cm]
\item[{\normalfont{($\star_1$)}}] $\lambda \in \Lambda^m_{a+m-2k}$.
\item[{\normalfont{($\star_2$)}}] $\lambda + a + m -1-k =2i\gamma(\lambda-1, a+m)p + (a+m-2k)q = 0$, with $(p,q) \neq (0,0)$.
\item[{\normalfont{($\star_3$)}}] $a+m \in 2\N$ and $q = \lambda + \frac{a+m-2}{2} = 0$, with $p \neq 0$.
\item[{\normalfont{($\star_4$)}}] $(p,q) = (0,0)$.
\end{itemize}
\end{enumerate}
\end{lemma}
\begin{proof}
(1) From the definition of the renormalized Gegenbauer polynomials $\Geg^\mu_\ell(z)$ in (\ref{Gegenbauer-polynomial(renormalized)}), the formula of the term $A_{a+m-2k} \enspace (k = 0, 1, \dots, m-1)$ can be easily computed:
\begin{equation}\label{expression-higherterms}
A_{a+m-2k} = \frac{(-1)^k(2i)^{a+m-1-2k}}{k!(a+m-2k)!}\cdot\frac{\Gamma(\lambda+a+m-1-k)}{\Gamma(\lambda + \left[\frac{a+m}{2}\right])}\Big(2i\gamma(\lambda-1, a+m)p + (a+m-2k)q\Big).
\end{equation}
Suppose $a > m$ or $0 \leq k <m-1$. Since
\begin{equation*}
\begin{aligned}
\frac{\Gamma(\lambda+a+m-1-k)}{\Gamma(\lambda + \left[\frac{a+m}{2}\right])}   = & \left(\lambda + \left[\frac{a+m}{2}\right]\right) \left(\lambda + \left[\frac{a+m}{2}\right]+1\right)\\ & \cdot \ldots \cdot
 \left(\lambda + \left[\frac{a+m}{2}\right] + \left[\frac{a+m-1}{2}\right] -1 -k\right),
\end{aligned}		
\end{equation*}
we deduce from (\ref{expression-higherterms}) that $A_{a+m-2k}$ vanishes if and only if $\lambda \in \Lambda^m_{a+m-2k}$ or $2i\gamma(\lambda-1, a+m)p + (a+m-2k)q = 0$.

If $a = m$ and $k = m-1$, (\ref{expression-higherterms}) amounts to
\begin{equation*}
A_{a+m-2k} = A_{2} = \frac{2i(-1)^{m-1}}{(m-1)!}\left(i(\lambda+m-1)p +q\right).
\end{equation*}
Thus, $A_2 = 0$ if and only if
\begin{equation*}
2i\gamma(\lambda-1, a+m)p + (a+m-2k)q = 2\left(i(\lambda+m-1)p + q\right) = 0.
\end{equation*}
Since $\Lambda^m_{a+m-2k} = \Lambda^m_{2} = \emptyset$, the statement holds. 

(2) We prove the second statement by induction on $k = 1, 2, \dots, m-1$. 
Suppose that $k=1$. Then, by (1) we have that 
$A_{a+m} = A_{a+m-2} = 0$ if and only if the two following conditions are satisfied:
\begin{itemize}
\item$\lambda \in \Lambda^m_{a+m}$, or $2i\gamma(\lambda-1, a+m)p + (a+m)q = 0$.
\item $\lambda \in \Lambda^m_{a+m-2}$, or $2i\gamma(\lambda-1, a+m)p + (a+m-2)q = 0$.
\end{itemize}
Since $\Lambda^m_{a+m} \supset \Lambda^m_{a+m-2}$ by Remark \ref{Remark-Lambdasets} (1), this is equivalent to one of the following three conditions being satisfied:
\begin{itemize}
\item $\lambda \in \Lambda^m_{a+m} \cap \Lambda^m_{a+m-2} = \Lambda^m_{a+m-2}$.
\item $\lambda \in \Lambda^m_{a+m}$ and $2i\gamma(\lambda-1, a+m)p + (a+m-2)q = 0$.
\item $2i\gamma(\lambda-1, a+m)p + (a+m)q = 2i\gamma(\lambda-1, a+m)p + (a+m-2)q = 0$.
\end{itemize}
By Remark \ref{Remark-Lambdasets} (2), $\Lambda^m_{a+m} \supset \Lambda^m_{a+m-2} = \{-a-m+2\}$, so the second condition can be rewritten as
\begin{equation*}
\lambda +a+m-2 = 2i\gamma(\lambda-1, a+m)p + (a+m-2)q = 0.
\end{equation*}
On the other hand, the third amounts to
\begin{equation*}
2i\gamma(\lambda-1, a+m)p = q = 0,
\end{equation*}
from which we have two cases. If $p\neq 0$, then necessarily $\gamma(\lambda-1, a+m) = q = 0$, i.e., $a+m \in 2\N$ and $\lambda + \frac{a+m-2}{2} = q = 0$. The other case is $(p,q) = (0,0)$. Thus, the three conditions above can be rewritten in such a way that each condition cannot be simultaneously satisfied as follows.
\begin{itemize}
\item $\lambda \in \Lambda^m_{a+m-2}$.
\item $\lambda + a + m -2 =2i\gamma(\lambda-1, a+m)p + (a+m-2)q = 0$, with $(p,q) \neq (0,0)$.
\item $a+m \in 2\N$ and $q = \lambda + \frac{a+m-2}{2} = 0$, with $p \neq 0$.
\item $(p,q) = (0,0)$.
\end{itemize}
These are precisely the conditions $(\star_1)$--$(\star_4)$ for $k=1$.

Now, suppose that the assertion is true for some $k = 1, \dots, m-2$. Then, by assumption and by (1), $A_{a+m} = \dots = A_{a+m-2k} = 0 = A_{a+m-2(k+1)}$ if and only if one of the conditions $(\star_1)$--$(\star_4)$ (for $k$) together with one of the conditions below (for $k+1$) are satisfied
\begin{itemize}
\item[(I)] $\lambda \in \Lambda^m_{a+m-2(k+1)}$.
\item[(II)] $2i\gamma(\lambda-1, a+m)p + (a+m-2(k+1))q = 0$.
\end{itemize}
By Remark \ref{Remark-Lambdasets} (1), $(\star_1)$ together with (I) amounts to
\begin{equation*}
\lambda \in \Lambda^m_{a+m-2(k+1)}.
\end{equation*}
Similarly, by Remark \ref{Remark-Lambdasets} (2), $(\star_1)$ together with (II) amounts to
\begin{equation*}
\lambda +a+m-2-k = 2i\gamma(\lambda-1,a+m)p + (a+m-2(k+1))q = 0.
\end{equation*}
Since $-a-m+k+1 \in \Lambda^m_{a+m-2(k-1)}\setminus\Lambda^m_{a+m-2k}$, and $\Lambda^m_{a+m-2k} \supset \Lambda^m_{a+m-2(k+1)}$, $(\star_2)$  and (I) cannot be simultaneously satisfied. The same happens with $(\star_2)$ and (II) since if they were both satisfied, that would imply
\begin{equation*}
\lambda + a +m -k-1 = \lambda + \frac{a+m-2}{2} = q = 0 \Rightarrow a = 2k-m \leq m-2
\end{equation*}
which contradicts the hypothesis $a \geq m$. If $a+m\in 2\N$, $\frac{-a-m+2}{2} \notin \Lambda^m_{a+m-2(k+1)}$, so $(\star_3)$ and (I) cannot be satisfied at the same time neither. On the other hand, $(\star_3)$ together with (II) amounts to
\begin{equation*}
a+m \in 2\N \text{ and } q = \lambda + \frac{a+m-2}{2} = 0. 
\end{equation*} 
The last two pairs of conditions are $(\star_4)$ with (I) or (II), but in either case we obtain $(p,q) = (0,0)$. In conclusion, the only possibilities we obtained in order to have $A_{a+m} = \dots = A_{a+m-2k} = 0 = A_{a+m-2(k+1)}$ are the following four:
\begin{itemize}
\item $\lambda \in \Lambda^m_{a+m-2(k+1)}$.
\item $\lambda + a + m -2-k =2i\gamma(\lambda-1, a+m)p + (a+m-2(k+1))q = 0$, with $(p,q) \neq (0,0)$.
\item $a+m \in 2\N$ and $q = \lambda + \frac{a+m-2}{2} = 0$, with $p \neq 0$.
\item $(p,q) = (0,0)$.
\end{itemize}
These are precisely the conditions $(\star_1)$--$(\star_4)$ for $k+1$. The second statement is now proved by induction on $k$.
\end{proof}

\begin{rem} By the lemma above, if $m>1$, we know that a necessary condition for $(\ref{the-equation-pqrs-twosided})$ to be satisfied is that $\lambda$ must belong to a finite subset of $\Z$ unless $(p,q) = (0,0)$. We will see in the proof of Lemma \ref{lemma-equation-pqrs} that this subset is precisely $\{-a, 1-a, 2-a\}$.
\end{rem}

In the next lemma, we give a necessary and sufficient condition on the tuple $(\lambda, m, p,q,r,s)$ such that the first three \lq\lq lower terms\rq\rq{ }of (\ref{the-equation-pqrs-onesided}) vanish. Concretely, we set 
\begin{equation*}
\varepsilon := a+m-2\left[\frac{a+m}{2}\right] =  \begin{cases}
0, \text{ if } a+m \in 2\N,\\
1, \text{ if } a+m \in 2\N + 1.
\end{cases}
\end{equation*}
In Lemma \ref{lemma-lowerterms} below, we consider the following equations:
\begin{equation}\label{second-equations-1}
\begin{gathered}
2ip + \varepsilon q + (-1)^m \frac{\Gamma\left(\left[\frac{a+m+2}{2}\right]\right)}{\Gamma\left(\left[\frac{a-m+2}{2}\right] \right)}\left(-2ir + \varepsilon s\right) = 0,
\end{gathered}
\end{equation}
\begin{equation}\label{second-equations-2}
\begin{gathered}
(-1)^m \frac{\Gamma\left(\left[\frac{a+m}{2}\right]\right)}{\Gamma\left(\left[\frac{a-m+2}{2}\right]\right)}
\Big(m(\lambda + a -1)(2ir - \varepsilon s)
+ 2\left[\frac{a-m}{2}\right]\gamma(\lambda-1, a-m+1)s \Big)\\ + 2\gamma(\lambda-1, a+m+1)q = 0,
\end{gathered}
\end{equation}
\begin{equation}\label{second-equations-3}
\begin{gathered}
(\lambda + a-1) \Big[(- m(\lambda + a -1) + \lambda)ir - (a-m)s \Big] = 0, \text{ if } a+m \in 2\N
\end{gathered}
\end{equation}
\begin{equation}\label{second-equations-4}
\begin{gathered}
(\lambda + a-1)\Big[2\big(-m(\lambda + a -1) + \lambda + 1)ir + \big((-2a+3m+1)\lambda -a^2\\
 + (2 + 3m)a -m^2 -3m -2\big)s\Big] = 0, \text{ if } a+m \in 2\N +1.
 \end{gathered}
\end{equation}
\begin{lemma}\label{lemma-lowerterms}
Suppose $a \geq m+1$. Then the following holds.
\begin{enumerate}[label=\normalfont{(\arabic*)}, topsep=0pt]
\item $A_\varepsilon = 0$  if and only if {\normalfont{(\ref{second-equations-1})}} is satisfied.
\item If $a \geq m+2$, then $A_\varepsilon = A_{\varepsilon + 2} =0$ if and only if {\normalfont{(\ref{second-equations-1})}} and {\normalfont{(\ref{second-equations-2})}} are satisfied.
\item If $a \geq m+4$, then $A_\varepsilon = A_{\varepsilon + 2} = A_{\varepsilon + 4} = 0$ if and only if {\normalfont{(\ref{second-equations-1})}},
{\normalfont{(\ref{second-equations-2})}},  {\normalfont{(\ref{second-equations-3})}} and {\normalfont{(\ref{second-equations-4})}} are satisfied.\\
\end{enumerate}
\end{lemma}
\begin{proof} The conditions on $a$ in (2) and (3) are necessary to guarantee that the coefficients $A_{\varepsilon+j}$ in the lower terms are not identically zero. For example, if $a\geq m+2$, then generically $\deg(r\Geg_{a-m}^{\lambda-1}(it) - st\Geg_{a-m-1}^\lambda(it)) \geq 2$. Thus, if $a+m\in 2\N$, then $A_\varepsilon = A_0$, $A_{\varepsilon +2} = A_2 \not\equiv 0$. Similarly, if $a+m\in 2\N+1$, then $A_\varepsilon = A_1$, $A_{\varepsilon+2} = A_3 \not\equiv 0$.

From the definition of the renormalized Gegenbauer polynomials, the term $A_{a+m-2k}$ (for $k = m, \ldots, \left[\frac{a+m}{2}\right]$) of (\ref{the-equation-pqrs-onesided}) is given by the formula below.
\begin{equation*}
\begin{aligned}
A_{a+m-2k} & = \frac{(-1)^k(2i)^{a+m-1-2k}}{k!(a+m-2k)!} \Bigg[ 
\frac{\Gamma(\lambda+a+m-1-k)}{\Gamma(\lambda + \left[\frac{a+m}{2}\right])}\Big(2i\gamma(\lambda-1, a+m)p+ (a+m-2k)q\Big)\\
+ &(-1)^m \frac{\Gamma(k+1)}{\Gamma(k-m+1)}\frac{\Gamma(\lambda+a-1-k)}{\Gamma(\lambda + \left[\frac{a-m}{2}\right])}\Big(-2i\gamma(\lambda-1, a-m)r + (a+m-2k)s\Big)\Bigg].
\end{aligned}
\end{equation*}
The lemma follows from a direct computation by using this formula.
\end{proof}
\begin{lemma}\label{lemma-claim} Let $(\lambda, m, a) \in \C \times (2 + \N) \times \N$ with $a \geq m+1$, and suppose that $(p,q,r,s) \in \C^4\setminus\{(0,0,0,0)\}$ satisfies {\normalfont{(\ref{r_and_s_relation})}} and {\normalfont{(\ref{p_and_q_relation})}}. If $(p,q,r,s)$ satisfies {\normalfont{(\ref{the-equation-pqrs-twosided})}}, then $(p,q) = (0,0)$ implies $(r,s) = (0,0)$ unless $(a, \lambda) = (m+1, 1-m)$.
\end{lemma}
\begin{proof}
Suppose $(p,q) = (0,0)$. Then (\ref{the-equation-pqrs-twosided}) amounts to
\begin{equation*}
0 = r\Geg_{a-m}^{\lambda-1}(it) - st\Geg_{a-m-1}^\lambda(it).
\end{equation*}
If $a + m$ is even, we have $A_0 = 0$, which implies $r = 0$. Thus $st\Geg_{a-m-1}^\lambda(it) = 0$. Since $a\geq m+1$, this implies $s = 0$. If $a+m$ is odd and $a \geq m+3$, from $(p,q) = (0,0)$ we obtain $A_1 = A_3 = 0$, which by Lemma \ref{lemma-lowerterms} amounts to $-2ir + s = \lambda + \frac{a-m-1}{2} = 0$. By a direct computation, this conditions together with (\ref{r_and_s_relation}) amount to $(m-1)(a+m+1)s=0$. Since $m>1$ and $a \in \N$, this implies $s = 0$, which in turn implies $(r,s) = (0,0)$. On the other hand, if $a=m+1$, again by Lemma \ref{lemma-lowerterms}, from $A_1=0$ we have $-2ir + s = 0$, which together with (\ref{r_and_s_relation}) amounts to $(m+1)(\lambda + m-1)s = 0$. Hence, if $(a,\lambda) = (m+1, 1-m)$, the right-hand side of (\ref{the-equation-pqrs-twosided}) is zero while $(r,s) \neq (0,0)$. This concludes the lemma.
\end{proof}

We now prove Lemma \ref{lemma-equation-pqrs}.
\begin{proof}[Proof of Lemma \ref{lemma-equation-pqrs}]
(i) $\Rightarrow$ (ii) First, we prove that if (\ref{the-equation-pqrs-twosided}) is satisfied, then $(\lambda, m)$ must belong to Case (I) or (II) and that $(p,q,r,s)$ satisfies the identities (\ref{coefficients-ps-special-case-1}), (\ref{coefficients-ps-special-case-2}) and (\ref{coefficients-qs}).

Suppose $m = 1$. If (\ref{the-equation-pqrs-twosided}) is satisfied, then $A_{a+1} = 0$. It then follows from Lemma \ref{lemma-lambda-sets} that $\lambda \in \Lambda_{a+1}^1$ or (\ref{p_and_q_relation}) holds. By assumption $(p,q,r,s)$ satisfies (\ref{p_and_q_relation}). Thus, $\lambda$ can be any complex number. Therefore, in this case, $(\lambda, m)$ belongs to Case (I).

Suppose $m>1$. If (\ref{the-equation-pqrs-twosided}) is satisfied, then
\begin{equation*}
A_{a+m} = A_{a+m-2} = \cdots = A_{a-m+2} = 0.
\end{equation*}
Then, by Lemma \ref{lemma-lambda-sets}, one of the conditions $(\star_1)$--$(\star_4)$ in Lemma \ref{lemma-lambda-sets} holds. To look at $\lambda$ carefully, we consider the following two cases for $a \geq m+1$ separately.
\begin{enumerate}[leftmargin=2cm]
\item[Case 1.] $a \in \{m+1, m+2, m+3\}$.
\item[Case 2.] $a \geq m+4$.
\end{enumerate}

\underline{Case 1}. Suppose $a \in \{m+1, m+2, m+3\}$.
We first observe the case $a = m+3$. In this case, as $a+m$ is odd, $(\star_3)$ cannot occur. Further, $(\star_4)$ cannot occur either. Indeed, if $(p,q) = (0,0)$, then Lemma \ref{lemma-claim} shows that $(p,q,r,s) = (0,0,0,0)$, which contradicts the hypothesis of Lemma \ref{lemma-equation-pqrs}. Thus only either $(\star_1)$ or $(\star_2)$ occurs. Since $\Lambda_{a-m+2}^m = \{1- a, 2-a\}$ and since $(\star_2)$ is precisely (\ref{p_and_q_relation}) for $\lambda = -a$, we have $\lambda \in \{-a, 1- a, 2-a\}$, that is, $(\lambda, m)$ belongs to Case (II).

If $a \in \{m+1, m+2\}$, a similar argument can be applied. In these cases, $\Lambda_{a-m+2}^m = \{1- a\}$ and $(\star_2)$ leads to $\lambda = -a$ as before. If $a = m+2$, $(\star_3)$ can be satisfied and corresponds to $\lambda = -m = 2-a$; while $(\star_4)$ cannot occur again by Lemma \ref{lemma-claim}. On the other hand, if $a=m+1$, $(\star_3)$ cannot be satisfied while $(\star_4)$ can occur if $\lambda = 1-m = 2-a$ again by Lemma \ref{lemma-claim}.

Thus, for $m>1$, if $a \in \{m+1, m+2, m+3\}$ then $(\lambda, m)$ must belong to Case (II).

\underline{Case 2}. Suppose $a \geq m+4$. Then, by Lemma \ref{lemma-lowerterms} (iii), $A_\varepsilon$, $A_{\varepsilon + 2}$ and $A_{\varepsilon+4}$ vanish if and only if (\ref{second-equations-1})--(\ref{second-equations-4}) are satisfied. Suppose in addition that $\lambda \neq 1-a$. Then, as $m$ is assumed to be $m>1$, one can easily check that (\ref{r_and_s_relation}) and (\ref{second-equations-3})--(\ref{second-equations-4}) are both satisfied if and only if $\lambda \in \{-a, 2-a\}$. In fact, if $a+m$ is even, multiplying both sides of (\ref{r_and_s_relation}) by $a-m$ and substituting into the equation the value of $s$ obtained from (\ref{second-equations-3}), leads to 
\begin{equation*}
\begin{aligned}
& 2(a-m)\left(\lambda + \frac{a-m-2}{2}\right)r + \left((m+1)\lambda+ ma -m-2\right)\left((1-m)\lambda -ma +m \right)r\\
& = -(m+1)(m-1)(\lambda +a)(\lambda +a-2)r.
\end{aligned}
\end{equation*}
For $m >1$, this is zero if and only if $\lambda \in \{-a, 2-a\}$ or $r = 0$. However, if $r = 0$, then $s = p = q = 0$ by (\ref{second-equations-3}), (\ref{second-equations-1}) and (\ref{second-equations-2}) respectively, which contradicts the hypothesis $(p, q, r, s) \neq (0, 0, 0, 0)$. The case $a + m$ odd can be checked analogously.\\

Clearly $\lambda$ can be $1-a$, since (\ref{second-equations-3}) and (\ref{second-equations-4}) are trivially satisfied for $\lambda = 1-a$. Therefore, $(\lambda, m)$ belongs to Case (II) for $a \geq m+4$.

Now, we show that $(p,q,r,s)$ satisfies (\ref{coefficients-ps-special-case-1}), (\ref{coefficients-ps-special-case-2}), and (\ref{coefficients-qs}) for Cases (I) and (II) (we remark that only (\ref{coefficients-qs}) is relevant for Case (II)).

Suppose $a \in \{m+1, m+2, m+3\}$. Then, a direct computation using (\ref{p_and_q_relation}), (\ref{r_and_s_relation}), (\ref{second-equations-1}) and (\ref{second-equations-2}) shows that (\ref{coefficients-ps-special-case-1}), (\ref{coefficients-ps-special-case-2}) and (\ref{coefficients-qs}) hold for $(\lambda, m)$ in both cases.

Suppose $a \geq m+4$. If $m=1$ (and $\lambda \in \C$), or $m>1$ and $\lambda \in \{-a, 2-a\}$, then a direct computation shows that the identities (\ref{second-equations-1}), (\ref{second-equations-2}) and (\ref{second-equations-3}) can be simplified as
\begin{equation}\label{first-equations-1-again}
2i\gamma(\lambda-1, a+m)p + \left(m(\lambda + a -1) -\lambda +2\right)q = 0,
\end{equation}
\begin{equation}\label{coefficients-qs-again}
q + \frac{\Gamma\left(\lambda + \left[\frac{a+m}{2}\right]\right)}{\Gamma\left(\lambda + \left[\frac{a-m}{2}\right]\right)}s = 0,
\end{equation}
\begin{equation}\label{first-equations-2-again}
2i\gamma(\lambda-1, a-m)r + \left(m(\lambda + a -1) +\lambda -2\right)s = 0,
\end{equation}
unless $(\lambda, m, a)$ belongs to one of the following three cases:
\begin{enumerate}[leftmargin=2cm]
\item[Case (a):] $(\lambda, m) = (-\frac{a-1}{2}, 1)$ and $a \in 2\N+1$.
\item[Case (b):] $(\lambda, m) = (-\frac{a}{2}, 1)$ and $a \in 2\N$.
\item[Case (c):] $(\lambda, m) = (-\frac{a-3}{2}, 1)$ and $a \in 2\N+1$.
\end{enumerate}
In these cases the identities (\ref{second-equations-1}), (\ref{second-equations-2}), (\ref{second-equations-3}) are equivalent to the following ones:
\begin{equation}\label{equations-cases-abc}
\begin{array}{r r}
\text{Case (a):} & ir + s = q = 2ip - (a+1)s = 0.\\[4pt]
\text{Case (b):} & 2ir - 3s =
2ip + q + as = 0. \\[4pt]
\text{Case (c):} & ir + s = q + s = 2ip + (a+1)q = 0.
\end{array}
\end{equation}

Note that (\ref{first-equations-1-again}) and (\ref{first-equations-2-again}) are precisely (\ref{p_and_q_relation}) and (\ref{r_and_s_relation}), which are satisfied by assumption; and that (\ref{coefficients-qs-again}) is precisely (\ref{coefficients-qs}). Thus, if $m>1$ and $\lambda \in \{-a, 2-a\}$, (\ref{coefficients-qs}) holds. If $\lambda = 1-a$, one can check that (\ref{second-equations-2}) coincides with (\ref{coefficients-qs}).

By the same reasoning, if $m =1$ (and $\lambda \in \C$), (\ref{coefficients-qs}) is also satisfied unless one of the cases (a)--(c) holds. In case (b), although (\ref{first-equations-1-again}) cannot be deduced from $A_\varepsilon = A_{\varepsilon + 2} = A_{\varepsilon + 4} = 0$, it holds by assumption (by (\ref{p_and_q_relation})), which together with (\ref{equations-cases-abc})(b) can be proved to imply (\ref{coefficients-qs}). In cases (a) and (c), by (\ref{equations-cases-abc})(a) and (\ref{equations-cases-abc})(c), (\ref{coefficients-ps-special-case-1}) and (\ref{coefficients-ps-special-case-2}) are satisfied respectively. Thus, we have proved that if the equation (\ref{the-equation-pqrs-twosided}) is satisfied, then $(\lambda, m)$ belongs to Cases (I) or (II), and $(p,q,r,s)$ satisfies (\ref{coefficients-ps-special-case-1}), (\ref{coefficients-ps-special-case-2}) and (\ref{coefficients-qs}).\\

(ii) $\Rightarrow$ (i) Now we suppose that $(\lambda, m)$ belongs to Cases (I) or (II) and that (\ref{coefficients-ps-special-case-1}), (\ref{coefficients-ps-special-case-2}) and (\ref{coefficients-qs}) hold. We wish to show that $(p,q,r,s)$ satisfies (\ref{the-equation-pqrs-twosided}).\\

\underline{Case (I)}. Suppose $m = 1$ and suppose in addition that $\gamma(\lambda, a-1) \neq 0$, i.e., $a\in 2\N$, or $a \in 2\N+1$ and $\lambda \neq -\frac{a-1}{2}$. Then, by using (\ref{p_and_q_relation}), the left-hand side of (\ref{the-equation-pqrs-twosided}) amounts to 
\begin{equation*}
\begin{aligned}
(a+1)^{-1}p\left((a+1)\Geg_{a+1}^{\lambda-1}(it) -2i\gamma(\lambda, a-1)\Geg_{a}^\lambda(it)\right) & = -2(a+1)^{-1}p\Geg_{a-1}^\lambda(it) \\
& = -i\gamma(\lambda, a-1)^{-1}q\Geg_{a-1}^\lambda(it),
\end{aligned}
\end{equation*}
where in the first equality we used the identity (\ref{KKP-2}).
On the other hand, observe that (\ref{coefficients-qs}) amounts to $q + \left(\lambda + \left[\frac{a-1}{2}\right]\right)s = 0$, which is satisfied even if $a \in 2\N+1$ and $\lambda = - \frac{a-3}{2}$ by (\ref{coefficients-ps-special-case-2}). Thus, by $(\ref{coefficients-qs})$ and by the relation $\gamma(\lambda, a-1)\gamma(\lambda, a-2) = \left(\lambda + \left[\frac{a-1}{2}\right]\right)$, the left-hand side of (\ref{the-equation-pqrs-twosided}) amounts to
\begin{equation}\label{left-hand-side-m=1}
i\gamma(\lambda, a-2)s\Geg_{a-1}^\lambda(it).
\end{equation}
Now, observe that (\ref{r_and_s_relation}) together with (\ref{coefficients-ps-special-case-2}) can be rewritten as $ir + \gamma(\lambda-\frac{1}{2}, a-2)s = 0$. Thus, by (\ref{KKP-3}), the right-hand side of (\ref{the-equation-pqrs-twosided}) amounts to
\begin{equation}\label{right-hand-side-m=1}
s\left(i\gamma(\lambda-\frac{1}{2}, a-2)\Geg_{a-1}^{\lambda-1}(it) - t\Geg_{a-2}^\lambda(it)\right) = i\gamma(\lambda, a-2)s\Geg_{a-1}^\lambda(it),
\end{equation}
which is precisely (\ref{left-hand-side-m=1}).

If $\gamma(\lambda, a-1) = 0$, i.e., $a\in 2\N+1$ and $\lambda = -\frac{a-1}{2}$, then, by (\ref{KKP-2}) and (\ref{coefficients-ps-special-case-1}), the left-hand side of (\ref{the-equation-pqrs-twosided}) amounts to
\begin{equation*}
\begin{aligned}
(a+1)^{-1}p\left((a+1)\Geg_{a+1}^{\lambda-1}(it) -2i\gamma(\lambda, a-1)\Geg_{a}^\lambda(it)\right) & = -2(a+1)^{-1}p\Geg_{a-1}^\lambda(it) \\
& = is\Geg_{a-1}^\lambda(it).
\end{aligned}
\end{equation*}
On the other hand, (\ref{p_and_q_relation}) coincides with $ir + \gamma(\lambda-\frac{1}{2}, a-2)s = 0$. Thus, as in the previous case, by using (\ref{KKP-3}) the right-hand side of (\ref{the-equation-pqrs-twosided}) is (\ref{right-hand-side-m=1}), which clearly coincides with the left-hand side since $a\in 2\N+1$.\\
 
\underline{Case (II)}. Suppose now $m>1$ and $\lambda \in \{-a, 1-a, 2-a\}$. If $\lambda = 2-a$, observe that since $a \geq m+1$, then $\gamma(\lambda, a+m-2)\neq 0$. As in the previous case, by using (\ref{KKP-2}) and (\ref{p_and_q_relation}), the left-hand side of (\ref{the-equation-pqrs-twosided}) amounts to
\begin{equation*}
\begin{aligned}
&(a+m)^{-1}p\left((a+m)\Geg_{a+m}^{\lambda-1}(it) -2i\gamma(\lambda-1, a+m)\Geg_{a+m-1}^\lambda(it)\right) \\
& = -2(a+m)^{-1}p\Geg_{a-1}^\lambda(it)\\
& = -i\gamma(\lambda-1, a+m)^{-1}q\Geg_{a-1}^\lambda(it).
\end{aligned}
\end{equation*}
Now, by (\ref{identity-KOSS}), we have 
\begin{equation*}
\frac{\Gamma\left(\lambda + \left[\frac{a+m-1}{2}\right]\right)}{\Gamma\left(\lambda + \left[\frac{a-m-1}{2}\right]\right)}\Geg_{a+m-2}^\lambda(it) = \Geg_{a-m-2}^\lambda(it).
\end{equation*}
This, together with (\ref{coefficients-qs}), (\ref{r_and_s_relation}) and the following equality
\begin{equation*}
\begin{gathered}
\frac{\Gamma\left(\lambda + \left[\frac{a+m}{2}\right]\right)}{\Gamma\left(\lambda + \left[\frac{a-m}{2}\right]\right)} \gamma(\lambda, a-m-2)
= \frac{\Gamma\left(\lambda + \left[\frac{a+m-1}{2}\right]\right)}{\Gamma\left(\lambda + \left[\frac{a-m-1}{2}\right]\right)}\gamma(\lambda, a+m-2),
\end{gathered}
\end{equation*}
show that the left-hand side of (\ref{the-equation-pqrs-twosided}) can be written as
\begin{equation}
-2(a-m)^{-1}r\Geg_{a-m-2}^\lambda(it).
\end{equation}
On the other hand, by (\ref{r_and_s_relation}), the right-hand side of (\ref{the-equation-pqrs-twosided}) amounts to
\begin{equation*}
(a-m)^{-1}r\left(\Geg_{a-m}^{\lambda-1}(it) - 2i\gamma(\lambda-1, a-m)t\Geg_{a-m-1}^\lambda(it)\right).
\end{equation*}
Now, both sides clearly coincide by (\ref{KKP-2}).

If $\lambda = -a$, the argument is similar but we apply it in different order. We first apply (\ref{identity-KOSS}) to the left-hand side of (\ref{the-equation-pqrs-twosided}) and then simplify it by using Lemma \ref{Lemma-KKP-identities}. First, by (\ref{identity-KOSS}) we have
\begin{equation}\label{proof-equation-KOSS}
\frac{\Gamma\left(\lambda + \left[\frac{a+m-1}{2}\right]\right)}{\Gamma\left(\lambda + \left[\frac{a-m+1}{2}\right]\right)} \Geg_{a+m}^{\lambda-1}(it) = \Geg_{a-m+2}^{\lambda-1}(it), \quad \frac{\Gamma\left(\lambda + \left[\frac{a+m}{2}\right]\right)}{\Gamma\left(\lambda + \left[\frac{a-m+2}{2}\right]\right)} \Geg_{a+m-1}^{\lambda}(it) = \Geg_{a-m+1}^{\lambda}(it).
\end{equation}
Now, by using (\ref{p_and_q_relation}), (\ref{coefficients-qs}), and the equality
\begin{equation*}
\frac{\Gamma\left(\lambda + \left[\frac{a+m-1}{2}\right]\right)}{\Gamma\left(\lambda + \left[\frac{a-m+1}{2}\right]\right)} \gamma(\lambda, a+m-2) = 
\frac{\Gamma\left(\lambda + \left[\frac{a+m}{2}\right]\right)}{\Gamma\left(\lambda + \left[\frac{a-m+2}{2}\right]\right)} \gamma(\lambda, a-m),
\end{equation*}
the left-hand side of (\ref{the-equation-pqrs-twosided}) amounts to
\begin{equation*}
\begin{aligned}
&-i\gamma(\lambda, a+m-2)s\left((a-m+2)\Geg_{a+m}^{\lambda-1} -2i\gamma(\lambda, a-m)t\Geg_{a-m+1}^{\lambda}(it)\right)\\
&= i\gamma(\lambda, a-m-1)s\Geg_{a-m}^\lambda(it),
\end{aligned}
\end{equation*}
where in the last equality we used (\ref{KKP-2}). On the other hand, observe that (\ref{r_and_s_relation}) can be rewritten as
\begin{equation*}
ir + \gamma(\lambda - \frac{1}{2}, a-m-1)s = 0.
\end{equation*}
Thus, the right-hand side of (\ref{the-equation-pqrs-twosided}) amounts to
\begin{equation*}
s\left(i\gamma(\lambda - \frac{1}{2}, a-m-1)\Geg_{a-m}^{\lambda-1}(it) - t\Geg_{a-m-1}^\lambda(it) \right),
\end{equation*}
which coincides with the left-hand side by (\ref{KKP-3}).

For the last case, namely $\lambda = 1-a$, we use a different approach. A direct computation shows that the terms $A_{a+m-2k}$ satisfy the following recursive relation:
\begin{equation*}
A_{a+m-2(k-1)} = C_1\cdot A_{a+m-2k} +
C_2 \cdot \Big
[q + \frac{\Gamma\left(\lambda + \left[\frac{a+m}{2}\right]\right)}{\Gamma\left(\lambda + \left[\frac{a-m}{2}\right]\right)}s\Big],
\end{equation*}
for $k = m+1, \dots, \left[\frac{a+m}{2}\right]$, where $C_1$ and $C_2$ are appropiate non-zero constants depeding on $k, a$ and $m$. Thus, (\ref{the-equation-pqrs-twosided}) is satisfied as soon as $A_\varepsilon = 0$ and (\ref{coefficients-qs}) holds, but the condition $A_\varepsilon = 0$ can be easily proved from (\ref{p_and_q_relation}), (\ref{r_and_s_relation}) and (\ref{coefficients-qs}) by Lemma \ref{lemma-lowerterms} (i). Hence, (\ref{the-equation-pqrs-twosided}) holds.
\end{proof}

\subsection{Proof of Theorem \ref{Thm-solvingequations}}
\label{section-proof_of_solving_equations-subsection}
Now, we are ready to show Theorem \ref{Thm-solvingequations}.
\begin{proof}[Proof of Theorem \ref{Thm-solvingequations}]
Take polynomials
\begin{equation}\label{polynomial-even-order}
f_j \in \Pol_{a-m+1-j}[t]_\text{{\normalfont{even}}}, \enspace \text{ for } j=0,1,2.
\end{equation} 
We show that $(f_0, f_1, f_2)$ satisfies equations $(E_1)$-$(E_6)$ (see the beginning of Section \ref{section-proof_of_solving_equations}) if and only if the triple $(\lambda, a, m)$ belongs to one of the cases (1)--(3) of the statement. Moreover, we prove that if $(f_0, f_1, f_2) \in \Xi_{\lambda, a, m}$, then
\begin{equation*}
(f_0, f_1, f_2)= \alpha (g_{m-1}, g_m, g_{m+1}),
\end{equation*}
for some $\alpha \in \C$, where $(g_{m-1}, g_m, g_{m+1})$ is the triple given in (\ref{solution-all}).





If $a < m-1$, then $a-m+1-j < m-1-m+1 -j = -j \leq 0$. Thus, $\Pol_{a-m+1-j}[t]_\text{even} = \{0\}$. Then, it suffices to consider the case $a \geq m-1$. We divide the proof into the following three cases:
\begin{enumerate}[leftmargin=1cm, topsep=0pt]
\item[(1)] $a = m-1$.
\item[(2)] $a = m$.
\item[(3)] $a \geq m+1$.
\end{enumerate}
\vspace{0.1cm}
We start by (1) and consider next (3). The case (2) will be done at the end.

(1) Suppose $a = m-1$. By (\ref{polynomial-even-order}), it is clear that, a priori, any solution has to be a constant multiple of $(1, 0, 0)$. Now, any multiple of this tuple satisfies $(E_k) \enspace k = 2,3,4,6$ trivially and it clearly satisfies $(E_1)$ if and only if $(2m-2)(2m-2+2\lambda) = 0$, or equivalently, if $m = 1$ (and $\lambda \in \C$), or $m > 1$ and $\lambda = 1-m = -a$. A similar argument holds for $(E_5)$. Thus, if $a = m-1$, the system has a non-zero solution only in the cases (1.I) $m = 1$ and $\lambda \in \C$, or (1.II) $m >1$ and $\lambda = -a$, and the solution up to multiple is $(1,0,0)$.

(3) Suppose $a \geq m+1$. By (\ref{expression_f0-and-f2}) and Lemmas \ref{lemma-step-1} and \ref{lemma-step-2}, we know that $(f_0, f_1, f_2)$ satisfies $(E_1)$--$(E_6)$ if and only if
\begin{equation}\label{expression_f0-and-f2-again}
f_0(t) = q\Geg_{a+m-1}^\lambda(it), \quad f_2(t) = s\Geg_{a-m-1}^\lambda(it),
\end{equation}
\begin{equation}\label{expression_f1_1_again}
f_1(t) = -p\Geg_{a+m}^{\lambda-1}(it) -qt\Geg_{a+m-1}^\lambda(it)
\end{equation}
\begin{equation}\label{expression_f1_2_again}
\hspace{0.9cm} = -r\Geg_{a-m}^{\lambda -1}(it) + st\Geg_{a-m-1}^\lambda(it),
\end{equation}
where $p, q, r, s \in \C$ satisfy (\ref{p_and_q_relation}) and (\ref{r_and_s_relation}).
Moreover, by Lemma \ref{lemma-equation-pqrs}, the two expressions for $f_1$ in (\ref{expression_f1_1_again}) and (\ref{expression_f1_2_again}) coincide if and only if $(\lambda, m)$ belongs to Case (3.I) or Case (3.II) and $(p,q,r,s)$ satisfies (\ref{coefficients-ps-special-case-1}), (\ref{coefficients-ps-special-case-2}) and (\ref{coefficients-qs}).\\
Suppose $(\lambda, m)$ belongs to Case (3.I); that is, $m = 1$ and $\lambda \in \C$. Observe that by (\ref{p_and_q_relation}) and (\ref{coefficients-ps-special-case-1}) we have
\begin{equation*}
-2(a+1)^{-1}p = \begin{cases}
\left(i\gamma(\lambda-1, a+1)\right)^{-1}q, & \text{if } a\in 2\N \text{ or } \lambda \neq -\frac{a-1}{2},\\
is, & \text{if } a\in 2\N+1 \text{ and } \lambda = -\frac{a-1}{2}.
\end{cases}
\end{equation*}
By using (\ref{coefficients-qs}) and the identity $\gamma(\lambda-1, a)\gamma(\lambda-1, a+1) = \lambda + \left[\frac{a-1}{2}\right]$, $-2(a+1)^{-1}p$ can be written uniformly as
\begin{equation*}
-2(a+1)^{-1}p = i \gamma(\lambda-1,a)s.
\end{equation*}
Then, by (\ref{expression_f1_1_again}), (\ref{p_and_q_relation}) and by using (\ref{KKP-2}), we deduce $f_1(t) = -i\gamma(\lambda-1, a)s\Geg_{a-1}^\lambda(it)$. 
Thus, by (\ref{expression_f0-and-f2-again}) and (\ref{coefficients-qs}) we have
\begin{equation*}
\begin{aligned}
\left(f_0(t), f_1(t), f_2(t)\right) & = s\left(-\left(\lambda + \left[\frac{a-1}{2}\right]\right)\Geg_a^\lambda(it), -i\gamma(\lambda-1, a)\Geg_{a-1}^\lambda(it), \Geg_{a-2}^\lambda(it)\right)\\[4pt]
& = s\left(g_{m-1}(t), g_m(t), g_{m+1}(t)\right).
\end{aligned}
\end{equation*}

Suppose now that $(\lambda, m)$ belongs to Case (3.II); that is, $m > 1$ and $\lambda \in \{-a, 1-a, 2-a\}$. Observe that since $a \geq m+1 > 2$, we have $\gamma(\lambda-1, a-m) \neq 0$. Then by (\ref{expression_f0-and-f2-again}), (\ref{expression_f1_2_again}),  (\ref{r_and_s_relation}) and (\ref{coefficients-qs}) we have
\begin{equation*}
\begin{aligned}
f_0(t) &=  2i\gamma(\lambda, a-m-2)\alpha\frac{\Gamma\left(\lambda + \left[\frac{a+m}{2}\right]\right)}{\Gamma\left(\lambda + \left[\frac{a-m}{2}\right]\right)}\widetilde{C}^\lambda_{a+m-1}(it), \\
f_1(t) &= \alpha\left(-(m(\lambda + a-1) + \lambda -2) \Geg_{a-m}^{\lambda-1}(it) - 2i\gamma(\lambda-1, a-m)t \Geg_{a-m-1}^\lambda(it)\right),\\
f_2(t) &= -2i\gamma(\lambda, a-m-2)\alpha\widetilde{C}^\lambda_{a-m-1}(it),
\end{aligned}
\end{equation*}
where $\alpha = \left(-2i\gamma(\lambda-1, a-m)\right)^{-1}s$.
Now, simplifying $f_0$ by using (\ref{identity-KOSS}), the triple above amounts to
\begin{equation*}
\begin{aligned}
&(f_0(t), f_1(t), f_2(t)) \\
& = \alpha \left(-iAB  \widetilde{C}^\lambda_{a-m+1-2\nu}(it), -C\Geg_{a-m}^{\lambda-1}(it) + iBt \Geg_{a-m-1}^\lambda(it), iB \widetilde{C}^\lambda_{a-m-1}(it)\right) \\
& = \alpha \left(g_{m-1}(t), g_m(t), g_{m+1}(t)\right),
\end{aligned}
\end{equation*}
where $A, B$ and $C$ are the constants given in (\ref{const-A}), (\ref{const-B}) and (\ref{const-C}), respectively.\\

(2) Suppose $a = m$. The argument is essentially the same as in the case $a \geq m+1$. However, instead of using Lemma \ref{lemma-equation-pqrs}, we prove (\ref{the-equation-pqrs-twosided}) directly by doing some computations. As before, by Lemma \ref{lemma-step-1} and Lemma \ref{lemma-step-2} (1), $(f_0, f_1, f_2)$ satisfies $(E_1)$--$(E_6)$ if and only if $f_0, f_2$ are given by (\ref{expression_f0-and-f2-again}) and $f_1$ is given by (\ref{expression_f1_1_again}) and (\ref{expression_f1_2_again}), for some constants $p, q, r$ satisfying (\ref{p_and_q_relation}). Note that $f_2 = 0$ (i.e. $s = 0$). Now, by Lemma \ref{lemma-lambda-sets}, a necessary condition such that (\ref{expression_f1_1_again}) and (\ref{expression_f1_2_again}) 
coincide is that $m = 1$ and $\lambda \in\C$, or $m > 1$ and one of the conditions ($\star_1$)--($\star_4$) in Lemma \ref{lemma-lambda-sets} is satisfied.\\
If $m = 1$ and $\lambda \in \C$, by (\ref{p_and_q_relation}), (\ref{expression_f1_1_again}) amounts to
\begin{equation*}
f_1(t) = p\left(-\Geg_2^{\lambda-1}(it) +i\lambda t\Geg_1^\lambda(it)\right),
\end{equation*}
which coincides with (\ref{expression_f1_2_again}) if and only if $r = -p$ by (\ref{KKP-2}). Thus,
\begin{equation*}
\begin{aligned}
\left(f_0(t), f_1(t), f_2(t)\right) & = \left(-i\lambda p\Geg_1^\lambda(it), p, 0\right)\\
& = ip\left(g_{0}(t), g_1(t), g_{2}(t)\right).
\end{aligned}
\end{equation*}
Suppose now $m > 1$. Since $\Lambda_{a-m+2}^m = \Lambda_2^m = \emptyset$, and since if $(p,q) = (0,0)$, then there is no non-zero solution of the system, we only have the cases $(\star_2)$ and $(\star_3)$, which correspond to the cases $\lambda = -a$ and $q = \lambda + m-1 = 0$. In the case $\lambda = -a$, by (\ref{p_and_q_relation}), (\ref{expression_f1_2_again}) amounts to
\begin{equation*}
f_1(t) = -p\left(\Geg_{2m}^{-m-1}(it) + it\Geg_{2m-1}^{-m}(it)\right).
\end{equation*}
Following a similar strategy as in the proof of (ii) $\Rightarrow$ (i) in Lemma \ref{lemma-equation-pqrs}, by using (\ref{identity-KOSS}) and (\ref{KKP-2}), we have
\begin{equation*}
f_1(t) = \frac{(-1)^{m}p}{m!}\left(\Geg_{2}^{-m-1}(it) + imt\Geg_{1}^{-m}(it)\right) = \frac{(-1)^{m+1}p}{m!},
\end{equation*}
which coincides with (\ref{expression_f1_2_again}) if and only if $p + (-1)^{m+1}m! r = 0$. Now, by using (\ref{identity-KOSS})
again with $f_0$, we have
\begin{equation*}
\begin{aligned}
(f_0(t), f_1(t), f_2(t)) 
& = \frac{(-1)^{m+1} p}{m!}\left(im\Geg_{1}^{-m}(it), 1, 0\right)\\[4pt]
& = \frac{(-1)^{m+1} p}{2(m+1)!}\left(g_{m-1}(t), g_m(t), g_{m+1}(t)\right).
\end{aligned}
\end{equation*}

In the case $q = 0$ and $\lambda = 1-m = 1-a$, we have $f_0 = f_2 = 0$ and (\ref{expression_f1_1_again}) amounts to
\begin{equation*}
f_1(t) = -p\Geg_{2m}^{-m}(it) = \frac{(-1)^{m+1} p}{m!},
\end{equation*}
where in the second equality we used (\ref{identity-KOSS}). Thus, again, (\ref{the-equation-pqrs-twosided}) holds if and only if $p + (-1)^{m+1}m! r = 0$ and
\begin{equation*}
\begin{aligned}
(f_0(t), f_1(t), f_2(t)) 
& = \frac{(-1)^{m+1} p}{m!}\left(0, 1, 0\right)\\[4pt]
& = \frac{(-1)^{m+1} p}{(m+1)!}\left(g_{m-1}(t), g_m(t), g_{m+1}(t)\right).
\end{aligned}
\end{equation*}
Hence, also in the case $a = m$, $(f_0, f_1, f_2)$ is a solution of the system if and only if it is a multiple of $\left(g_{m-1}(t), g_m(t), g_{m+1}(t)\right)$ given in (\ref{solution-all}) and $(\lambda, m)$ belongs to Case (2.I) or Case (2.II).
\end{proof}
\begin{rem}
In the proof above we obtained $f_1$ by using the left-hand side of (\ref{the-equation-pqrs-twosided}) when $m=1$, and the right-hand side when $m>1$. The reason for doing this distinction is to obtain the simplest formula for the solution of the system. In fact, as the expression below shows, a single formula for the solution can be obtained by using the right-hand side of (\ref{the-equation-pqrs-twosided}) in both $m = 1$ and $m>1$ cases. The disadvantage of doing this is that the expression is more complicated than the one given in (\ref{solution-all}).
\begin{equation}\label{solution-1}
\left(-i\widetilde{A}\widetilde{B} \widetilde{C}^\lambda_{a+m-1}(it), -\widetilde{C} \Geg_{a-m}^{\lambda-1}(it) + i\widetilde{B}t \Geg_{a-m-1}^\lambda(it), i\widetilde{B} \widetilde{C}^\lambda_{a-m-1}(it)\right)
\end{equation}
Here $\widetilde{A}, \widetilde{B}, \widetilde{C}$ are the following constants:
\begin{equation*}
\begin{aligned}
\widetilde{A} := & \enspace \frac{\Gamma\left(\lambda + \left[\frac{a+m}{2}\right]\right)}{\Gamma\left(\lambda + \left[\frac{a-m}{2}\right]\right)},\\[4pt]
\widetilde{B} := & \enspace \begin{cases}
-1, &\text{if } m =1,\enspace a \in 2\N+1, \text{ and } \lambda = -\frac{a-3}{2},\\
-2\gamma(\lambda-1, a-m), &\text{otherwise},
\end{cases}\\
\widetilde{C}:= & \enspace \begin{cases}
1, &\text{if } m =1,\enspace a \in 2\N+1, \text{ and } \lambda = -\frac{a-3}{2},\\
m(\lambda+a-1) + \lambda -2, &\text{otherwise}.
\end{cases}
\end{aligned}
\end{equation*}
\end{rem}

\section{The case $m\leq -1$}\label{section-case_m_lessthan_-1}
In this section we prove that there is a duality between the F-systems for $m \geq 1$ and for $m \leq -1$ (Proposition \ref{prop-duality_m_and_-m}). This result, together with Theorem \ref{Thm-step2}, proves Theorems \ref{mainthm1} and \ref{mainthm2} for $m \leq -1$. We start by fixing some notation.

Let $m \in -\N_+$ and write $m = -p$ for $p \in \N_+$. Given $a\in \N$, we define
\begin{equation*}
\psi^\pm: \bigoplus_{k = p-1}^{p+1} \Pol_{a-k}[t]_\text{{\normalfont{even}}} \rightarrow \Hom_{SO(2)}\left(V^3, \C_{\pm p} \otimes \Pol^a(\n_+)\right)
\end{equation*}
by
\begin{equation*}
\psi^\pm(g_{p-1}, g_p, g_{p+1})(\zeta) = \sum_{k = p-1}^{p+1}\left(T_{a-k}g_k\right)(\zeta)h_k^\pm,
\end{equation*}
where $h_k^\pm$ are the generators defined in (\ref{generators-Hom(V3,CmHk)}). In coordinates we have
\begin{equation*}
\begin{gathered}
\psi^+(g_{p-1}, g_{p}, g_{p+1})(\zeta) =
\begin{pmatrix}
\displaystyle{(T_{a-p+1}g_{p-1})(\zeta) (\zeta_1 + i\zeta_2)^{p-1}}\\[4pt]
\displaystyle{(T_{a-p}g_{p})(\zeta) (\zeta_1 + i\zeta_2)^{p}}\\[4pt]
\displaystyle{(T_{a-p-1}g_{p+1})(\zeta) (\zeta_1 + i\zeta_2)^{p+1}}
\end{pmatrix},\\[4pt]
\psi^-(g_{
p-1}, g_{p}, g_{p+1})(\zeta) =
\begin{pmatrix}
\displaystyle{(T_{a-p-1}g_{p+1})(\zeta) (\zeta_1 - i\zeta_2)^{p+1}}\\[4pt]
\displaystyle{(T_{a-p}g_{p})(\zeta) (\zeta_1 - i\zeta_2)^{p}}\\[4pt]
\displaystyle{(T_{a-p+1}g_{p-1})(\zeta) (\zeta_1 - i\zeta_2)^{p-1}}
\end{pmatrix}.
\end{gathered}
\end{equation*}
For simplicity, for a given $(g_{p-1}, g_p, g_{p+1})$ we also write
\begin{equation*}
\psi^\pm(\zeta) = \psi^\pm(g_{p-1}, g_p, g_{p+1})(\zeta). 
\end{equation*}

Let $\Phi$ be an involution on $\Hom_\C\left(V^3, \Pol(\n_+)\right)$ defined by
\begin{equation*}
\begin{gathered}
\Phi(\varphi)(\zeta_1, \zeta_2, \zeta_3) = 
\begin{pmatrix}
& & 1\\
& -1 &\\
1 & &
\end{pmatrix} \varphi(\zeta_1, -\zeta_2, \zeta_3),
\end{gathered}
\end{equation*}
for $\varphi(\zeta) = {}^t(\varphi_1(\zeta),\varphi_2(\zeta), \varphi_3(\zeta)) \in \Hom_\C\left(V^3, \Pol(\n_+)\right)$.
Note that for $\varphi = \psi^+(\zeta) \equiv \psi^+(g_{p-1}, g_p, g_{p+1})(\zeta)$, we have
\begin{equation}\label{relation-Phi(psi+)-and-psi-}
\Phi(\psi^+)(\zeta_1, \zeta_2, \zeta_3) = \psi^-(g_{p-1}, -g_p, g_{p+1})(\zeta_1, \zeta_2, \zeta_3).
\end{equation}

\begin{lemma}\label{lemma-duality_m_and_-m}
Let $C_1^+$ be the element of $\n_+$ defined in {\normalfont{(\ref{elements})}}. Then, we have
\begin{equation}\label{interwtining-property}
\Phi\left(\left(\widehat{d \pi_{(\sigma^3, \lambda)^*}}(C_1^+) \otimes \id_{\C_{p,\nu}}\right)\psi^+\right) = \left(\widehat{d \pi_{(\sigma^3, \lambda)^*}}(C_1^+) \otimes \id_{\C_{-p,\nu}}\right)\Phi(\psi^+).
\end{equation}
\end{lemma}
\begin{proof}
Let $M_s(\psi^+)(\zeta)$ and $M_s(\Phi(\psi^+))(\zeta)$ denote the vector coefficients of
\begin{equation*}
\begin{gathered}
\left(\widehat{d \pi_{(\sigma^3, \lambda)^*}}(C_1^+) \otimes \id_{\C_{p,\nu}}\right)\psi^+
\end{gathered}
\end{equation*}
and
\begin{equation*}
\left(\widehat{d \pi_{(\sigma^3, \lambda)^*}}(C_1^+) \otimes \id_{\C_{-p,\nu}}\right)\Phi(\psi^+)
\end{equation*}
respectively (see (\ref{vector-coefficients-scalar}), (\ref{vector-coefficients-vect})). Then, in order to prove the lemma it suffices to show
\begin{equation*}
\begin{pmatrix}
M_3(\psi^+)(\zeta_1, -\zeta_2, \zeta_3)\\[6pt]
-M_2(\psi^+)(\zeta_1, -\zeta_2, \zeta_3)\\[6pt]
M_1(\psi^+)(\zeta_1, -\zeta_2, \zeta_3)
\end{pmatrix} = 
\begin{pmatrix}
M_1(\Phi(\psi^+))(\zeta_1, \zeta_2, \zeta_3)\\[6pt]
M_2(\Phi(\psi^+))(\zeta_1, \zeta_2, \zeta_3)\\[6pt]
M_3(\Phi(\psi^+))(\zeta_1, \zeta_2, \zeta_3)\\
\end{pmatrix}.
\end{equation*}
Since $M_s = M_s^\text{{\normalfont{scalar}}} + M_s^\text{{\normalfont{vect}}}$, we consider the scalar parts $M_s^\text{{\normalfont{scalar}}}$ and the vector parts $M_s^\text{{\normalfont{vect}}}$ of each side of the identity above separately.

1) \underline{$M_s^\text{{\normalfont{scalar}}}$}: For $\varphi(\zeta_1, \zeta_2, \zeta_3) \in \Hom_\C\left(V^3, \Pol(\n_+)\right)$ we let
\begin{equation*}
(\alpha_2\varphi)(\zeta_1, \zeta_2, \zeta_3) = \varphi(\zeta_1, -\zeta_2, \zeta_3).
\end{equation*}
Then, we have
\begin{equation*}
\begin{gathered}
\Delta_{\C^3} \circ \alpha_2 = \alpha_2 \circ \Delta_{\C^3},\\
E_\zeta \circ \alpha_2 =  \alpha_2 \circ E_\zeta,
\end{gathered}
\end{equation*}
which implies that
\begin{equation*}
\widehat{d\pi_{(\sigma^3, \lambda)^*}^\text{{\normalfont{scalar}}}}(C_1^+) \circ \alpha_2 = \alpha_2 \circ \widehat{d\pi_{(\sigma^3, \lambda)^*}^\text{{\normalfont{scalar}}}}(C_1^+).
\end{equation*}
Therefore
\begin{equation*}
\begin{gathered}
\begin{pmatrix}
M_3^\text{{\normalfont{scalar}}}(\psi^+)(\zeta_1, -\zeta_2, \zeta_3)\\[6pt]
-M_2^\text{{\normalfont{scalar}}}(\psi^+)(\zeta_1, -\zeta_2, \zeta_3)\\[6pt]
M_1^\text{{\normalfont{scalar}}}(\psi^+)(\zeta_1, -\zeta_2, \zeta_3)
\end{pmatrix} = 
\alpha_2
\begin{pmatrix}
M_3^\text{{\normalfont{scalar}}}(\psi^+)(\zeta_1, \zeta_2, \zeta_3)\\[6pt]
- M_2^\text{{\normalfont{scalar}}}(\psi^+)(\zeta_1, \zeta_2, \zeta_3)\\[6pt]
 M_1^\text{{\normalfont{scalar}}}(\psi^+)(\zeta_1, \zeta_2, \zeta_3)
\end{pmatrix} \\[6pt]
=
\begin{pmatrix}
 M_3^\text{{\normalfont{scalar}}}(\alpha_2\psi^+)(\zeta_1, \zeta_2, \zeta_3)\\[6pt]
- M_2^\text{{\normalfont{scalar}}}(\alpha_2\psi^+)(\zeta_1, \zeta_2, \zeta_3)\\[6pt]
 M_1^\text{{\normalfont{scalar}}}(\alpha_2\psi^+)(\zeta_1, \zeta_2, \zeta_3)
\end{pmatrix} 
=
\begin{pmatrix}
 M_1^\text{{\normalfont{scalar}}}(\Phi(\psi^+))(\zeta_1, \zeta_2, \zeta_3)\\[6pt]
 M_2^\text{{\normalfont{scalar}}}(\Phi(\psi^+))(\zeta_1, \zeta_2, \zeta_3)\\[6pt]
 M_3^\text{{\normalfont{scalar}}}(\Phi(\psi^+))(\zeta_1, \zeta_2, \zeta_3)
\end{pmatrix}.
\end{gathered}
\end{equation*}

2) \underline{$M_s^\text{{\normalfont{vect}}}$}:
For the vector parts, by (\ref{Ms-vect}) the left-hand side amounts to
\begin{equation*}
\begin{pmatrix}
M_3^\text{{\normalfont{vect}}}(\psi^+)(\zeta_1, -\zeta_2, \zeta_3)\\[6pt]
-M_2^\text{{\normalfont{vect}}}(\psi^+)(\zeta_1, -\zeta_2, \zeta_3)\\[6pt]
M_1^\text{{\normalfont{vect}}}(\psi^+)(\zeta_1, -\zeta_2, \zeta_3)
\end{pmatrix}
= \begin{pmatrix}
 -2 \frac{\partial}{\partial \zeta_3}\psi_2^+(\zeta_1, -\zeta_2, \zeta_3) + 2i \frac{\partial}{\partial (-\zeta_2)}\psi_3^+(\zeta_1, -\zeta_2, \zeta_3)\\[6pt]
 \frac{\partial}{\partial \zeta_3}\left(\psi_1^+(\zeta_1, -\zeta_2, \zeta_3) - \psi_3^+(\zeta_1, -\zeta_2, \zeta_3)\right)
 \\[6pt] 
 -2i \frac{\partial}{\partial (-\zeta_2)}\psi_1^+(\zeta_1, -\zeta_2, \zeta_3) + 2 \frac{\partial}{\partial \zeta_3}\psi_2^+(\zeta_1, -\zeta_2, \zeta_3)
\end{pmatrix},
\end{equation*}
but by the definition of $\Phi$ and by using the direct formula $\frac{\partial}{\partial (-\zeta_2)} = -\frac{\partial}{\partial \zeta_2}$ this amounts to
\begin{equation*}
\begin{pmatrix}
 2 \frac{\partial}{\partial \zeta_3}\Phi(\psi^+)_2(\zeta_1, \zeta_2, \zeta_3) - 2i \frac{\partial}{\partial \zeta_2}\Phi(\psi^+)_1(\zeta_1, \zeta_2, \zeta_3)\\[6pt]
 \frac{\partial}{\partial \zeta_3}\left(\Phi(\psi^+)_3(\zeta_1, \zeta_2, \zeta_3) - \Phi(\psi^+)_1(\zeta_1, \zeta_2, \zeta_3)\right)
 \\[6pt] 
 2i \frac{\partial}{\partial \zeta_2}\Phi(\psi^+)_3(\zeta_1, \zeta_2, \zeta_3) - 2 \frac{\partial}{\partial \zeta_3}\Phi(\psi^+)_2(\zeta_1, \zeta_2, \zeta_3)
\end{pmatrix},
\end{equation*}
which again by (\ref{Ms-vect}) is nothing but
\begin{equation*}
\begin{pmatrix}
M_1^\text{{\normalfont{vect}}}(\Phi(\psi^+))(\zeta_1, \zeta_2, \zeta_3)\\[6pt]
M_2^\text{{\normalfont{vect}}}(\Phi(\psi^+))(\zeta_1, \zeta_2, \zeta_3)\\[6pt]
M_3^\text{{\normalfont{vect}}}(\Phi(\psi^+))(\zeta_1, \zeta_2, \zeta_3)\\
\end{pmatrix}.
\end{equation*}
Hence, the identity (\ref{interwtining-property}) holds.
\end{proof}
Now, observe that
\begin{equation*}
\Sol(\n_+; \sigma_\lambda^3, \tau_{\pm p, \nu}) \subset \Hom_\C\left(V^3, \Pol(\n_+)\right).
\end{equation*}
We write
\begin{equation*}
\Phi_\text{Sol}^+ := \Phi\big\rvert_{\Sol(\n_+; \sigma_\lambda^3, \tau_{p, \nu})}.
\end{equation*}
By Lemma \ref{lemma-duality_m_and_-m}, we have
\begin{equation*}
\Phi_\text{Sol}^+: \Sol(\n_+; \sigma_\lambda^3, \tau_{p, \nu}) \rightarrow \Sol(\n_+; \sigma_\lambda^3, \tau_{-p, \nu}).
\end{equation*}
Moreover, we have the following
\begin{prop}\label{prop-duality_m_and_-m} The map
\begin{equation}\label{duality-isomorphism}
\begin{gathered}
\Phi_{\normalfont{\text{Sol}}}^+: \Sol(\n_+; \sigma_\lambda^3, \tau_{p, \nu}) \arrowsimeq \Sol(\n_+; \sigma_\lambda^3, \tau_{-p, \nu}),
\end{gathered}
\end{equation}
is a linear isomorphism.
\end{prop}
\begin{proof}
Since $\Phi$ is an involution on $\Hom_\C\left(V^3, \Pol(\n_+)\right)$,
\begin{equation*}
\Phi_\text{Sol}^+ = \Phi\rvert_{\Sol(\n_+; \sigma_\lambda^3, \tau_{p, \nu})}
\end{equation*}
is clearly injective. To show the surjectivity, take $\varphi(\zeta) \in \Sol(\n_+; \sigma_\lambda^3, \tau_{-p, \nu})$. Since
\begin{equation*}
\Sol(\n_+; \sigma_\lambda^3, \tau_{-p, \nu}) \subset \Hom_{SO(2)}\left(V^3, \C_{-p, \nu}\otimes \Pol(\n_+)\right),
\end{equation*}
it follows from Proposition \ref{Firststep-prop} and from
 (\ref{relation-Phi(psi+)-and-psi-}) that
 there exists 
\begin{equation*}
\psi^+(g_{p-1}, g_p, g_{p+1}) \in \Hom_{SO(2)}\left(V^3, \C_{p, \nu}\otimes \Pol(\n_+)\right)
\end{equation*} 
such that $\Phi(\psi^+)(\zeta) = \varphi(\zeta)$. Now, by Lemma \ref{lemma-duality_m_and_-m}, we have
\begin{equation*}
\begin{aligned}
\Phi\left(\left(\widehat{d \pi_{(\sigma^3, \lambda)^*}}(C_1^+) \otimes \id_{\C_{p,\nu}}\right)\psi^+\right)(\zeta) & = \left(\widehat{d \pi_{(\sigma^3, \lambda)^*}}(C_1^+) \otimes \id_{\C_{-p,\nu}}\right)\Phi(\psi^+)(\zeta) \\
& =
\left(\widehat{d \pi_{(\sigma^3, \lambda)^*}}(C_1^+) \otimes \id_{\C_{-p,\nu}}\right)\varphi(\zeta) \\
& = 0 \quad (\text{by assumption}).
\end{aligned}
\end{equation*}
Since $\Phi$ is injective, this shows that
\begin{equation*}
\left(\widehat{d \pi_{(\sigma^3, \lambda)^*}}(C_1^+) \otimes \id_{\C_{p,\nu}}\right)\psi^+(\zeta) = 0,
\end{equation*}
that is, $\psi^+(\zeta) \in \Sol(\n_+; \sigma_\lambda^3, \tau_{p, \nu})$. Now the proposition follows.
\end{proof}

\begin{rem}
(1) From Proposition \ref{prop-duality_m_and_-m} we can deduce that the space $\Sol(\n_+, \sigma_\lambda^3, \tau_{m, \nu})$ has dimension at most one, and that the necessary and sufficient condition on the parameters $(\lambda, \nu, m)$ such that it is non-zero is the same that the one in the case $m \geq 1$, where one should replace $m$ by $-m$. This proves Theorem \ref{mainthm1} for $m \leq -1$.

(2) As for the differential operators $\D_{\lambda, \nu}^m$ for $m \leq -1$, from the result above we can deduce that the operator $\D_{\lambda, \nu}^m$ can be obtained by re-ordering the scalar operators in each coordinate and by switching $i$ by $-i$ in the expression of the operator $\D_{\lambda, \nu}^{-m}$.

Concretely, by the expression of the map $\Phi_\text{Sol}^+$, $\D_{\lambda, \nu}^{m}$ can be obtained by switching the first and third coordinates, the second one by minus itself, and $i$ by $-i$ (or equivalently $\zeta_2$ by $-\zeta_2$) in the operator $\D_{\lambda, \nu}^{-m}$. This proves Theorem \ref{mainthm2} for $m \leq -1$.
\end{rem}

\section{The case $m=\pm1$ and comparison with the Kobayashi--Kubo--Pevzner operators $\widetilde{\C}^{1,1}_{\lambda, \nu}$}\label{section-KoKuPe}
In \cite{kkp}, all the $O(n,1)$-equivariant differential operators for differential forms on spheres $\widetilde{\C}_{\lambda, \nu}^{i,j}$ are constructed and classified (see \cite[Thm. 2.8]{kkp}):
\begin{equation}\label{diagram-operator-KKP}
\begin{tikzcd}[row sep=0.5cm,column sep=1cm]
\Ind_{H}^{O(n+1, 1)}\left({\textstyle{\bigwedge}}^i(\C^n) \otimes (-1)^\alpha \otimes \C_\lambda\right)\arrow[r, "\displaystyle{\widetilde{\C}_{\lambda,\nu}^{i,j}}"]\arrow[d, equal] & \Ind_{H^\prime}^{O(n,1)}\left({\textstyle{\bigwedge}}^j(\C^{n-1}) \otimes (-1)^\beta \otimes \C_\nu\right)\arrow[d, equal]\\
I(i,\lambda)_\alpha  & J(j,\nu)_\beta
\end{tikzcd}
\end{equation}
Here, $H$ and $H^\prime$ are parabolic subgroups of $O(n+1,1)$ and $O(n,1)$ respectively, with Langlands decompositions as follows:
\begin{equation*}
\begin{aligned}
H &\simeq O(n) \times O(1) \times A \times N_+, &N_+ \simeq \R^n,\\
H^\prime &\simeq O(n-1) \times O(1) \times A \times N_+^\prime, & N_+^\prime \simeq \R^{n-1}.
\end{aligned}
\end{equation*}
For $(i,j,n) = (1,1,3)$, one may expect that these operators have a close relation with the differential operators $\D_{\lambda,\nu}^{\pm 1}$ of Theorem \ref{mainthm2}. In fact, as we pointed out in Remark \ref{rem-mainthms}, since
\begin{equation*}
\begin{aligned}
I(1,\lambda)_\alpha\lvert_{SO_0(4,1)} & \simeq I(1,\lambda) \simeq C^\infty(S^3, \V^3_\lambda),\\
J(1, \nu)_\beta\lvert_{SO_0(3,1)} & \simeq J(1,\nu) \simeq C^\infty(S^2, \L_{-1,\nu}) \oplus C^\infty(S^2, \L_{1,\nu}),
\end{aligned}
\end{equation*}
for $\delta := \nu - \lambda \mod 2\Z$, by \cite[Thm. 2.10]{kkp} we have
\begin{equation*}
\begin{aligned}
\Diff_{SO_0(3,1)}&\left(C^\infty(S^3, \V^3_\lambda), C^\infty(S^2, \L_{1, \nu}) \oplus C^\infty(S^2, \L_{-1, \nu})\right) \\
& \simeq \Diff_{SO_0(3,1)}\left(I(1,\lambda), J(1, \nu)\right) \\
& \simeq \Diff_{O(3,1)}\left(I(1,\lambda)_0, J(1, \nu)_\delta\right) \oplus
\Diff_{O(3,1)}\left(I(1,\lambda)_0, J(1, \nu)_{\delta}\right)\\ 
&\simeq \C\widetilde{\C}_{\lambda, \nu}^{1,1} \oplus \C\widetilde{\C}_{\lambda, \nu}^{1,1}.
\end{aligned}
\end{equation*}
In this section we prove that the operator $\widetilde{\C}_{\lambda, \nu}^{1,1}$ can be expressed as a linear combination of the operators $\D_{\lambda, \nu}^{\pm 1}$. This is done in Proposition \ref{prop-KoKuPe-comparison}.

We recall from \cite{kkp} that, for $(i, j, n) = (1,1, 3)$, the operator $\widetilde{\C}_{\lambda, \nu}^{1,1}$ in (\ref{diagram-operator-KKP}) exists if and only if $a := \nu - \lambda \in \N$ and $\beta - \alpha \equiv a\mod 2$, and it is given as follows (\cite[(2.24), (2.30)]{kkp}):
\begin{equation}\label{def-KKP-operator}
\begin{aligned}
\C_{\lambda, \nu}^{1,1} & = \widetilde{\C}_{\lambda+1, \nu-1}d_{\R^3}d^*_{\R^3} - \gamma(\lambda - \frac{3}{2}, a)\widetilde{\C}_{\lambda, \nu-1}d_{\R^3}\iota_\frac{\partial}{\partial x_3} + \frac{\lambda+a -1}{2}\widetilde{\C}_{\lambda, \nu},\\
\widetilde{\C}_{\lambda, \nu}^{1,1} & = \begin{cases}
\Rest_{x_3 = 0}, \text{ if } a = 0,\\
\C_{\lambda, \nu}^{1,1}, \text{ if } a \geq 1.
\end{cases}
\end{aligned}
\end{equation}
As in Theorem \ref{mainthm2}, this operator is written in coordinates $(x_1, x_2, x_3) \in \R^3 \hookrightarrow S^3$,  where $\widetilde{\C}_{\lambda, \nu}$ is the scalar-valued differential operator defined in (\ref{def-operator-Ctilda}), $d_{\R^3}$ is the differential (exterior derivative) on $\R^3$, $d_{\R^3}^*$ is its adjoint operator (the codifferential), and $\iota_{\frac{\partial}{\partial x_3}}$ is the interior multiplication by the vector field $\frac{\partial}{\partial x_3}$.\\

We identify $C^\infty(S^3, \V_\lambda^3) \simeq I(1,\lambda)$ by using the $SO(3)$-isomorphism $A: V^3 \arrowsimeq \C^3$ given by
\begin{equation*}
A = \begin{pmatrix}
1 & 0 & -1\\
-i & 0 & -i\\
0 & -1 & 0
\end{pmatrix}.
\end{equation*}
In addition, we identify $C^\infty(S^2, \L_{-1, \nu}) \oplus C^\infty(S^2, \L_{1, \nu}) \simeq J(1, \nu)$ by using the SO(2)-isomorphism $\iota : \C_{-1} \oplus \C_{1} \arrowsimeq \C^2$ given by
\begin{equation*}
\iota: \C_{-1} \oplus \C_1 \rightarrow \C^2,\enspace
(z_1, z_2) \mapsto z_1\begin{pmatrix}
1\\
i
\end{pmatrix}
+ z_2\begin{pmatrix}
1\\
-i
\end{pmatrix}.
\end{equation*}
We have the following diagram:
\begin{equation*}
\begin{tikzcd}
& C^\infty(S^2, \L_{-1,\nu}) \arrow[drr, hookrightarrow] && \\
 C^\infty(S^3, \V_\lambda^3) \arrow[ur, "{\displaystyle{\D_{\lambda, \nu}^{-1}}}"] \arrow[dr, "{\displaystyle{\D_{\lambda, \nu}^{1}}}"'] \arrow[r, "{\displaystyle{A}}", "\thicksim"'] &[-1cm] I(1,\lambda) \arrow[r, "{\displaystyle{\widetilde{\C}_{\lambda, \nu}^{1,1}}}"] &  J(1,\nu)  & \arrow[l, "\thicksim", "{\displaystyle{\iota}}"'] C^\infty(S^2, \L_{-1, \nu}) \oplus C^\infty(S^2, \L_{1,\nu}) \\
& C^\infty(S^2, \L_{1,\nu}) \arrow[rru, hookrightarrow]
\end{tikzcd}
\end{equation*}

\begin{prop}\label{prop-KoKuPe-comparison} For any $\lambda, \nu \in \C$ such that $a:= \nu - \lambda \in \N$ we have
\begin{equation}\label{equation-KoKuPe-comparison}
\iota \circ\left(\D_{\lambda, \nu}^{1} - \D_{\lambda, \nu}^{-1}\right) \circ A^{-1} = 
K \widetilde{\C}_{\lambda, \nu}^{1,1},
\end{equation}
where $K \equiv K(\lambda,\nu) := \begin{cases}
1, \text{ if } a = 0,\\
2, \text{ if } a \geq 1.
\end{cases}$
\end{prop}
Before proving this result, we introduce some notation and a useful lemma. For $\mu \in \C$ and $\ell \in \N$ let
\begin{equation*}
\mathscr{D}_\ell^\mu := \left(I_\ell \Geg_\ell^\mu\right)\left(-\Delta_{\R^2}, \frac{\partial}{\partial x_3}\right).
\end{equation*}
This is a homogeneous differential operator on $\R^3$ of order $\ell$ and by (\ref{def-operator-Ctilda}) it clearly satisfies
\begin{equation*}
\widetilde{\C}_{\lambda, \nu} = \Rest_{x_3 = 0} \circ \mathscr{D}_{\nu-\lambda}^{\lambda-1}.
\end{equation*}
The following lemma follows from the Kobayashi--Kubo--Pevzner three-term relations among renormalized Gegenbauer polynomials (see Lemma \ref{Lemma-KKP-identities}).

\begin{lemma}[{\cite[Prop. 9.2]{kkp}}] For any $\mu \in \C$ and any $\ell \in \N$, the following identities hold:
\begin{align}
\label{KKP-D-1}
&(\mu + \ell)\mathscr{D}_\ell^\mu - \mathscr{D}_{\ell-2}^{\mu+1}\Delta_{\R^2} = \left(\mu + \left[\frac{\ell+1}{2}\right]\right)\mathscr{D}_\ell^{\mu+1},\\
\label{KKP-D-2}
&\mathscr{D}_{\ell-2}^{\mu+1}\Delta_{\R^2} + \gamma(\mu, \ell)\mathscr{D}_{\ell-1}^{\mu+1}\frac{\partial}{\partial x_3}  = \frac{\ell}{2}\mathscr{D}_\ell^\mu,\\
\label{KKP-D-3}
&\mathscr{D}_{\ell-2}^{\mu+1}\frac{\partial}{\partial x_3} - \gamma(\mu, \ell)\mathscr{D}_{\ell-1}^{\mu+1} + \gamma(\mu - \frac{1}{2}, \ell)\mathscr{D}_{\ell-1}^\mu = 0.
\end{align}
\end{lemma}

\begin{proof}[Proof of Proposition \ref{prop-KoKuPe-comparison}]
We prove the lemma for $a \geq 1$. The case $a = 0$ is immediate since the operators $\widetilde{\C}^{1,1}_{\lambda, \nu}$ and $\D_{\lambda, \nu}^{\pm 1}$ are trivial.

From the definition of $A$ and $\iota$, we have
\begin{equation*}
\begin{aligned}
A^{-1}: C^\infty(\R^3, {\textstyle{\bigwedge}}^1(\C^3)) & \arrowsimeq C^\infty(\R^3, V^3)\\
f_1dx_1 + f_2dx_2 + f_3dx_3 & \longmapsto \frac{1}{2}(f_1 + if_2)u_1 - f_3u_2 + \frac{1}{2}(-f_1 + if_2)u_3,
\end{aligned}
\end{equation*}
and
\begin{equation*}
\begin{aligned}
\iota: C^\infty(\R^2, \C_{-1}) \oplus C^\infty(\R^2, \C_1) & \arrowsimeq C^\infty(\R^2, {\textstyle{\bigwedge}}^1(\C^2))\\
\psi_{-1} + \psi_1 & \longmapsto \psi_{-1}(d x_1 + id x_2) + \psi_{1}(d x_1 - idx_2).
\end{aligned}
\end{equation*}
Now, for $\omega = f_1dx_1 + f_2dx_2 + f_3dx_3 \in C^\infty(\R^3, \bigwedge^1(\C^3))$, from (\ref{def-KKP-operator}), (\ref{Operator-general1}) and (\ref{Operator-general-1}) we have:
\begin{multline*}
\bullet\enspace \C^{1,1}_{\lambda, \nu}(\omega) =
\Bigg[\widetilde{\C}_{\lambda+1, \nu-1}\left(-\frac{\partial^2}{\partial x_1^2} f_1 -\frac{\partial^2}{\partial x_1 \partial x_2} f_2 -\frac{\partial^2}{\partial x_1 \partial x_3} f_3\right) \\
- \gamma(\lambda - \frac{3}{2}, a)\widetilde{\C}_{\lambda, \nu-1}\frac{\partial}{\partial x_1}f_3
 + \frac{\lambda + a -1}{2}\widetilde{\C}_{\lambda, \nu}f_1 \Bigg] d x_1 \\
 + \Bigg[\widetilde{\C}_{\lambda+1, \nu-1}\left(-\frac{\partial^2}{\partial x_1 \partial x_2} f_1 -\frac{\partial^2}{\partial x_2^2} f_2 -\frac{\partial^2}{\partial x_2 \partial x_3} f_3\right) \\- \gamma(\lambda - \frac{3}{2}, a)\widetilde{\C}_{\lambda, \nu-1}\frac{\partial}{\partial x_2}f_3 + \frac{\lambda + a -1}{2}\widetilde{\C}_{\lambda, \nu}f_2\Bigg] dx_2.
\end{multline*}
\begin{multline*}
\bullet\enspace \iota \circ \D_{\lambda,\nu}^1 \circ A^{-1}(\omega) = \Bigg[\frac{1}{2}\left(\lambda + \left[\frac{a-1}{2}\right]\right) \widetilde{\C}_{\lambda+1, \nu+1} (f_1 + if_2)\\
-\gamma(\lambda-1, a)\widetilde{\C}_{\lambda+1, \nu}\left(\frac{\partial}{\partial x_1} + i\frac{\partial}{\partial x_2} \right) f_3
+ \frac{1}{2}\widetilde{\C}_{\lambda+1, \nu-1} \left(\frac{\partial}{\partial x_1} + i\frac{\partial}{\partial x_2} \right)^2 (-f_1+if_2)\Bigg] dx_1\\
+\Bigg[\frac{-i}{2}\left(\lambda + \left[\frac{a-1}{2}\right]\right) \widetilde{\C}_{\lambda+1, \nu+1} (f_1 + if_2)
+i\gamma(\lambda-1, a)\widetilde{\C}_{\lambda+1, \nu}\left(\frac{\partial}{\partial x_1} + i\frac{\partial}{\partial x_2} \right) f_3\\
- \frac{i}{2}\widetilde{\C}_{\lambda+1, \nu-1} \left(\frac{\partial}{\partial x_1} + i\frac{\partial}{\partial x_2} \right)^2 (-f_1+if_2)\Bigg] dx_2.
\end{multline*}
\begin{multline*}
\bullet\enspace \iota \circ \D_{\lambda,\nu}^{-1} \circ A^{-1}(\omega) =  \Bigg[\frac{1}{2}\widetilde{\C}_{\lambda+1, \nu-1}\left(\frac{\partial}{\partial x_1} - i\frac{\partial}{\partial x_2} \right)^2 (f_1 + if_2) \\
+\gamma(\lambda-1, a)\widetilde{\C}_{\lambda+1, \nu} \left(\frac{\partial}{\partial x_1} - i\frac{\partial}{\partial x_2} \right)f_3
+ \frac{1}{2}\left(\lambda + \left[\frac{a-1}{2}\right]\right) \widetilde{\C}_{\lambda+1, \nu+1} (-f_1 + if_2)\Bigg] dx_1\\
+ \Bigg[\frac{i}{2}\widetilde{\C}_{\lambda+1, \nu-1}\left(\frac{\partial}{\partial x_1} - i\frac{\partial}{\partial x_2} \right)^2 (f_1 + if_2)  +i\gamma(\lambda-1, a)\widetilde{\C}_{\lambda+1, \nu} \left(\frac{\partial}{\partial x_1} - i\frac{\partial}{\partial x_2} \right)f_3 \\
+ \frac{i}{2}\left(\lambda + \left[\frac{a-1}{2}\right]\right) \widetilde{\C}_{\lambda+1, \nu+1} (-f_1 + if_2)\Bigg] dx_2.
\end{multline*}

We prove the identity (\ref{equation-KoKuPe-comparison}) for the component $d x_1$. By an analogous argument the component $dx_2$ can be easily checked. By the expressions above, the $f_1$ part of the right-hand side of (\ref{equation-KoKuPe-comparison}) amounts to
\begin{equation*}
-2\widetilde{\C}_{\lambda+1, \nu-1}\frac{\partial^2}{\partial x_1^2}f_1 + (\nu-1)\widetilde{\C}_{\lambda, \nu}f_1,
\end{equation*}
while the $f_1$ part of the left-hand side is
\begin{equation*}
\left(\lambda + \left[\frac{a-1}{2}\right]\right) \widetilde{\C}_{\lambda+1, \nu+1}f_1 - \widetilde{\C}_{\lambda+1, \nu-1}\left(\frac{\partial^2}{\partial x_1^2} - \frac{\partial^2}{\partial x_2^2} \right)f_1.
\end{equation*}
Thus, the equality for the $f_1$ part of the $dx_1$ component holds if and only if
\begin{equation*}
-\Delta_{\R^2}\widetilde{\C}_{\lambda+1, \nu-1} + (\nu-1)\widetilde{\C}_{\lambda, \nu} = \left(\lambda + \left[\frac{a-1}{2}\right]\right) \widetilde{\C}_{\lambda+1, \nu+1},
\end{equation*}
which is indeed satisfied by (\ref{KKP-D-1}).

For the $f_2$ part of the $dx_1$ component, the right-hand side of (\ref{equation-KoKuPe-comparison}) is
\begin{equation*}
-2\widetilde{\C}_{\lambda+1, \nu-1}\frac{\partial^2}{\partial x_1 \partial x_2} f_2
\end{equation*}
which clearly coincides with the left-hand side, since the other terms cancel out.

Finally, the right and left-hand sides of the $f_3$ part of the $dx_1$ component of (\ref{equation-KoKuPe-comparison}) are respectively
\begin{equation*}
-2\widetilde{\C}_{\lambda+1, \nu-1}\frac{\partial^2}{\partial x_1 \partial x_3} f_3 - 2\gamma(\lambda - \frac{3}{2}, \nu-\lambda)\widetilde{\C}_{\lambda, \nu-1}\frac{\partial}{\partial x_1}f_3,
\end{equation*}
and
\begin{equation*}
-\gamma(\lambda-1, a)\widetilde{\C}_{\lambda+1,\nu}\frac{\partial}{\partial x_1},
\end{equation*}
which coincide thanks to (\ref{KKP-D-3}). 
\end{proof}

\section{Appendix: Gegenbauer polynomials}
In this section we define and give some properties of the renormalized Gegenbauer polynomials we use throughout the paper.

\subsection{Renormalized Gegenbauer polynomials}
For any $\mu \in \C$ and any $\ell \in \N$, the Gegenbauer polynomial (or ultraspherical polynomial), is defined as follows (\cite[6.4]{aar}, \cite[3.15(2)]{emot}):
\begin{equation*}
\begin{aligned}
C_\ell^\mu(z) & := \sum_{k = 0}^{[\frac{\ell}{2}]}(-1)^k \frac{\Gamma(\ell - k + \mu)}{\Gamma(\mu)k!(\ell -  2k)!}(2z)^{\ell - 2k}\\
& = \frac{\Gamma(\ell + 2\mu)}{\Gamma(2\mu)\Gamma(\ell + 1)}{}_2F_1\left(-\ell, \ell +2\mu; \mu+ \frac{1}{2}; \frac{1-z}{2} \right).
\end{aligned}
\end{equation*}
This polynomial satisfies the Gegenbauer differential equation
\begin{equation*}
G_\ell^\mu f(z) = 0,
\end{equation*}
where $G_\ell^\mu$ is the following differential operator known as the Gegenbauer differential operator
\begin{equation*}
G_\ell^\mu := (1-z^2)\frac{d^2}{dz^2} - (2\mu +1)z\frac{d}{dz} + \ell(\ell +2\mu).
\end{equation*}
Observe that $C_\ell^\mu \equiv 0$ if $\mu = 0, -1, -2, \cdots, - [\frac{\ell-1}{2}]$. As in \cite[(14.23)]{kkp}, we renormalize the Gegenbauer operator as follows:
\begin{equation}\label{Gegenbauer-polynomial(renormalized)}
\widetilde{C}^\mu_\ell(z) := \frac{\Gamma(\mu)}{\Gamma\left(\mu + [\frac{\ell + 1}{2}]\right)}C^\mu_\ell(z) = \frac{1}{\Gamma\left(\mu+ [\frac{\ell + 1}{2}]\right)}\sum_{k=0}^{[\frac{\ell}{2}]}(-1)^k \frac{\Gamma(\mu + \ell - k)}{k! (\ell - 2k)!}(2z)^{\ell - 2k}.
\end{equation}
By doing this, we obtain a non-zero polynomial for any value of the parameters. That is, for any $\mu\in \C$ and any $\ell \in \N$, the renormalized Gegenbauer polynomial $\Geg^\mu_\ell(z)$ is a non-zero polynomial with degree at most $\ell$. We write below the first seven Gegenbauer polynomials.
\begin{itemize}
\item $\widetilde{C}^\mu_0(z) = 1$.
\item $\widetilde{C}^\mu_1(z) = 2z$.
\item $\widetilde{C}^\mu_2(z) = 2(\mu + 1) z^2 -1$.
\item $\widetilde{C}^\mu_3(z) = \frac{4}{3}(\mu + 2) z^3 -2z$.
\item $\widetilde{C}^\mu_4(z) = \frac{2}{3}(\mu +2)(\mu+3)z^4 -
2(\mu + 2)z^2 + \frac{1}{2}$.
\item $\widetilde{C}^\mu_5(z) = \frac{4}{15}(\mu+3)(\mu+4)z^5 - \frac{4}{3}(\mu+3)z^3 + z$.
\item $\widetilde{C}^\mu_6(z) = \frac{4}{45}(\mu+3)(\mu+4)(\mu+5)z^6 - \frac{2}{3}(\mu+3)(\mu+4)z^4 + (\mu+3)z^2 - \frac{1}{6}$.
\end{itemize}

Observe that from the definition, the degree of each term of $\Geg_\ell^\mu(z)$ has the same parity as $\ell$, and the number of terms is at most $\left[\frac{\ell}{2}\right]+1$. Depending on the value of $\mu$, some terms may vanish, but the constant term is always non-zero.

The next result shows that the solutions of the Gegenbauer differential equation on $\Pol_\ell[z]_\text{{\normalfont{even}}}$ are precisely the renormalized Gegenbauer polynomials.

\begin{thm}[{\cite[Thm 11.4]{kob-pev2}}] \label{Gegenbauer-solutions} For any $\mu \in \C$ and any $\ell \in \N$ the following holds.
\begin{equation*}
\{f(z) \in \Pol_\ell[z]_\text{{\normalfont{even}}} : G_\ell^\mu f(z) = 0\} = \C \widetilde{C}_\ell^\mu(z).
\end{equation*}
\end{thm}

For any $\mu \in \C$ and any $\ell \in \N$, the \textit{imaginary} Gegenbauer differential operator $S_\ell^\mu$ is defined as follows (cf. \cite[(4.7)]{kkp}):
\begin{equation}\label{Gegen-imaginary}
S_\ell^\mu = - \left((1 + t^2) \frac{d^2}{dt^2} + (1 + 2\mu)t\frac{d}{dt} - \ell(\ell + 2\mu)\right).
\end{equation}
This operator and $G_\ell^\mu$ are in the following natural relation.

\begin{lemma}[{\cite[Lem. 14.2]{kkp}}]\label{relationS-G} Let $f(z)$ be a polynomial in the variable $z$ and define $g(t) = f(z)$, where $z = it$. Then, the following identity holds:
\begin{equation*}
\left(S_\ell^\mu g\right)(t) = \left(G_\ell^\mu f\right)(z).
\end{equation*}
\end{lemma}

\subsection{The Kobayashi--Kubo--Pevzner three-term relations among renormalized Gegenbauer polynomials}
In this section we collect useful properties of Gegenbauer polynomials and Gegenbauer differential operators. We start by showing some algebraic identities and follow by recalling the Kobayashi--Kubo--Pevzner three-term relations among renormalized Gegenbauer polynomials. At the end, we show a \lq\lq degree decay\rq\rq{ }property of Gegenbauer polynomials when special parameters are considered. 

For $\mu \in \C$ and $\ell \in \N$, we set
\begin{equation}\label{gamma-def}
\gamma(\mu, \ell) := \displaystyle{\frac{\Gamma(\mu + \left[\frac{\ell+2}{2}\right])}{\Gamma(\mu + \left[\frac{\ell + 1}{2}\right])} = \begin{cases}
1,&\text{if } \ell \text{ is odd},\\
\mu + \frac{\ell}{2},&\text{if } \ell \text{ is even}.
\end{cases}}
\end{equation}
\begin{lemma}[{\cite[Lem. 14.5]{kkp}}] For any $\mu \in \C$ and any $\ell \in \N$ the following identities hold:
\begin{align}
\label{Slmu-identity-1}
&S_\ell^{\mu + 1} - S_\ell^\mu =  2(\ell - \vartheta_t),\\
\label{Slmu-identity-2}
&tS_{\ell-1}^{\mu + 1} - S_\ell^\mu t =  2\frac{d}{dt}.
\end{align}
\end{lemma}

\begin{lemma}[{\cite[Lem. 14.4]{kkp}}] For any $\mu \in \C$ and any $\ell \in \N$ the following identities hold:
\begin{align}
\label{derivative-Gegenbauer-1}
&\frac{d}{dz}\Geg_\ell^\mu(it) = 2i\gamma(\mu, \ell)\Geg_{\ell -1}^{\mu + 1}(it),\\[4pt]
\label{derivative-Gegenbauer-2}
&\left(t\frac{d}{dt}-\ell\right)\widetilde{C}_\ell^\mu(it) = 2\Geg_{\ell -2}^{\mu + 1}(it).
\end{align}
\end{lemma}

\begin{lemma}[{\cite[Sec. 14.3]{kkp}}]\label{Lemma-KKP-identities}  For any $\mu \in \C$ and any $\ell \in \N$ the following three-term relations of renormalized Gegenbauer polynomials hold:
\begin{align}
\label{KKP-1}
&(\mu + \ell)\Geg_\ell^\mu(it) + \Geg_{\ell-2}^{\mu+1}(it) = \left(\mu + \left[\frac{\ell+1}{2}\right]\right)\Geg_\ell^{\mu+1}(it),\\
\label{KKP-2}
&\Geg_{\ell-2}^{\mu+1}(it) = \gamma(\mu, \ell)it\Geg_{\ell-1}^{\mu + 1}(it) - \frac{\ell}{2}\Geg_{\ell}^\mu(it),\\
\label{KKP-3}
&it\Geg_{\ell-1}^{\mu+1}(it) - \gamma(\mu, \ell + 1)\Geg_\ell^{\mu+1}(it) + \gamma(\mu - \frac{1}{2}, \ell+1)\Geg_\ell^\mu(it) = 0.
\end{align}
\end{lemma}

\begin{lemma}[{\cite[Lem. 14.12]{koss}}]\label{lemma-KOSS} For $\ell,k \in \N$ such that, $\ell \leq 2k$ the following holds:
\begin{equation*}
C_\ell^{-k}(z) = C_{2k-\ell}^{-k}(z),
\end{equation*}
or equivalently
\begin{equation}\label{identity-KOSS}
\frac{\Gamma\left(-k + \left[\frac{\ell+1}{2}\right]\right)}{\Gamma\left(-k + \left[\frac{2k-\ell+1}{2}\right]\right)}\Geg_{\ell}^{-k}(z) = \Geg_{2k-\ell}^{-k}(z).
\end{equation}
\end{lemma}

\section*{Acknowledgements}
I would like to thank Professor Toshiyuki Kobayashi for his invaluable guidance and advice, without which this paper would not have been possible. I also would like to give a special thanks to Professor Toshihisa Kubo for his support and help, providing comments and suggestions that improved substantially the legibility of the paper. 

\renewcommand{\refname}{References}

\end{document}